\documentclass[preprint,11pt]{elsarticle1}

\usepackage{amsmath,amsfonts,amssymb,amsthm}
\usepackage{bbm}
\usepackage{dsfont}
\usepackage{graphicx,color}
\usepackage[english]{babel}
\usepackage{pb-diagram}
\usepackage[all]{xy}
\usepackage{cancel}
\usepackage{textcomp}
\usepackage{mathrsfs}
\usepackage{float}
\usepackage[textures]{epsfig}
\usepackage{subcaption}
\usepackage{amsfonts}%
\usepackage{amssymb}%
\usepackage{graphicx}

\usepackage{selinput}
\usepackage{enumitem}

\usepackage[dvipsnames]{xcolor}

\usepackage{geometry}

\usepackage{tikz}
\usepackage{tikz-3dplot}
\usetikzlibrary{patterns}
\usetikzlibrary{shapes.geometric}

\newtheorem{lemma}{Lemma}[section]

\theoremstyle{definition}
\newtheorem{remark}{Remark}[section]

\numberwithin{equation}{section}

\begin{document}

\begin{frontmatter}

\title
{Folding Domain Functions (FDF): a Random Variable Transformation
technique for the non-invertible case, with applications to RDEs}

\author[A1]{Fabrizio Masullo}
\ead{fabrizio.masullo89@gmail.com}
\tnotetext[note]
{F.M. acknowledges the Universitat Rovira i Virgili
for the Mart\'{\i} i Franqu\`{e}s scholarship 2016 MF-PIPF-27. J.E.}

\author[A2]{Fabio Zanolin}
\ead{fabio.zanolin@uniud.it}

\author[A1]{Josep Bonet Avalos}
\ead{josep.bonet@urv.cat}

\address[A1]{Departament d'Enginyeria Qu\'{\i}mica,
Escola Tecnica Superior d'Enginyeria Qu\'{\i}mica (ETSEQ),
Universitat Rovira i Virgili, Avda. dels Pa\"{\i}sos Catalans 26, 43007,
Tarragona, Spain}

\address[A2]{Department of Mathematics, Computer Science and Physics (DMIF),
University of Udine,
via delle Scienze 206, 33100 Udine, Italy}

\begin{abstract}{\small{
The Random Variable Transformation (RVT) method is a fundamental
tool for determining the probability distribution function
associated with a Random Variable (RV) $Y=g(X),$ where $X$ is a RV
and $g$ is a suitable transformation. In the usual applications of
this method, one has to evaluate the derivative of $h \equiv
g^{-1}.$ This can be a straightforward procedure when $g$ is
invertible, while difficulties may arise when $g$ is
non-invertible. The RVT method has received a great deal of
attention in the recent years, because of its crucial relevance in
many applications. In the present work we introduce a new approach
which allows to determine the probability density function $\mu_Y$
of the RV $Y=g(X)$, when $g$ is non-invertible due to its
non-bijective nature. The main interest of our approach is that it
can be easily implemented, from the numerical point of view, but
mostly because of its low computational cost, which makes it very
competitive. As a proof of concept, we apply our method to some
numerical examples related to random differential equations, as
well as discrete mappings, all of them of interest in the domain
of applied Physics. }}
\end{abstract}

\begin{keyword}{\small{
Random variable transformation \sep Distribution function \sep
Density function \sep Ordinary differential equations \sep Random
differential equations \sep Non-invertible random transformation \sep
Applications \sep
Numerical simulations\\
{\emph{MSC classification:}} {34F05} \sep {60E05} \sep {65C30}.}}
\end{keyword}

\date{}

\end{frontmatter}

\begin{center}
{To Roberto Blanchini, in memoriam}
\end{center}

\section{Introduction}\label{section-1}
The theory of random dynamical systems has been greatly developed
in the past fifty years, both from the theoretical and applied
point of view \cite{Ar-1998,So-1973}. Particularly relevant are the applications to
biology, economics and social sciences. More recent
applications involve models related to problems in climatology,
combustion theory and conduction of electrical impulses in nerve axons
\cite{BCKV-2010},
just to mention a few examples and to show the great deal of interest in this area
of research.

The more general models involve the study of a random differential
equations with random initial data, along with the effect of
time-dependent noise sources. However, there are many significant
examples where such time-dependent noise sources are not present.
In these cases, one deals with a deterministic (non-random) law
describing the dynamic evolution of the model, where the
probabilistic effects only concern the randomness of the initial
conditions. In this situation, we can model the dynamic phenomenon
as a map which transforms random initial conditions, expressed by
a random variable (RV), to an output, whose probability
distribution is to be determined \cite[Ch.5]{Pa-1984}. Problems of
this kind can be theoretically settled within the framework of the
theory of functions of a random variable. More in detail, let $X$
be a random variable, with density function $\mu_X$, defined on a
probability space $\Omega=(\Omega,{\mathcal A},p)$, \footnote{As
usual, ${\mathcal A}$ denotes a $\sigma$-algebra of measurable
sets (in the case of our applications ${\mathcal A}$ will denote
the collection of Lebesgue measurable sets in ${\mathbb R}$) and
$p$ is a probability defined on ${\mathcal A}.$} with values into
${\mathbb R}$, and let $g$ be a real function defined on ${\mathbb
R}$, which may require some properties that we discuss later on. A
classical problem considered in the literature consists in
determining the probability distribution of $Y=g(X)$ from the
information about $X$ and $g$. This problem is solved by a
standard procedure, by assuming $g$ to be invertible (see
\cite[p.81]{JP-2000}). In particular, in the one-dimensional case,
$g:{\mathbb R}\to {\mathbb R},$ we have a precise formula for the
density function $\mu_Y$ of $Y$ provided that $g$ is continuously
differentiable with a non-vanishing first derivative
\cite[Corollary 11.2]{JP-2000}. A more complex situation arises if
$g$ is only piecewise strictly monotone, but continuously
differentiable, with a finite set of isolated points where the
derivative vanishes, which is the problem that we address in the
present article. In this case, if $I_1,\dots,I_n$ is a finite
sequence of intervals decomposing the domain of $g$ and such that
$g$ is strictly monotone and continuously differentiable in the
interior of each interval, we have that
\begin{equation}\label{eq-1.1}
\mu_Y(y)= \sum_{i=1}^{n} \mu_X(h_i(y))|h'_i(y)| \mathbbm{1}_{\land}(y),
\end{equation}
where $h_i$ is the inverse function of $g$ but restricted to the interval $I_i,$
$\land=\{y: y=g(x): x\in {\mathbb R}\}$ and
$\mathbbm{1}_{\land}(y)$ is the characteristic function which takes
a value $1$ on $\land$ and $0$ otherwise
(see \cite[Corollary 11.3]{JP-2000}).
Although the formula \eqref{eq-1.1} provides a precise definition of the
density function associated with the random variable $Y,$ its implementation in concrete examples can be very difficult. Recent articles
applying this method, named \textit{Random Variable Transformation technique},
to different models of present interest in mathematical
biology and population dynamics, for instance,
deals with more specific situations when $g$ is monotone on its whole domain
or has a quadratic shape, thus reducing the complexity only to two symmetric
intervals of invertibility (see \cite{CC-2016,CC-2016b,CC-2016c,BCCN-2018,CCDJ-2022}
and the references therein). The aim of this article is precisely to provide the framework to build an efficient algorithm for these cases of interest that cannot be straightforwardly addressed within the standard approach.

A few typical cases, in which a function of a RV may occur in dynamical systems, are the following.

{\bf a)\;\;} As a first example, let us consider the case of a
discrete Dynamical System defined by a difference equation of the
form
\begin{equation}\label{eq-1.2}
x_{n+1}= L(x_{n}),
\end{equation}
with $L:A\to A,$ where $A$ is a suitable domain in ${\mathbb R}^d$ and
$L$ is a sufficiently smooth
function. Typical models in population dynamics may
involve a non-invertible map $L,$ due to
the presence of natural growth restrictions, for instance, of logistic type limitations
\cite{Mu-2002}.
The situation becomes even more complicated when we look for periodic points, a problem which
is equivalent to the study of fixed points for iterates of $L.$
For a non-invertible $L,$ the
shape of the $k$-th iterate $L^k=\underbrace{L\circ \dots \circ L}_\text{$k$ times}$
can be extremely far from a monotone function.
This is apparent even when studying
one-dimensional classical logistic equations (see \cite{LiYo-1975}),
as shown in Figure \ref{fig-1.1} where the logistic map, together with its third iterate,
is considered.

\begin{figure}[htp]
\begin{center}
\includegraphics[scale=0.3]{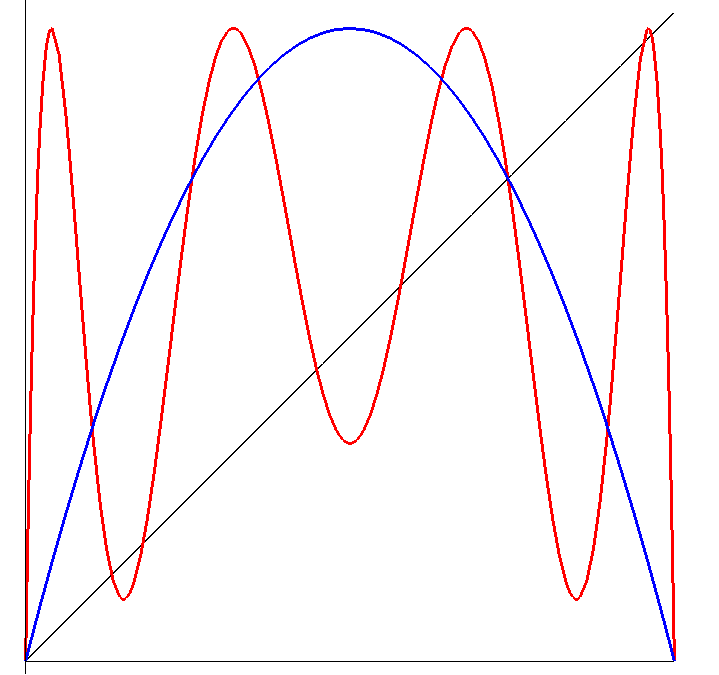}
\end{center}
\caption{\small{The figure represents the graph of
the logistic map $L(x)=rx(1-x)$ and its third iterate $L(L(L(x))$ in the interval
$[0,1].$ For our example, we have considered the case $r=3.9$ where there are points
of (minimal) period three, as seen from the intersections with the line $y=x$, which is known to yield chaotic dynamics, according to \cite{LiYo-1975}.}}
\label{fig-1.1}
\end{figure}

If we assume that the initial point $x_0 := \phi$ is represented by a RV, we are interested in determining the probability distribution
that describes the state of the system after a certain number $k$ of iterates, or, equivalently,
the probabilistic outcome at the $k$-generation.

{\bf b)\;\;}
As a second example, we consider the case of a
continuous dynamical system associated with an ordinary
differential equation. Here, we are interested in
the particular case of a vector system
ordinary differential equations of the form,
\begin{equation}\label{rde-1.1}
\begin{cases}
\frac{d}{dt}\underline{x}(t)=\vec{f}\bigl(\underline{x}(t),t\bigr)\\
\underline{x}(t_0)=(x^0_1,x^0_2, \dots, x^0_i, \dots x^0_d)
\end{cases}
\end{equation}
where $\underline{x}(t) \in {\mathbb R}^d$. Let us assume that the
initial condition $\underline{x}(t_0)$ contains one component,
$x_r$, say, which is a RV, i.e. $x^0_r = \phi$. If, for a moment,
we suppose that $\underline{x}(t_0)$ is a given specific a vector
under natural regularity assumptions \cite{Ha-1980} on $\vec{f},$
we know that there is a unique solution
$\varphi_{\underline{x}(t_0)}(t)$ defined on a maximal interval of
existence. Suppose also that for a fixed interval $[t_0,t_1]$,
$t_1>t_0,$ all the solutions with initial value
$\underline{x}(t_0)$ in a given set $D,$ are defined. In this
case, the map $\Psi: \underline{x}(t_0) \mapsto
\varphi_{\underline{x}(t_0)}(t_1)$ (the so-called Poincar\'{e}
map) is a homeomorphism of $D$ onto $\Psi(D)$, as it is well known
in the theory of Dynamical Systems. Therefore, if
$\underline{x}(t_0)=\phi \in {\mathbb R}$ (i.e. in the
one-dimensional case), the map $\Psi$ must be a strictly monotone
function, and the RVT method can be straightforwardly applied.
However, in all the other situations in which $d>1$, which are
also relevant in Physics, the monotonicity of the maps associated
to $\Psi$ is lost.

A significant example comes from
the \textit{shooting method}. In this case, we study a
second-order scalar differential equation of Newtonian type of the form
\begin{equation}\label{rde-1.2}
y''(t)+f\bigl(y(t),y'(t),t\bigr)=0,
\end{equation}
corresponding to the planar system
\begin{equation}\label{rde-1.2bis}
\begin{cases}
y'(t)= v(t)\\
v'(t) = - f\bigl(y(t),v(t),t\bigr).
\end{cases}
\end{equation}
The initial condition
$$y(t_0)=a, \; y'(t_0)=b,$$
reads now as
$$y(t_0)=a, \; v(t_0)=b.$$
Then, to the pair $(a,b)\in {\mathbb R}^2,$ we associate the point
$$(a_1,b_1):=\varphi_{(a,b)}(t_1)=
\bigl(\varphi^{(1)}_{(a,b)}(t_1),\varphi^{(2)}_{(a,b)}(t_1)\bigr)
\in {\mathbb R}^2,$$ by means of the two-dimensional Poincar\'{e}
map. Suppose now that for the initial pair $(a,b)$ one of the
parameters is determined (for instance the initial position $a$ of
the particle), while the other (for instance the initial velocity
$b$) is a RV. In this case, we set $b=\phi$ and, instead of being
interested in the complete map $\varphi_{(a,b)}(t_1)$, we restrict
ourselves to the map which associates to $\phi$ the final
position, $\varphi^{(1)}_{(a,\phi)}(t_1)$; this map is not
invertible.

In the deterministic case, this procedure consists of fixing the
initial position $\phi$ and look for the final one, given the
initial velocity, which needs to be known at the beginning. Then,
the probability distribution $\mu_Y$ can be reconstructed by {\em
shooting} a bundle of trajectories by varying $\phi$, and
measuring the density of the end points $y(t)$. We are going to
compare the presented algorithm with this simple method, which we
will refer to as the {\em brute-force} algorithm. In this example,
we transform the RV $\phi$ for the initial velocity into a new
random variable for the final position at $t$ but the map $g$
describing the transformation is no longer monotone, as one can
see from the elementary examples, taken from physical and
mechanical models, shown below (see Figure \ref{fig-1.2} and
Figure \ref{fig-1.3} for two possible cases).

\begin{figure}[htp]
\begin{center}
\includegraphics[scale=0.18]{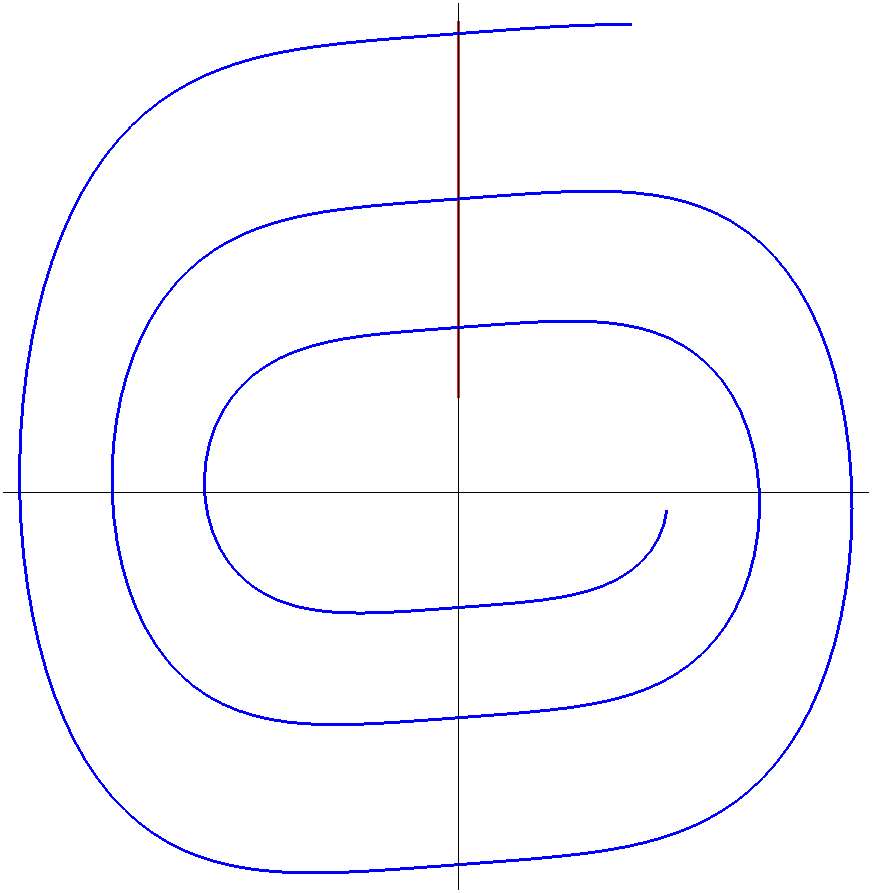}
\quad
\includegraphics[scale=0.18]{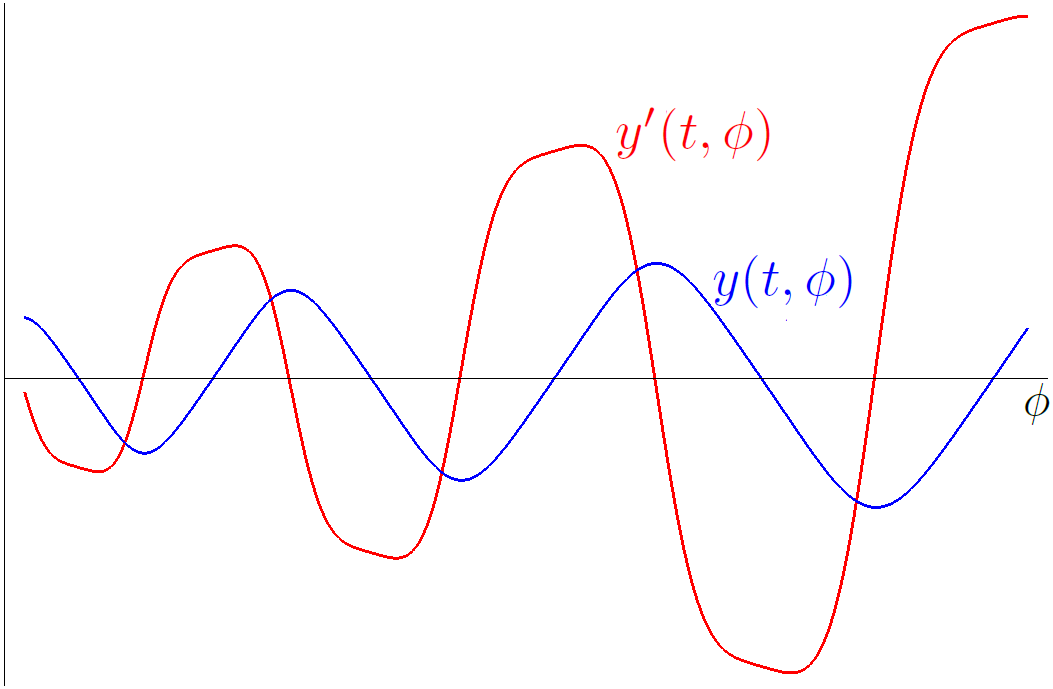}
\end{center}
\caption{\small{The figure illustrates the use of the shooting method
for the Duffing equation, Eq. (\ref{Duffing}).}}
\label{fig-1.2}
\end{figure}

\begin{figure}[htp]
\begin{center}
\includegraphics[scale=0.15]{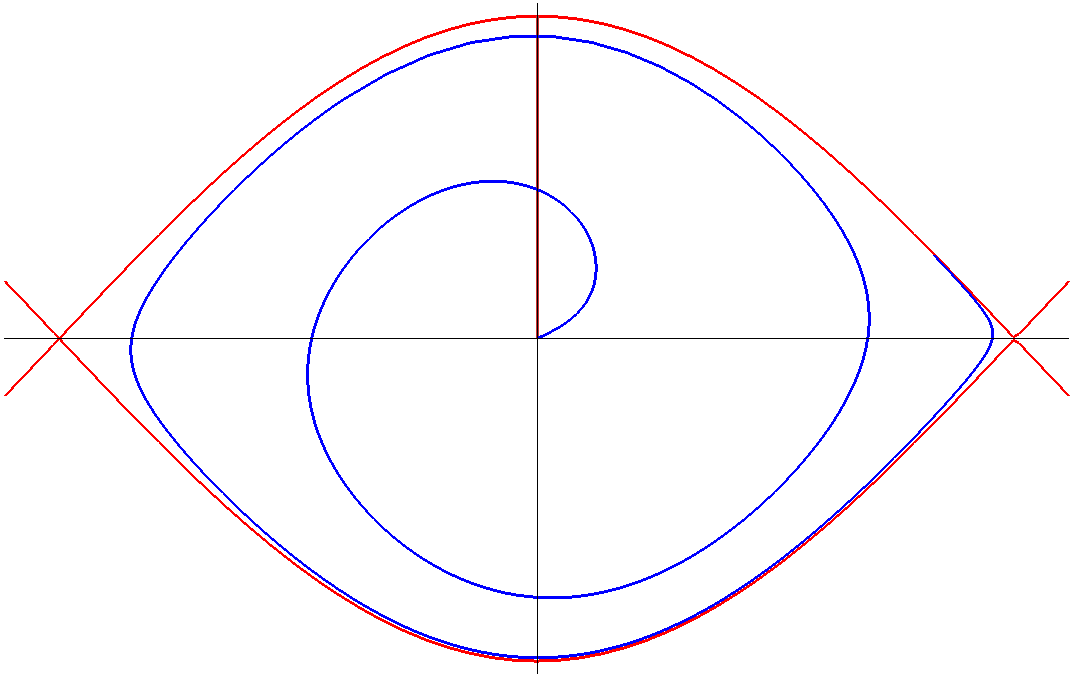}
\quad
\includegraphics[scale=0.15]{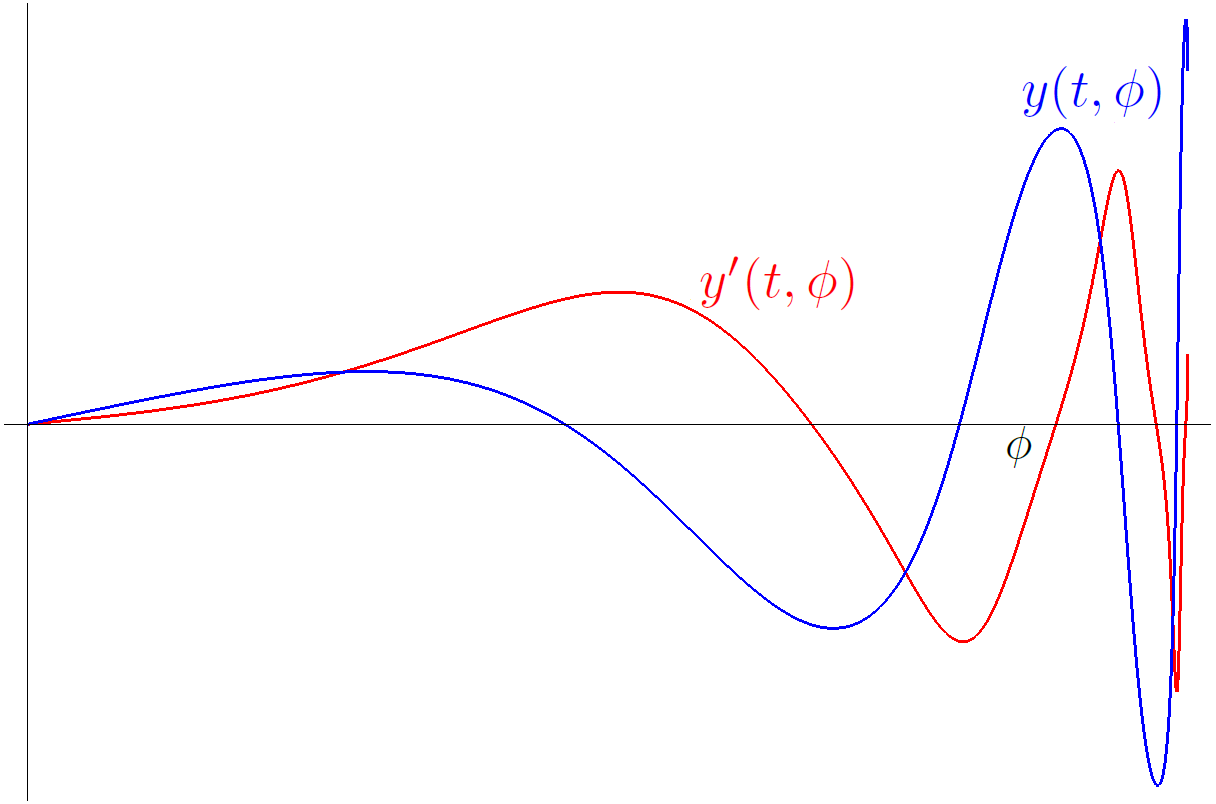}
\end{center}
\caption{\small{The figure illustrates the use of the shooting
method for the pendulum equation, Eq. \eqref{pendulum}.}}
\label{fig-1.3}
\end{figure}

In Figure \ref{fig-1.2}, we consider the Duffing equation
\begin{equation}\label{Duffing}
y''(t) + \rho(y(t))=0,
\end{equation}
with the superlinear term $\rho(y)=4y^3.$ In this case, it is well-known
that the origin is a global center in the phase plane $(y,v)=(y,y')$
and all the orbits around the origin lie on the energy levels
$\frac{1}{2}v^2 + y^4 =\text{constant} >0.$ The left panel shows the image of the
Poincar\'{e} map after time $t=10$ of the vertical segment $y(0)=0,$ $y'(0)=\phi\in [1,5].$
The right panel shows the graph of the map $\varphi$ which associates to $\phi\in [1,5]$
the value of the solutions $y(t,\phi)$, respectively $y'(t,\phi)$, at the time $t=10.$

In Figure \ref{fig-1.3}, we represent the pendulum equation
\begin{equation}\label{pendulum}
y''(t) + \sin(y(t))=0,
\end{equation}
where the origin is a local center in the phase plane $(y,v)=(y,y')$
bounded above and below by two heteroclinic trajectories. These heteroclinic trajectories connect the
unstable equilibrium points $(\pm\pi,0)$, which represent the geometric configuration
of a pendulum going to the inverted vertical position. For an energy constant $0 < c < 2,$
all the orbits around the origin lie within the energy levels
$\frac{1}{2}v^2 + (1-\cos(y)) =c$ and are closed curves relatively to the strip
$]-\pi,\pi[\times{\mathbb R}.$
The left panel shows the image of the
Poincar\'{e} map after time $t=20$ of the vertical segment $y(0)=0,$
$y'(0)=\phi\in ]0,1.998].$
The right panel shows the graph of the map $\varphi$ which associates to $\phi\in [1,5]$
the value of the solutions $y(t,\phi)$, respectively $y'(t,\phi)$, at the time $t=20.$

Both numerical examples above provide simple cases from classical
physical problems, in which the map $g$, is far from
being monotone.

{\bf c)\;\;} As a third example, we consider a variant of case
$b),$ where the dependence of the solution on a random parameter
is more involved. We study a
second-order problem of the form
\begin{equation}\label{rde-1.3}
\begin{cases}
y''(t)+f\bigl(y(t),y'(t),t\bigr) = k\phi\\
y(t_0)=a(\phi), \; y'(t_0)=b.
\end{cases}
\end{equation}
where we have consider a more general dependence on the single random variable other than in the initial conditions, although the RV does not change with time. In this third instance, we can produce very complicated non-monotone maps
from apparently rather simple boundary value problems. let us consider harmonic oscillator with constant external force,
\begin{equation}\label{eq-1.4}
\frac{1}{\omega^2} y''(t) +  y(t) = k\phi
\end{equation}
which produces the mapping
\begin{equation}\label{eq-1.5}
\varphi(\phi,t)=k\phi +A \cos(\omega t+\phi).
\end{equation}
In addition, we further assume the following initial conditions
$$y(0)=k\phi+A, \quad y'(0)=0.$$
This example will be analyzed more in detail in the next section.

\bigskip

Therefore, as we mentioned above, the aim of the present paper is to introduce a new
method for the numerical approximation of $\mu_Y$ from a given
$\mu_X$ in cases in which the theoretical procedure described the
formula \eqref{eq-1.1} is not effectively applicable. Moreover, we stress that our method can be applied even
if we do not have an analytical expression of the function $g$, which however can be numerically constructed.
This is particularly relevant because, in most of the physical
applications, the transformation is obtained only in the last way.

Our approach is developed in detail in Section
\ref{section-2}. In Section
\ref{section-3} we apply our theory to the example given by
\eqref{eq-1.5}, as a proof of concept. The implemented algorithm is
presented in Section \ref{section-4}. Finally, in Section \ref{section-5} we
provide some further numerical examples where we compare our new
method, with the application of \eqref{eq-1.1} and with the
brute force approach. From these examples,
the advantage of the new point of view will be evident particularly with respect to the
numerical applications.

\section{The problem and a fast introduction to FDF}\label{section-2}
Let us consider an ordinary differential system (like, for
instance, \eqref{rde-1.2bis}) where the initial conditions depend
on a one-dimensional unknown random number $\phi$ with mass
probability function
$\mu:[\alpha,\beta]\subset\mathbb{R}\to\mathbb{R}$. As a typical
application we might consider a projection of the solution onto a
one-dimensional subspace, thus obtaining a new mass probability
distribution function (PDF) $\psi(y,t)$ that evolves in time. That
is, we look for the PDF $\mu_Y :=\psi(y,t)$ that describes the
probability of finding our system in the position $y$ at time $t$,
provided that we know the probability distribution $\mu_X :=
\mu(\phi)$.

For each $\phi\in[\alpha,\beta]$, the (projected)
\underline{deterministic} solution of the differential equation is
a function $\varphi_{\phi}(t):[0,t_{max}]\to\mathbb{R}$ that
describes how the system evolves in time $t$. Under standard
smoothness conditions on the vector field (cf. \cite{Ha-1980}) the
function $\varphi(\phi,t):=\varphi_{\phi}(t)$ is of class
$\mathcal{C}^1\bigl([\alpha,\beta]\times[0,t_{max}]\bigr)$. As we
have discussed in the introduction, for any fixed $t>0,$ the map
$\phi\mapsto \varphi_{\phi}(t)$ may be non-invertible. The idea
behind the present work is that the evolution in space
$y(t)=\varphi(\phi,t)$ of the system, spreads, squeezes or
stretches and overlaps or folds the probability mass $\mu(\phi)$
in the space. Then we propose to replace the representation of the
physical system given by $\varphi(\phi,t)$, with a new
mathematical entity, which we call \emph{folding domain function
FDF}, such that we can determine $\psi(y,t)$ with the minimum
computational cost. Although in the present demonstration the FDF
method is applied assuming that we already know the analytical
solution $\varphi$ of the mapping, this does not affect the
applicability of the FDF algorithm (FDF-A) to cases where the
differential system is not analytically solvable, but a suitable
approximation to the solution can be obtained from {\em numerical
integrators}. As a matter of fact, the FDF-A algorithm only
requires the knowledge of a finite set of values
$\varphi_i:=\varphi(\phi_i,t)$ (with $i\in\{1,\dots,N_{div}\}$ and
$N_{div}\in\mathbb{N}$), for some initial conditions
$\phi_i\in[\alpha,\beta]$, which contain all the relevant
information for the mapping and that we will define later on.
Interestingly, these particular values can be determined from the
knowledge of the analytical solution of the differential equation
but, more important, from the application of any numerical solver
of the initial value problem.

\subsection{The folding domain functions}\label{sub-1.1}

Let $X$ be a RV defined on a probabilistic measure space with values in a compact interval $[\alpha,\beta]$
and let $g:[\alpha,\beta]\to [g_{\min},g_{\max}]$ be a continuously differentiable function
which is strictly piecewise monotone. Accordingly, we decompose the domain of $g$ into a finite sequence
of adjacent intervals, by setting
$$\alpha=\alpha_0<\alpha_1<\dots< \alpha_{j-1} < \alpha_{j} <\dots \alpha_{k}=\beta,$$
with the convention that $g$ is strictly monotone with nonzero derivative
on each subinterval $]\alpha_{j-1},\alpha_j[$
and, moreover, each $\alpha_j$ with $0<j<k$ is a
stationary point which is a strict local maximum or
local minimum. Let also
$$A:=\{\alpha_j: j=0,\dots, k\}.$$
Notice that, by construction, the map $g$ is strictly increasing/decreasing restricted
between two consecutive points in $A.$
It will be also convenient to introduce the following notation
$$
\begin{cases}
A_j:= [\alpha_{j-1},\alpha_{j}[, \quad\text{for }\; j=1,\dots, k-1\\
A_k:= [\alpha_{k-1},\alpha_k]= [\alpha_{k-1},\beta],
\end{cases}
$$
so that the intervals $A_j$ ($j=1,\dots,k$)
determine a partition of the interval $[\alpha,\beta].$

We define the vector $\Lambda:=(\lambda_1,\dots,\lambda_k),$
where, for each $j=1,\dots,k,$ we set
$$\lambda_j:= g(\alpha_j)-g(\alpha_{j-1})$$
(see Figure \ref{fig-A1} for an illustrative example).
\begin{figure}[ht]
\begin{center}
\begin{tikzpicture}[scale=0.4]
\pgfmathsetmacro{\p}{0.5}
\draw[->](-1.5,-1.12)--(23,-1.12)node[below]{$x$};
\draw[->](-1,-2.62)--(-1,12)node[left]{$y$}; \draw[ultra
thick,color=magenta](1,-1.12)node[below]{$\alpha\color{gray}=\alpha_0$}--(21,-1.12)node[below]{$\beta\color{gray}=\alpha_k$};
 \draw[thick,color=green] plot[samples=130, domain=1:21] (\x,{(0.08*(\x-7)^2-4*cos(deg(\x-7)))+5)*\p}) node[right] {$g(x)$};
 \pgfmathsetmacro{\bu}{(0.08*(-6)^2-4*cos(deg(-6)))+5)*\p}
 \draw[dashed,gray](-1,\bu)node[left]{$b_1$}--(21,\bu);
 \node[color=yellow] at (-1,\bu+1.3){$\bullet$};
 \node[right,color=pink] at (-1,\bu+1.3){$c_2$};
 \pgfmathsetmacro{\bd}{(0.08*(0.18+pi)^2-4*cos(deg(0.18+pi)))+5)*\p}
 \draw[dashed,gray](-1,\bd)node[left]{$b_2$}--(21,\bd);
 \node[color=yellow] at (-1,\bd+0.9){$\bullet$};
 \node[right,color=pink] at (-1,\bd+0.9){$c_3$};
 \pgfmathsetmacro{\bt}{(0.08-4*cos(deg(0)))+5)*\p}
 \draw[dashed,gray](-1,\bt)node[left]{$b_0$}--(21,\bt);
 \node[color=yellow] at (-1,\bt+0.75){$\bullet$};
 \node[right,color=pink] at (-1,\bt+0.75){$c_1$};
 \node at (-1,\bt)[below right]{$g_{min}$};
 \pgfmathsetmacro{\bq}{(0.08*(0.41+3*pi)^2-4*cos(deg(0.41+3*pi)))+5)*\p}
 \draw[dashed,gray](-1,\bq)node[left]{$b_4$}--(21,\bq);
 \node[color=yellow] at (-1,\bq+0.85){$\bullet$};
 \node[right,color=pink] at (-1,\bq+0.85){$c_5$};
 \pgfmathsetmacro{\bc}{(0.08*(-0.48+4*pi)^2-4*cos(deg(-0.48+4*pi)))+5)*\p}
 \draw[dashed,gray](-1,\bc)node[left]{$b_3$}--(21,\bc);
 \node[color=yellow] at (-1,\bc+0.75){$\bullet$};
 \node[right,color=pink] at (-1,\bc+0.75){$c_4$};
 \pgfmathsetmacro{\bs}{(0.08*(14)^2-4*cos(deg(14)))+5)*\p}
 \draw[dashed,gray](-1,\bs)node[left]{$b_5=b_{\ell}$}--(21,\bs);
 \node at (-1,\bs)[above right]{$g_{max}$};
 \draw[dashed,color=gray]({7-0.15-pi},2)--({7-0.15-pi},-1.12)node[below]{$\alpha_1$};
 \draw[->,color=red,thick]({7-0.15-pi},2)--({7-0.15-pi},{2+\bd-\bu})node[above]{$\lambda_1$};
 \draw[dashed,color=gray](7,\bt)--(7,-1.12)node[below]{$\alpha_2$};
 \draw[->,color=blue,thick](7,{2+\bd-\bu})node[above]{$\lambda_2$}--(7,\bt);
 \draw[dashed,color=gray]({7+0.18+pi},2)--({7+0.18+pi},-1.12)node[below]{$\alpha_3$};
 \draw[->,color=red,thick]({7+0.18+pi},2)--({7+0.18+pi},{2+\bd-\bu})node[above]{$\lambda_3$};
 \draw[dashed,color=gray]({7-0.26+2*pi},{2+\bd-\bu})--({7-0.26+2*pi},-1.12)node[below]{$\alpha_4$};
 \draw[->,color=blue,thick]({7-0.26+2*pi},{2+\bd-\bu})node[above]{$\lambda_4$}--({7-0.26+2*pi},2);
 \draw[dashed,color=gray]({7+0.41+3*pi},2)--({7+0.41+3*pi},-1.12)node[below]{$\alpha_5$};
 \draw[->,color=red,thick]({7+0.41+3*pi},2)--({7+0.41+3*pi},\bq)node[above]{$\lambda_5$};
  \draw[dashed,color=gray]({7-0.48+4*pi},\bc)--({7-0.48+4*pi},-1.12)node[below]{$\alpha_6$};
 \draw[->,color=blue,thick]({7-0.48+4*pi},\bq)node[above]{$\lambda_6$}--({7-0.48+4*pi},\bc);
 \draw[dashed,color=gray](21,\bc)--(21,-1.12);
 \draw[->,color=red,thick](21,\bc)--(21,\bs)node[above]{$\lambda_7$};
\end{tikzpicture}
\end{center}
\caption{The function $g$.} \label{fig-A1}
\end{figure}
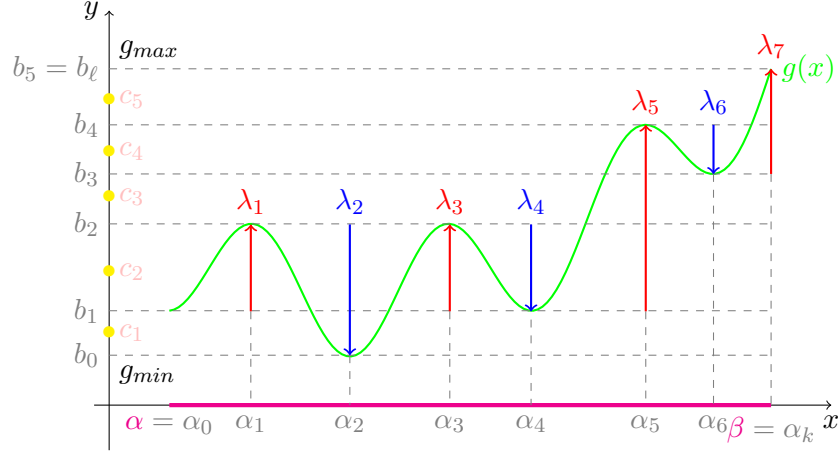
We also introduce the constant
$$S=S_{\Lambda}:= \sum_{j=1}^{k} |\lambda_j|.$$

As a next step, we define the transformation $\hat{g}:[\alpha,\beta]\to [0,S]$ as
\begin{equation}\label{eq-mgj}
\begin{cases}
\hat{g}(x)= m_{j-1} + |g(x)-g(\alpha_{j-1})|,
\quad\text{for }\; x\in [\alpha_{j-1},\alpha_{j}]\\
j=1,\dots, k,
\end{cases}
\end{equation}
where
$$m_0 = 0 \quad \text{and }\; \;
m_j:= |\lambda_1|+\dots+|\lambda_{j}|, \quad j=1,\dots,k.$$
By definition, $\hat{g}(\alpha)=0$ and $\hat{g}(\beta)=S.$
The formula \eqref{eq-mgj} defining $\hat{g}$ is introduced as
a way to unfold the graph of $g.$
Indeed, transformation $\hat{g}$ acts as follows: given a graph of a
piecewise monotone function $g,$ first we shift it as to have
value $0$ at $x=\alpha.$ Then, if $[\alpha_{j-1},\alpha_{j}]$ is
an interval where $g$ is increasing, we further shift this portion
of the graph of $g$ by the quantity $m_{j-1}-g(\alpha_{j-1})=
|\lambda_1|+\dots+|\lambda_{j-1}|-g(\alpha_{j-1}).$ On the other
hand, if $[\alpha_{j-1},\alpha_{j}]$ is an interval where $g$ is
decreasing, we reflect the graph of $g$ symmetrically with respect
to the horizontal line $y= g(\alpha_{j-1})$ and then we shift it
as to obtain the graph of a continuous and monotonously increasing function (see
Figure \ref{fig-A2} for an illustrative example).
Observe that, if $g$ is smooth (of class $C^1$), $\hat{g}$ is smooth, too.

\begin{figure}[ht]
\begin{center}
\begin{tikzpicture}[scale=0.33]
\pgfmathsetmacro{\p}{0.5} \draw[opacity=0.3,ultra
thick,->](-1.5,-1.12)--(25,-1.12)node[below]{$x$};
\draw[opacity=0.3,ultra
thick,->](-1,-1.62)--(-1,12)node[left]{$y$};
 \draw[->](-1.5,2)--(25,2)node[below]{$\hat{x}$};
\draw[->](-1,1.6)--(-1,29)node[left]{$u$}; \draw[ultra
thick,color=magenta](1,2)node[below
right]{$\alpha$}--(21,2)node[below right]{$\beta$};
 \draw[ultra thick,color=green] plot[samples=130, domain=1:21] (\x,{(0.08*(\x-7)^2-4*cos(deg(\x-7)))+5)*\p}) node[right] {$g(x)$};
\draw[thick,color=brown] plot[samples=30, domain=1:6.85-pi]
(\x,{(0.08*(\x-7)^2-4*cos(deg(\x-7)))+5)*\p});
\pgfmathsetmacro{\a}{0.08*(-0.15-pi)^2-4*cos(deg(-0.15-pi)))+5}
\draw[thick,color=brown] plot[samples=30, domain=6.85-pi:7]
(\x,{(\a-0.08*(\x-7)^2+4*cos(deg(\x-7)))+4.8)*\p});
\pgfmathsetmacro{\b}{-0.08*(-7)^2+4*cos(deg(-7)))+5}
\draw[thick,color=brown] plot[samples=30, domain=7:7.18+pi]
(\x,{(\a+\b+0.08*(\x-7)^2-4*cos(deg(\x-7)))+8.68)*\p});
\draw[thick,color=brown] plot[samples=30,
domain=7.18+pi:6.84+2*pi]
(\x,{(\a+\b-0.08*(\x-7)^2+4*cos(deg(\x-7)))+18.3)*\p});
\draw[thick,color=brown] plot[samples=30,
domain=6.84+2*pi:7.41+3*pi]
(\x,{(\a+\b+0.08*(\x-7)^2-4*cos(deg(\x-7)))+20.25)*\p});
\draw[thick,color=brown] plot[samples=30,
domain=7.41+3*pi:6.62+4*pi]
(\x,{(\a+\b-0.08*(\x-7)^2+4*cos(deg(\x-7)))+43.1)*\p});
\draw[thick,color=brown] plot[samples=30, domain=6.62+4*pi:21]
(\x,{(\a+\b+0.08*(\x-7)^2-4*cos(deg(\x-7)))+26.8)*\p})node[right]
{$\hat{g}(x)$};
\draw[dashed,gray](1,-1.12)node[below]{$\alpha_0$}--(1,28);
 \draw[dashed,gray](7-0.15-pi,-1.12)node[below]{$\alpha_1$}--(7-0.15-pi,28);
 \draw[dashed,gray](7,-1.12)node[below]{$\alpha_2$}--(7,28);
 \draw[dashed,gray](7+0.18+pi,-1.12)node[below]{$\alpha_3$}--(7+0.18+pi,28);
 \draw[dashed,gray]({7-0.26+2*pi},-1.12)node[below]{$\alpha_4$}--({7-0.26+2*pi},28);
 \draw[dashed,gray]({7+0.41+3*pi},-1.12)node[below]{$\alpha_5$}--({7+0.41+3*pi},28);
 \draw[dashed,gray]({7-0.48+4*pi},-1.12)node[below]{$\alpha_6$}--({7-0.48+4*pi},28);
 \draw[dashed,gray](21,-1.12)node[below]{$\alpha_7$}--(21,28);
 \pgfmathsetmacro{\bu}{(0.08*(-6)^2-4*cos(deg(-6)))+5)*\p}
 \pgfmathsetmacro{\bd}{(0.08*(0.18+pi)^2-4*cos(deg(0.18+pi)))+5)*\p}
 \pgfmathsetmacro{\bt}{(0.08-4*cos(deg(0)))+5)*\p}
 \pgfmathsetmacro{\bq}{(0.08*(0.41+3*pi)^2-4*cos(deg(0.41+3*pi)))+5)*\p}
 \pgfmathsetmacro{\bc}{(0.08*(-0.48+4*pi)^2-4*cos(deg(-0.48+4*pi)))+5)*\p}
 \pgfmathsetmacro{\bs}{(0.08*(14)^2-4*cos(deg(14)))+5)*\p}
   \draw[->,color=red,thick](-1,2)--(-1,{2+\bd-\bu})node[left,color=black]{$m_1$};
   \node at (-1,{1+\bd-\bu})[below right,color=red]{$\lambda_1$};
   \draw[dashed](-1,{2+\bd-\bu})--({7-0.15-pi},{2+\bd-\bu});
   \draw[->,color=blue,thick](-1,{2+\bd-\bu})--(-1,{2+\bd-\bu+\bd-\bt})node[left,color=black]{$m_2$};
   \node at (-1,{1+\bd-\bu+\bd-\bt})[below right,color=blue]{$\lambda_2$};
   \draw[dashed](-1,{2+\bd-\bu+\bd-\bt})--(7,{2+\bd-\bu+\bd-\bt});
   \pgfmathsetmacro{\yu}{13.7}
   \pgfmathsetmacro{\yd}{16.62}
   \draw[->,color=red,thick](-1,{2+\bd-\bu+\bd-\bt})--(-1,\yu)node[left,color=black]{$m_3$};
   \node at (-1,\yu-1.5)[below right,color=red]{$\lambda_3$};
   \draw[dashed](-1,\yu)--({7+0.18+pi},\yu);
   \draw[->,color=blue,thick](-1,\yu)--(-1,\yd)node[left,color=black]{$m_4$};
   \node at (-1,\yd-0.5)[below right,color=blue]{$\lambda_4$};
   \draw[dashed](-1,\yd)--({7-0.26+2*pi},\yd);
   \pgfmathsetmacro{\yt}{22.8}
   \draw[->,color=red,thick](-1,\yd)--(-1,\yt)node[left,color=black]{$m_5$};
   \node at (-1,\yt-2)[below right,color=red]{$\lambda_5$};
   \draw[dashed](-1,\yt)--({7+0.41+3*pi},\yt);
   \pgfmathsetmacro{\yq}{24.45}
   \draw[->,color=blue,thick](-1,\yt)--(-1,\yq)node[left,color=black]{$m_6$};
   \node at (-1,\yq-0.2)[below right,color=blue]{$\lambda_6$};
   \draw[dashed](-1,\yq)--({7-0.48+4*pi},\yq);
   \pgfmathsetmacro{\yc}{27.9}
   \draw[->,color=red,thick](-1,\yq)--(-1,\yc)node[left,color=black]{$S_{\Lambda}=m_7$};
   \node at (-1,\yc-1)[below right,color=red]{$\lambda_7$};
   \draw[dashed](-1,\yc)--(21,\yc);
\end{tikzpicture}
\end{center}
\caption{The functions $g$ and $\hat{g}$.} \label{fig-A2}
\end{figure}
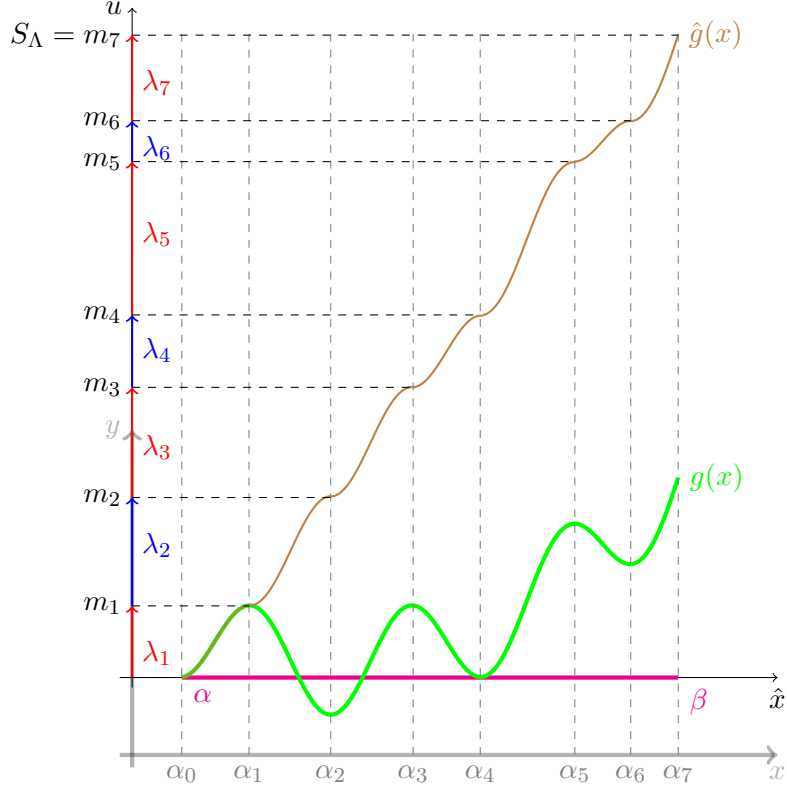

By construction, for any continuously differentiable and piecewise
monotone function $g:[\alpha,\beta]\to [g_{\min},g_{\max}]$
we have that $\hat{g}:[\alpha,\beta]\to [0,S]$ is strictly monotonously increasing with inflection
points at $\alpha_i$ ($i=1,\dots,k-1$) and, moreover,
$$\frac{d}{dx}\hat{g}(x)= \left|\frac{d}{dx}g(x)\right|,\quad\forall\, x\in [\alpha,\beta].$$
At this point, the inverse function of $\hat{g}$ is well defined as
$$\eta:[0,S]\to [\alpha,\beta].$$
This new map is again strictly monotonously increasing and continuously differentiable on
$[0,S]\setminus\{\hat{g}(\alpha_i), i=1,\dots, k-1\},$ where the derivative in the excluded points exists
with value $+\infty$ (see Figure \ref{fig-A3}).

\begin{figure}[ht]
\begin{center}
\begin{tikzpicture}[scale=0.35]
\pgfmathsetmacro{\p}{0.5}
 \draw[->](2,-1.5)--(2,23)node[left]{$\hat{x}$};
\draw[->](1.4,-1)--(29,-1)node[right]{$u$}; \draw[ultra
thick,color=magenta](2,1)node[below
right]{$\alpha$}--(2,21)node[below right]{$\beta$};
\draw[thick,color=orange] plot[samples=30, domain=1:6.85-pi]
({(0.08*(\x-7)^2-4*cos(deg(\x-7)))+5)*\p},\x);
\pgfmathsetmacro{\a}{0.08*(-0.15-pi)^2-4*cos(deg(-0.15-pi)))+5}
\draw[thick,color=orange] plot[samples=30, domain=6.85-pi:7]
({(\a-0.08*(\x-7)^2+4*cos(deg(\x-7)))+4.8)*\p},\x);
\pgfmathsetmacro{\b}{-0.08*(-7)^2+4*cos(deg(-7)))+5}
\draw[thick,color=orange] plot[samples=30, domain=7:7.18+pi]
({(\a+\b+0.08*(\x-7)^2-4*cos(deg(\x-7)))+8.68)*\p},\x);
\draw[thick,color=orange] plot[samples=30,
domain=7.18+pi:6.84+2*pi]
({(\a+\b-0.08*(\x-7)^2+4*cos(deg(\x-7)))+18.3)*\p},\x);
\draw[thick,color=orange] plot[samples=30,
domain=6.84+2*pi:7.41+3*pi]
({(\a+\b+0.08*(\x-7)^2-4*cos(deg(\x-7)))+20.25)*\p},\x);
\draw[thick,color=orange] plot[samples=30,
domain=7.41+3*pi:6.62+4*pi]
({(\a+\b-0.08*(\x-7)^2+4*cos(deg(\x-7)))+43.1)*\p},\x);
\draw[thick,color=orange] plot[samples=30, domain=6.62+4*pi:21]
({(\a+\b+0.08*(\x-7)^2-4*cos(deg(\x-7)))+26.8)*\p},\x)node[right]
{$\eta(u)$};
\draw[dashed,gray](1,1)node[left]{$\alpha_0$}--(28,1);
 \draw[dashed,gray](1,7-0.15-pi)node[below]{$\alpha_1$}--(28,7-0.15-pi);
 \draw[dashed,gray](1,7)node[left]{$\alpha_2$}--(28,7);
 \draw[dashed,gray](1,7+0.18+pi)node[left]{$\alpha_3$}--(28,7+0.18+pi);
 \draw[dashed,gray](1,{7-0.26+2*pi})node[left]{$\alpha_4$}--(28,{7-0.26+2*pi});
 \draw[dashed,gray](1,{7+0.41+3*pi})node[left]{$\alpha_5$}--(28,{7+0.41+3*pi});
 \draw[dashed,gray](1,{7-0.48+4*pi})node[left]{$\alpha_6$}--(28,{7-0.48+4*pi});
 \draw[dashed,gray](1,21)node[left]{$\alpha_7$}--(28,21);
 \pgfmathsetmacro{\bu}{(0.08*(-6)^2-4*cos(deg(-6)))+5)*\p}
 \pgfmathsetmacro{\bd}{(0.08*(0.18+pi)^2-4*cos(deg(0.18+pi)))+5)*\p}
 \pgfmathsetmacro{\bt}{(0.08-4*cos(deg(0)))+5)*\p}
 \pgfmathsetmacro{\bq}{(0.08*(0.41+3*pi)^2-4*cos(deg(0.41+3*pi)))+5)*\p}
 \pgfmathsetmacro{\bc}{(0.08*(-0.48+4*pi)^2-4*cos(deg(-0.48+4*pi)))+5)*\p}
 \pgfmathsetmacro{\bs}{(0.08*(14)^2-4*cos(deg(14)))+5)*\p}
   \draw[->,color=red,thick](2,-1)--({2+\bd-\bu},-1)node[below,color=black]{$m_1$};
   \node at ({2+\bd-\bu-0.5},-1) [above left,color=red]{$\lambda_1$};
   \draw[dashed]({2+\bd-\bu},-1)--({2+\bd-\bu},{7-0.15-pi});
   \draw[->,color=blue,thick]({2+\bd-\bu},-1)--({2+\bd-\bu+\bd-\bt},-1)node[below,color=black]{$m_2$};
   \node at ({2+\bd-\bu+\bd-\bt-1},-1) [above left,color=blue]{$\lambda_2$};
   \draw[dashed]({2+\bd-\bu+\bd-\bt},-1)--({2+\bd-\bu+\bd-\bt},7);
   \pgfmathsetmacro{\yu}{13.7}
   \pgfmathsetmacro{\yd}{16.62}
   \draw[->,color=red,thick]({2+\bd-\bu+\bd-\bt},-1)--(\yu,-1)node[below,color=black]{$m_3$};
   \node at (\yu-1.5,-1) [above left,color=red]{$\lambda_3$};
   \draw[dashed](\yu,-1)--(\yu,{7+0.18+pi});
   \draw[->,color=blue,thick](\yu,-1)--(\yd,-1)node[below,color=black]{$m_4$};
   \node at (\yd-0.5,-1)[above left,color=blue]{$\lambda_4$};
   \draw[dashed](\yd,-1)--(\yd,{7-0.26+2*pi});
   \pgfmathsetmacro{\yt}{22.8}
   \draw[->,color=red,thick](\yd,-1)--(\yt,-1)node[below,color=black]{$m_5$};
   \node at (\yt-2,-1)[above left,color=red]{$\lambda_5$};
   \draw[dashed](\yt,-1)--(\yt,{7+0.41+3*pi});
   \pgfmathsetmacro{\yq}{24.45}
   \draw[->,color=blue,thick](\yt,-1)--(\yq,-1)node[below,color=black]{$m_6$};
   \node at (\yq-0.2,-1)[above left,color=blue]{$\lambda_6$};
   \draw[dashed](\yq,-1)--(\yq,{7-0.48+4*pi});
   \pgfmathsetmacro{\yc}{27.9}
   \draw[->,color=red,thick](\yq,-1)--(\yc,-1)node[below,color=black]{$S_{\Lambda}=m_7$};
   \node at (\yc-1,-1)[above left,color=red]{$\lambda_7$};
   \draw[dashed](\yc,-1)--(\yc,21);
\end{tikzpicture}
\end{center}
\caption{The function $\eta$.} \label{fig-A3}
\end{figure}
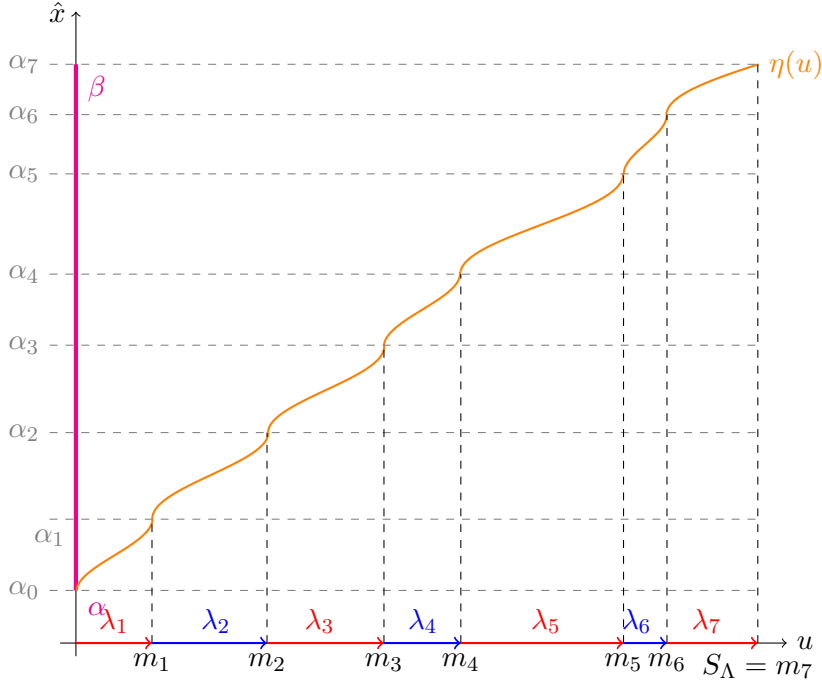

Let us consider the set
\begin{equation*}
\begin{split}
B=\{g(\alpha_{i}), i=0,\dots,k\} = \{b_0,b_1,\dots,b_{\ell}\},\\
\text{with } \qquad
g_{\min}=b_0 < b_1 < \dots < b_{\ell} = g_{\max}
\end{split}
\end{equation*}
and decompose $[g_{\min},g_{\max}]$ into non-overlapping sets $B_i$ as
$$[g_{\min},g_{\max}] = B_1\cup B_2\cup \dots \cup B_{\ell},$$
where
$$
\begin{cases}
B_i:= [b_{i-1},b_{i}[, \quad\text{for }\; i=1,\dots, \ell-1\\
B_{\ell}:= [b_{\ell-1},b_{\ell}]= [b_{\ell-1},g_{\max}].
\end{cases}
$$
Note that for each $y\in [g_{\min},g_{\max}]$ there is a (nonempty) finite set of
points in the inverse image $g^{-1}(\{y\})\subset [\alpha,\beta].$ Thus the (finite)
set of indexes
$$I(y):=\{j\in \{1,\dots,k\}: \exists \, x\in A_j\,: g(x)=y\}$$
is well defined.

\begin{lemma}\label{lem-A1}
If $W\subset [\alpha,\beta]$ is an open interval such that $g(W)=]b_{i-1},b_{i}[$
for some $i\in\{1,\dots,\ell\},$ then $g$ is strictly monotone in $W.$
\end{lemma}
\begin{proof}
For the proof it is sufficient to observe that, by our assumption,
the only possible critical points of $g$ are contained in the set $A.$ On the other hand,
$B=g(A)$ and $]b_{i-1},b_{i}[\subset [g_{\min},g_{\max}]\setminus B.$ As a consequence,
$g'(x)\not=0$ for all $x\in W$ and therefore, $g$ is strictly monotone in $W.$
\end{proof}

\begin{lemma}\label{lem-A2}
For each $y\in ]b_{i-1},b_{i}[$ (where $i\in\{1,\dots,\ell\}$ an arbitrary index),
the set $I(y)$ is constant. Therefore, $I(y)=I(c_i)$ where
$c_i:=\frac{b_{i-1}+b_{i}}{2}$.
\end{lemma}
\begin{proof}
For a fixed $i\in\{1,\dots,\ell\},$ let us consider the middle point
$c_i:=\frac{b_{i-1}+b_{i}}{2}$ and let $\gamma_1<\gamma_2<\dots<\gamma_p$ be such that
$g^{-1}(\{c_i\})=\{\gamma_1,\dots,\gamma_p\}.$ For each $\gamma_s,$ with $s=1,\dots,p,$
let $W_s$ be an open maximal interval containing $\gamma_s$ and such that
$g(W_s)=]b_{i-1},b_{i}[.$ By Lemma \ref{lem-A1}, $g$ is strictly monotone on $W_s$
and, therefore, the open intervals $W_s$ are pairwise disjoint. In fact, if, by
contradiction, two of these intervals, say $W_s$ and $W_r,$ overlap, then
$W_s\cup W_r$ turns out to be a unique interval $W$ such that $g(W)=]b_{i-1},b_{i}[$
and hence $g$ must be strictly monotone on $W,$ contradicting the fact that
$g(\gamma_s)=g(\gamma_r)= c_i$ (with $\gamma_s\not=\gamma_r$). As a consequence of the fact that the open intervals are pairwise disjoint,
we conclude that
$$g^{-1}(\,]b_{i-1},b_{i}[\,)= W_1\cup\dots W_p$$
and then
$$\# g^{-1}(\{y\}) = \text{constant}, \quad \forall\, y\in \,]b_{i-1},b_{i}[,$$
where $\# E$ denotes the cardinality of the set $E$ (in our case it is
just the number of points, as all the involved sets are finite).
\\
As a final observation, we claim that each of the intervals $W_s$ is contained
in exactly one of the intervals $A_j.$ Indeed, let $r\in \{1,\dots,p\}$ be such
that $W_r\subset \,]\alpha_{j-1},\alpha_{j}[$ and let $s\not=r.$ If, by contradiction,
also $W_s\subset \,]\alpha_{j-1},\alpha_{j}[$, then
$\gamma_r,\gamma_s\in\,]\alpha_{j-1},\alpha_{j}[$ with $\gamma_r\not=\gamma_s$
and $g(\gamma_r)=g(\gamma_s)=c_i,$ contrary to the fact that $g$ is strictly
monotone on $A_j=[\alpha_{j-1},\alpha_{j}].$

This latter observation guarantees that $I(y)$ is constant for each
$y\in \,]b_{i-1},b_{i}[$ and the proof is complete.
\end{proof}

\begin{remark}\label{rem-A1}
It is interesting to observe that the result about the constancy of
$\# g^{-1}(\{y\})$ for all $y\in \,]b_{i-1},b_{i}[$ is a general property
which is valid also for locally invertible and proper maps on metric spaces,
as proved in \cite[\S3]{AmPr-1995}.
\end{remark}

As a next step we want to introduce a formal method to determine,
for a given $y\in [g_{\min},g_{\max}],$ the set $I(y).$ We will treat separately the case
when $y=b_i$, or
$y\in \,]b_{i-1},b_{i}[$ (for some $i\in \{1,\dots,\ell\}$). The first situation
concerns the analysis of a finite set and will be treated in a second moment.
Concerning the second case, by virtue of Lemma \ref{lem-A2}, it will be sufficient
to consider $I(c_i)$, where $c_i$ is the middle point of each interval.
It is straightforward to check that $y\in ]g(\alpha_{j-1}),g(\alpha_j)[$
(when $g(\alpha_{j-1}) < g(\alpha_j)$) or, respectively
$y\in ]g(\alpha_{j}),g(\alpha_{j-1})[$
(when $g(\alpha_{j-1}) > g(\alpha_j)$), provided that
\begin{equation}\label{eq-yglambda}
0 < (y - g(\alpha_{j-1})){\rm sign}(\lambda_{j}) < |\lambda_j|,
\end{equation}
a condition which can be easily implemented from an algorithmic point of view.
To better understand the meaning of \eqref{eq-yglambda}, recall that
${\rm sign}(\lambda_{j})$ determines whether $g$ is increasing or decreasing
on each interval $[\alpha_j,\alpha_{j+1}].$

Since the map $g$ is (in general) non-invertible, the idea now is to spread
the point $y$ on the interval $[0,S]$, which is the range of the function $\hat{g}$
and the domain of the function $\eta,$ in such a way that, at each $y\in \,]b_{i-1},b_{i}[$, we associate
a set of points
$$u_{\rho_1}(y), u_{\rho_2}(y), \dots, u_{\rho_p}(y),
\quad \text{for }\; I(y) = \{\rho_1, \rho_2,\dots, \rho_p\}.$$
More precisely, we obtain these values by setting
\begin{equation}\label{form-1}
u_j(y)= m_{j-1} + (y-g(\alpha_{j-1})){\rm sign}(\lambda_{j}),\quad
\text{for }\; j\in I(y).
\end{equation}
Then, by means of the inverse mapping $\eta$, we can obtain the set of points of the initial domain which share the same image $y$. More formally, we have that
$$x_j(y) = \eta(u_j(y))= \left(g|_{]\alpha_{j-1},\alpha_j[}\right)^{-1}(y),
\quad \text{for }\; j\in I(y).$$ which is the collection of points in the domain of $g$ that share the same image.
In conclusion, via the function $\eta$ we obtain a cumulative inverse
of all the local inverse functions of $g$ restricted to the single intervals
$A_j$, where $g$ is strictly monotone.

We summarize the procedure
described above, by recalling the previous diagrams.
We start in Figure \ref{fig-A1}
from a function $g$ where we split its
domain into a finite number of adjacent intervals where $g$ is
either strictly increasing or strictly decreasing. On the vertical
axis we have also indicated the set of points
$B=\{b_0,\dots,b_{\ell}\}$ which is the set
$\{g(\alpha_0),\dots,g(\alpha_k)\}$ with its elements arranged in
the natural order. The figure also provides a visual illustration
of Lemma \ref{lem-A2}. Indeed, it is apparent that $I(y)$ is
constant and equal to $I(c_i)$ for each $y\in ]b_{i-1},b_{i}[.$
The same figure also puts in evidence the vector
$\Lambda=(\lambda_1,\dots,\lambda_k).$

Figure \ref{fig-A2} illustrates the transformation $g\mapsto \hat{g}$
and reports the starting coordinate system $(x,y)$ and the translate
$(\hat{x},u),$ which shifts
$g(\alpha_0)$ to $0.$ The figure also puts in evidence the new special points $m_i$
on the $u$-axis.

Finally, Figure \ref{fig-A3} shows the graph of the function $\eta= \hat{g}^{-1}$.
As explained theoretically above, for any given image $y$ we have a method
to ``invert'' the function $g$ via the function $\eta$, provided that we correctly associate the set of points $u_j(y)$ for $y\in I(y)$ to $y\in [g_{\min},g_{\max}]$

It remains to discuss the case when $y\in B,$ namely when $y=b_i$
for some $i=0,\dots,\ell$. We recall that for these points, it may
exist $x_j(y)$ such that $g'(x_j(y)) = 0$. The fact that the $\#B$
is finite assures that these points are irrelevant in probability
theory as they represent a ensemble of zero measure. However, we
provide a more detailed demonstration for completeness.

We start by mentioning again that
$\# g^{-1}(\{y\})$ is constant only on
$]b_{i},b_{i+1}[$ (according to Lemma \ref{lem-A2}), but it
changes at the ends of the interval. In general, the set
$g^{-1}(\{b_i\})$ contains at least a critical point $\alpha_j$ (with
$j=1,\dots,k-1$) or an extreme point $\alpha_j$ (with
$j=0,k$) and, possibly other points in the interior of the
intervals $A_j$ where $g$ is strictly monotone. This situation
is evident from the example in Figure \ref{fig-A1}: the
$g^{-1}(\{b_3\})$ contains a regular point and a critical point of minimum;
When we increase $y$ passing across $b_3,$ we find $1-2-3$ solutions.
On the other hand, the set $g^{-1}(\{b_4\})$
contains a regular point and a critical point of maximum and then, when
we increase $y$ passing across $b_4,$ we find $3-2-1$ solutions.
This example reflects a general situation and suggests the need to
distinguish, among the critical points,the strict local minima and the
strict local maxima. This intuitive explanation is now formalized
in what follows.

We split $g^{-1}(\{b_i\})$ as
$$g^{-1}(\{b_i\})= {\mathscr C}_m^i \cup {\mathscr C}_M^i \cup {\mathscr R}^i
\cup {\mathscr E}_m^i \cup {\mathscr E}_M^i,$$
where ${\mathscr C}_m^i$ and ${\mathscr C}_M^i$ denote the sets of
critical points which are interior strict local minima or strict local
maxima, respectively, while ${\mathscr R}^i$ is the set of regular
points in the interior of the interval. We denote by ${\mathscr E}_m^i$
and ${\mathscr E}_M^i$ the sets of extreme points
$\{\alpha,\beta\}$ (having $b_i$ as image) which are,
respectively, local minima or local maxima.
Now the rule to count the number of inverse images passing
from an interval $]b_{i-1},b_{i}[$ to the next one
$]b_{i},b_{i+1}[$ across the point $b_i$ is given by:
$$\#I(c_{i+1}) =
\#I(c_{i}) + 2\#{\mathscr C}_m^i - 2\#{\mathscr C}_M^i - \#{\mathscr E}_M^i.$$
Respectively, the rule passing
from an interval $]b_{i},b_{i+1}[$ to the previous one
$]b_{i-1},b_{i}[$ across the point $b_i$ is given by:
$$\#I(c_{i}) =
\#I(c_{i+1}) - 2\#{\mathscr C}_m^i + 2\#{\mathscr C}_M^i - \#{\mathscr E}_m^i.$$
Conversely, if we know the type of the points in $g^{-1}(\{b_i\}),$
we can determine the number of elements of $I(y)$ for $y$ in an
interval having $b_i$ as an extremal point. In fact, the following
holds:
$$\#I(y)=\#g^{-1}(\{b_i\}) + \#{\mathscr C}_m^i - \#{\mathscr C}_M^i  - \#{\mathscr E}_M^i,
\quad\forall\, y\in\, ]b_i,b_{i+1}[\,.$$
To justify the above formula, we observe that if $b_i= g(\alpha_j)$
(for some $j=1,\dots,k-1$), where $\alpha_j$ is an interior point of strict local minimum,
then for $\varepsilon>0$ and sufficiently small we have exactly two points
$x_{j,\rm left}< \alpha_j < x_{j,\rm right}$ in a (small) neighborhood of $\alpha_j$
such that $g(x_{j,\rm left}) = g(x_{j,\rm right}) =y.$ If $b_i= g(\alpha_j)$
(for $j=0,k$) is the image of an extreme point which is a local minimum,
then there is exactly one point $\tilde{x}$ in a right neighborhood of $\alpha$
(respectively in a left neighborhood of $\beta$) such that $g(\tilde{x})=y.$
If $b_i= g(x^*)$ for some $x^*\in {\mathscr R}^i,$ we have that $g$ is strictly
monotone (increasing or decreasing) in an open interval containing $x^*$
and therefore there exists a unique point in such an open interval having
$y$ as image. Finally, if $b_i= g(\alpha_j)$ for some $j=0,\dots,k$ with
$\alpha_j$ a point of local maximum (in the interior or at the boundary of $[\alpha,\beta]$),
there are no solutions of $g(x)=y$ in a neighborhood of $\alpha_j.$
In this manner we have counted all the possible solutions of
$g(x)=y\in ]b_i,b_i+\varepsilon[$ for $\varepsilon>0$ and sufficiently small.
Taking into account that the number of these solutions is constant
in the open interval $]b_i,b_{i+1}[$ we have proved the above formula.

In a similar manner we prove that
$$\#I(y)=\#g^{-1}(\{b_i\}) - \#{\mathscr C}_m^i + \#{\mathscr C}_M^i - \#{\mathscr E}_m^i,
\quad\forall\, y\in\, ]b_{i-1},b_i[\,.$$

As a final step, we propose a visual description
(see Figure \ref{fig-A4}) to detect the inverse images
of a point $y\in [g_{\min},g_{\max}]$ using the components of the vector $\Lambda$
and the set $I(y).$

Our procedure consists in starting from the point $g(\alpha_0)=b_{i^*}\in B$
(for some $i^*\{1,\dots,\ell\}$)
and superimpose suitable \textit{layers} by taking $\lambda_1,\dots,\lambda_k.$
All the points $b_i$ are obtained as points of the form
$$b_{i^*}+\lambda_1=g(\alpha_1), \quad
b_{i^*}+\lambda_1 +\lambda_2=g(\alpha_2), \;\dots, \;
b_{i^*}+\sum_{j=1}^{k}\lambda_j=g(\alpha_k)=g(\beta).$$

\begin{figure}[ht]
\begin{center}
\begin{tikzpicture}[scale=0.35]
\pgfmathsetmacro{\p}{0.5}
\draw[->](-2.62,-1)--(13,-1)node[below]{$y$};
 \pgfmathsetmacro{\bu}{(0.08*(-6)^2-4*cos(deg(-6)))+5)*\p}
 \draw[dashed,gray](\bu,-1)node[below right]{$b_1$}--(\bu,15);
 \pgfmathsetmacro{\bd}{(0.08*(0.18+pi)^2-4*cos(deg(0.18+pi)))+5)*\p}
 \draw[dashed,gray](\bd,-1)node[below]{$b_2$}--(\bd,15);
 \pgfmathsetmacro{\bt}{(0.08-4*cos(deg(0)))+5)*\p}
 \draw[dashed,gray](\bt,-1)node[below]{$b_0$}--(\bt,15);
 \node at (\bt,-1)[above left]{$g_{min}$};
 \pgfmathsetmacro{\bq}{(0.08*(0.41+3*pi)^2-4*cos(deg(0.41+3*pi)))+5)*\p}
 \draw[dashed,gray](\bq,-1)node[below right]{$b_4$}--(\bq,15);
 \pgfmathsetmacro{\bc}{(0.08*(-0.48+4*pi)^2-4*cos(deg(-0.48+4*pi)))+5)*\p}
 \draw[dashed,gray](\bc,-1)node[below]{$b_3$}--(\bc,15);
 \pgfmathsetmacro{\bs}{(0.08*(14)^2-4*cos(deg(14)))+5)*\p}
 \draw[dashed,gray](\bs,-1)node[below]{$b_{\ell}$}--(\bs,15);
 \node at (\bs,-1)[above right]{$g_{max}$};
 \draw[dashed,color=gray](12.1,1)node[right]{$j=1$}--(-1.12,1);
 \draw[->,color=red,thick](2,1)--({2+\bd-\bu},1)node[above left]{$\lambda_1$};
 \draw[dashed,color=gray](12.1,3)node[right]{$j=2$}--(-1.12,3);
 \draw[->,color=blue,thick]({2+\bd-\bu},3)node[above left]{$\lambda_2$}--(\bt,3);
 \draw[dashed,color=gray](12.1,5)node[right]{$j=3$}--(-1.12,5);
 \draw[->,color=red,thick](2,5)--({2+\bd-\bu},5)node[above left]{$\lambda_3$};
 \draw[dashed,color=gray]({2+\bd-\bu},5)--(-1.12,5);
 \draw[->,color=blue,thick]({2+\bd-\bu},7)node[above left]{$\lambda_4$}--(2,7);
 \draw[dashed,color=gray](12.1,7)node[right]{$j=4$}--(-1.12,7);
 \draw[->,color=red,thick](2,9)--(\bq,9);
\node[above left,color=red] at(\bq-3,9) {$\lambda_5$};
  \draw[dashed,color=gray](12.1,9)node[right]{$j=5$}--(-1.12,9);
 \draw[->,color=blue,thick](\bq,11)node[above right]{$\lambda_6$}--(\bc,11);
 \draw[dashed,color=gray](12.1,11)node[right]{$j=6$}--(-1.12,11);
 \draw[->,color=red,thick](\bc,13)--(\bs,13)node[above left]{$\lambda_7$};
 \draw[dashed,color=gray](12.1,13)node[right]{$j=7$}--(-1.12,13);
   \draw[->](1,-6.1)--(29,-6.1)node[below]{$u$};
   \node[right,color=pink] at (17,8){$I(y_1)=\{2\}$};
   \node[right,color=green] at (17,6){$I(y_2)=\{1,2,3,4,5\}$};
   \node[right,color=yellow] at (17,4){$I(y_3)=\{5,6,7\}$};
   \draw[->,color=red,thick](2,-6.1)--({2+\bd-\bu},-6.1)node[below,color=black]{$m_1$};
   \node at ({2+\bd-\bu-0.5},-6.1) [above left,color=red]{$\lambda_1$};
   \draw[->,color=blue,thick]({2+\bd-\bu},-6.1)--({2+\bd-\bu+\bd-\bt},-6.1)node[above,color=black]{$m_2$};
   \node at ({1+\bd-\bu+\bd-\bt},-6.1) [above left,color=blue]{$\lambda_2$};
   \pgfmathsetmacro{\yu}{13.7}
   \pgfmathsetmacro{\yd}{16.62}
   \draw[->,color=red,thick]({2+\bd-\bu+\bd-\bt},-6.1)--(\yu,-6.1)node[above,color=black]{$m_3$};
   \node at (\yu-1.5,-6.1) [above left,color=red]{$\lambda_3$};
   \draw[->,color=blue,thick](\yu,-6.1)--(\yd,-6.1)node[above,color=black]{$m_4$};
   \node at (\yd-0.5,-6.1)[above left,color=blue]{$\lambda_4$};
   \pgfmathsetmacro{\yt}{22.8}
   \draw[->,color=red,thick](\yd,-6.1)--(\yt,-6.1)node[above left,color=black]{$m_5$};
   \node at (\yt-2,-6.1)[above left,color=red]{$\lambda_5$};
   \pgfmathsetmacro{\yq}{24.45}
   \draw[->,color=blue,thick](\yt,-6.1)--(\yq,-6.1)node[above,color=black]{$m_6$};
   \node at (\yq-0.2,-6.1)[above left,color=blue]{$\lambda_6$};
   \pgfmathsetmacro{\yc}{27.9}
   \draw[->,color=red,thick](\yq,-6.1)--(\yc,-6.1)node[above,color=black]{$S_{\Lambda}$};
   \node at (\yc-1,-6.1)[above left,color=red]{$\lambda_7$};
    \draw[color=pink](\bd-3.5,-1)node{$\bullet$}--(\bd-3.5,15);
    \node[below,color=pink] at (\bd-3.5,-1){$y_1$};
   \node[color=pink] at ({2+\bd-\bu+3},-6.1){$\bullet$};
   \node[color=pink,below] at ({2+\bd-\bu+3},-6.1){$u_2(y_1)$};

   \draw[color=green](\bu+1,-1)node{$\bullet$}--(\bu+1,15);
    \node[below right,color=green] at (\bu+1,-1){$y_2$};
   \node[color=green] at ({2+1},-6.1){$\bullet$};
   \node[color=green,below] at ({2+1},-6.1){$u_1(y_2)$};
    \node[color=green] at (2+\bd-\bu+\bd-\bu-1,-6.1){$\bullet$};
    \node[color=green,above left] at (2+\bd-\bu+\bd-\bu-1,-6.1){$u_2(y_2)$};
   \node[color=green] at ({2+\bd-\bu+\bd-\bt+1},-6.1){$\bullet$};
   \node[color=green,below] at ({2+\bd-\bu+\bd-\bt+1},-6.1){$u_3(y_2)$};
    \node[color=green] at (\yd+1,-6.1){$\bullet$};
     \node[color=green,below] at (\yd+1,-6.1){$u_5(y_2)$};
    \node[color=green] at (\yu+\bd-\bu-1,-6.1){$\bullet$};
     \node[color=green,below left] at (\yu+\bd-\bu-1,-6.1){$u_4(y_2)$};

     \draw[color=yellow](\bc+1,-1)node{$\bullet$}--(\bc+1,15);
      \node[below ,color=yellow] at (\bc+1,-1){$y_3$};
    \node[color=yellow] at (\yd+1+\bc-\bu,-6.1){$\bullet$};
    \node[color=yellow,below left] at (\yd+1+\bc-\bu,-6.1){$u_5(y_3)$};
    \node[color=yellow] at (\yt-\bc+\bq-1,-6.1){$\bullet$};
     \node[color=yellow,below] at (\yt-\bc+\bq-1,-6.1){$u_6(y_3)$};
    \node[color=yellow] at (\yq+1,-6.1){$\bullet$};
     \node[color=yellow,below right] at (\yq+1,-6.1){$u_7(y_3)$};
\end{tikzpicture}
\end{center}
\caption{The layers.} \label{fig-A4}
\end{figure}
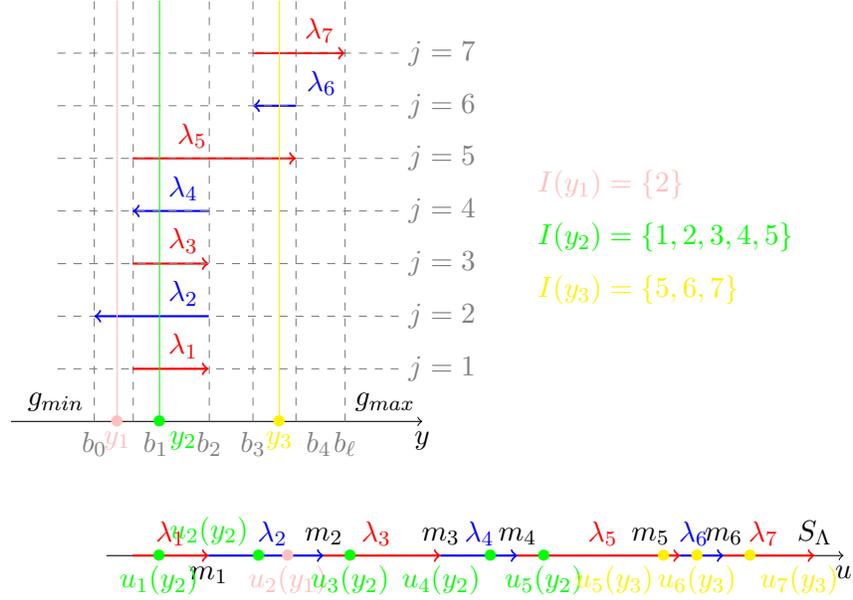

Once the points $g(\alpha_j)$ are arranged in the natural order for the set $B,$
as $\{b_1,\dots,b_{\ell}\},$ for any fixed $\bar{y}\in \,]b_{i-1},b_{i}[\,,$
in order to determine the set $I(y),$ we have just to look for
the number of layers meeting the vertical line $y=\bar{y}$
(as observed above in Lemma \ref{lem-A2}, we can take $\bar{y}= c_i$).
The corresponding scheme is illustrated in Figure \ref{fig-A4} below.
For the figure we take as a function $g,$ the one reported in Figure \ref{fig-A1}.

\begin{remark}\label{rem-A2}
In the sequel we will apply the method described above to maps obtained
from a class of random dynamical systems. In this case, the function $g$
comes from a map $\varphi(x,t),$ for a fixed value of the time $t.$
The derivative of the auxiliary function $\eta$ will determine the density
of a composite RV. From a topological point of view, the evolution of
$\varphi(x,t),$ as $t$ varies in a time-interval, can be interpreted as
sequence od functions $g_t(x)$ which transform (fold) the RV $X$ (with values
in $[\alpha,\beta]$) and mass probability function $\mu_X$ into a new
RV $Y_t=g_t(X).$ In this perspective, it is natural to consider
the above method as a technique which at any time $t$ ``unfolds'' the range
$[g_{\min},g_{\max}]$ of $g$ to the new range $[0,S_{\lambda}]$ which
becomes the new domain of the function $\eta,$ as shown in the example of Figure
\ref{fig-A4}.
\end{remark}

\subsection{The connection with the theory of Random Variable Transformation}\label{sub-1.2}
Let $X$ be a RV with range in a compact interval $[\alpha,\beta]$ and let
$g:[\alpha,\beta]\to [g_{\min},g_{\max}]$ be a smooth and piecewise monotone function as it is the one constructed
in Section \ref{sub-1.1}. Our goal is to show that the FDF is consistent with the general theory which provides the probability distribution function of the new RV $Y=g(X).$

Using the same notation as in the previous section, given any $\tilde{y}\in ]b_{i-1},b_{i}[$
and a sufficiently small $\varepsilon > 0$ such that
$$]\tilde{y}-\varepsilon,\tilde{y}+\varepsilon[\subset  ]b_{i-1},b_{i}[,$$ according to Lemma \ref{lem-A2} we have that
the set of indices $I(y)$, representing the
layers associated with $y$, is constant for all
$y\in ]\tilde{y}-\varepsilon,\tilde{y}+\varepsilon[.$ Thus,
\begin{eqnarray*}
P(\tilde{y}-\varepsilon < Y < \tilde{y}+\varepsilon)
&=&
P\bigl(X\in g^{-1}(\,]\tilde{y}-\varepsilon,\tilde{y}+\varepsilon[\,)\bigr)\\
&=&
\sum_{j\in I(\tilde{y})}
P\bigl(X\in (g|_{A_j})^{-1}(\,]\tilde{y}-\varepsilon,\tilde{y}+\varepsilon[\,)\bigr)\\
&=&
\sum_{j\in I(\tilde{y})}
P\bigl(X\in \eta\bigl(\,]u_{j,{\min}}(\tilde{y}),u_{j,{\max}(\tilde{y})}[\,
\bigr)\bigr)\\
&\phantom{=}&
u_{j,{\min}}(\tilde{y}):=\min\{u_j(\tilde{y}-\varepsilon),u_j(\tilde{y}+\varepsilon)\},\\
&\phantom{=}&
u_{j,{\max}}(\tilde{y}):=\max\{u_j(\tilde{y}-\varepsilon),u_j(\tilde{y}+\varepsilon)\}\\
&=&
\sum_{j\in I(\tilde{y})}
\int_{\min\{u_j(\tilde{y}-\varepsilon),u_j(\tilde{y}+\varepsilon)\}}
^{\max\{u_j(\tilde{y}-\varepsilon),u_j(\tilde{y}+\varepsilon)\}}
\mu_{X}(\eta(u))\eta'(u)du\\
{(*)}&=&
\sum_{j\in I(\tilde{y})}
\int_{\tilde{y}-\varepsilon}
^{\tilde{y}+\varepsilon}
\mu_{X}(\eta(u(\xi)))\eta'(u(\xi))\left|\frac{du(\xi)}{d\xi}\right| d\xi\\
{(*)}&=&
\sum_{j\in I(\tilde{y})}
\int_{\tilde{y}-\varepsilon}
^{\tilde{y}+\varepsilon}
\mu_{X}((g|_{A_j})^{-1}(\xi))\left|\frac{d}{d\xi}(g|_{A_j})^{-1}(\xi))\right| d\xi.
\end{eqnarray*}
where, in this equation $g^{-1}$ does not strictly stand for the inverse of $g$, which generally speaking is not invertible, but it represents the collection of pre-images of the segment $]\tilde{y}-\varepsilon,\tilde{y}+\varepsilon[$.
In contrast, $(g|_{A_j})^{-1}$ represents the actual local inverse of the segment for the set $A_j$. Moreover, in the last two steps (indicated by $(*)$),
we have used the fact that
$$\eta'(u(y))\left|\frac{du(y)}{dy}\right|= \left|\frac{d}{dy}(g|_{A_j})^{-1}(y)\right|$$
and also that in fact, $\frac{du(y)}{dy}=\pm 1.$

On the other hand, by definition,
$$P(\tilde{y}-\varepsilon< Y \leq \tilde{y}+\varepsilon)
=\int_{\tilde{y}-\varepsilon}^{\tilde{y}+\varepsilon} \mu_Y(\xi)d\xi.$$
Therefore, comparing the two expressions, we obtain that
\begin{equation}\label{eq-1.6}
\mu_Y(y) = \sum_{j\in I(y)}
\mu_{X}((g|_{A_j})^{-1}(y))\left|\frac{d}{dy}(g|_{A_j})^{-1}(y))\right|,
\end{equation}
holds for all $y\in [g_{\min},g_{\max}]\setminus B$. In this
manner, we have reestablished \eqref{eq-1.1}, for
$h_i:=(g|_{A_j})^{-1}$ indicating that the FDF method is
consistent with the general theory for the RVT. Eq. \eqref{eq-1.6}
also holds for the segment end points $y=g(\alpha_0)=g(\alpha)$
(respectively, for $y=g(\alpha_k)=g(\beta)$) provided that they
are not a \textit{critical value}. The formula is not applicable
to the critical values $b_0,\dots,b_{\ell}$ because for these
points the derivative of the inverse function is infinite. Notice
that such a limitation is also present in the general RVT theory
as expressed in \eqref{eq-1.1}. However, in our situation, as the
number of critical values is finite the ensemble of critical
points of $g$ in the interval $[\alpha,\beta]$ is thus a set of
zero measure. Therefore, the weight of the critical values is
negligible in the computation of the cumulative distribution
function $F_Y$, provided that $X$ is a continuous RV or a
discrete/mixed one with a finite number of jumps at points which
are not critical. Finally, in the very special case in which there
are jumps for the RV $X$ at some critical points of $g,$ the
problem will be solved by a direct computation as a limit from the
neighboring compact environment.

\section{An example}\label{section-3}
In order to demonstrate the applicability of the method, we start with a
concrete example. Let us assume that our RDE admits an
analytical solution $y(t):=\varphi(\phi,t)$. For example,
recalling that the RDE Eq. \eqref{eq-1.4} leads to solutions of the form \eqref{eq-1.5},
we assume that our physical system evolves according to:
\begin{equation}\label{eqINTRO1}
\varphi(\phi,t):=k\cdot\phi+A\cdot\cos(\omega t+\phi),
\end{equation}
Where $k>0$ is a given constant, $A>0$ is the \emph{amplitude},
$\omega>0$ is the \emph{angular frequency}, and $\phi$ is the
\emph{initial phase}. Furthermore we
assume that $\phi=X$ is a random number that belongs to
the interval $[\alpha,\beta]$. We suppose that $\phi$ has mass
probability function:
\begin{equation}\label{eqINTRO2}
\mu_X:[\alpha,\beta]\subset\mathbb{R}\to\mathbb{R}.
\end{equation}
In each fixed time instant
$t=t^{*}>0$ the physical system is in the position
$y(t^*)=\varphi(\phi,t^{*})$ and such a position depends
on the particular initial phase $\phi\in[\alpha,\beta]$.
We notice that, from the mathematical point of view,
$y=\varphi(\phi,t)$ is a scalar field
$\varphi:[\alpha,\beta]\times[0,t_{max}]\subset\mathbb{R}^2\to\mathbb{R}$
that associates to each element $(\phi,t)$ a point $y$ in the real physical space.
Our idea is that, in each instant of time $t^{*}\in[0,t_{max}]$,
we can get the probability mass function $\mu_Y:=\psi(y,t^{*})$ for the variable
$y(t^{*})$ by applying the new formal concept introduced above.
We start by replacing the physical map given by Eq. \eqref{eqINTRO1} by the FDF, motivated by
the theoretical results of the previous section.
In Figure \ref{fig1} we report the plot of the functions \eqref{eqINTRO1}, corresponding to the physical solution in the $(\phi,t)$ plane (yelow), together with \eqref{eqINTRO2}, giving the probability distribution of the random variable $X$. On the right-hand side of the figure we have singled out three instants of time $0<t^*_1<t^*_2<t^*_3<t_{max}$, to show that the problem depends on the final time chosen for the analysis.
In each of these times, we plot a function
$g_{t^*_i}:[\alpha,\beta]\subset{\mathbb{R}}\to\mathbb{R}$
($i\in\{1,2,3\}$), such that:
\begin{align*}
g_{t^*_1}(\phi):=&\varphi(\phi,t^*_1),\\
g_{t^*_2}(\phi):=&\varphi(\phi,t^*_2),\\
g_{t^*_3}(\phi):=&\varphi(\phi,t^*_3).
\end{align*}
We notice that each of these three functions maps the interval
$[\alpha,\beta]$ on the $x=\phi$-axis in a non-bijective way to
the $y$-axis. As explained, we proceed to the separation of
$g_{t^*_i}(\phi)$ into all the monotone pieces as we vary $\phi$
from $\alpha$ to $\beta$, as we described. Second, we identify the
increasing from the decreasing pieces to construct
$\hat{g}_{t^*_i}(\phi)$. As each of these pieces are invertible,
we can calculate the transformation of $\mu_X$ into the  {\em
unfolded intermediate} $\tilde{\mu}^t_Y$, with $t=t^*_1, t^*_2,
t^*_3$. Then, following the indicated procedure, we construct the
final folded $\mu_Y$. The key element is to realizing that the
{\em folding} process will introduce the overlap of the different
sections. In Figure \ref{fig1} we show precisely the overlapping
pieces that will give rise to the final $\mu_Y$.

In more detail, for each $t=t^*_1, t^*_2, t^*_3$
we have a different vector $\Lambda=\Lambda_t$ which collects all the
needed set of information to construct the functions $\hat{g},$ $\eta$ and the
associated layers. For $t=t^*_1$, in Figure \ref{fig1}
the first arrow is a blue arrow that corresponds to the first
decreasing part of $g_{t^*_1}$. The second is a red arrow that corresponds
the increasing part of $g_{t^*_1}$ and the last one is another blue arrow that
corresponds to the also decreasing part of $g_{t^*_1}$.
We proceeded similarly for $g_{t^*_2}$ and $g_{t^*_3}$.
Observe that, as the time varies, the corresponding $\Lambda_t$ changes and, therefore,
the size and direction of the {\em arrows}, as well.

\begin{figure}[ht]
\begin{center}
\tdplotsetmaincoords{40}{20} 
\begin{tikzpicture}[tdplot_main_coords,scale=0.7]

   \draw[thick,->] (0,0,0)node[left]{$t_{min}$} -- (0,7,0) node[anchor=north west]{$t$};

   \pgfmathsetmacro{\o}{1}
      \pgfmathsetmacro{\A}{1}

 \draw[dashed](1,0,0)node[below]{$\alpha$}--(1,5.2,0);
  \draw[dashed](6.2,0,0)node[below]{$\beta$}--(6.2,5.2,0);

 \draw[dashed](0,5.2,0)node[above left]{$t_{max}$}--(6.2,5.2,0);
 \draw[dashed](6.2,5.2,0)--(6.2,5.2,{6.2*0.41+\A*cos(deg(1.3*(\o*5.2+6.2)))})node[right,color=brown]{$\varphi(\phi,t)$};

\draw[dashed](6.2,4,0)--(6.2,4,{6.2*0.41+\A*cos(deg(1.3*(\o*4+6.2)))});
  \draw[dashed](1,4,0)--(6.2,4,0);

 \foreach \a in {1,1.2,1.4,1.6,1.8,2,2.2,2.4,2.6,2.8,3,3.2,3.4,3.6,3.8,4,4.2,4.4,4.6,4.8,5,5.2,5.4,5.6,5.8,6}
 {
  \foreach \t in {4,4.2,4.4,4.6,4.8,5}
  {
  \path[fill=yellow,opacity=0.89] (\a,\t,{\a*0.41+\A*cos(deg(1.3*(\o*\t+\a)))})--(\a,\t+0.2,{\a*0.41+\A*cos(deg(1.3*(\o*(\t+0.2)+\a)))})--(\a+0.2,\t+0.2,{(\a+0.2)*0.41+\A*cos(deg(1.3*(\o*(\t+0.2)+\a+0.2)))})--(\a+0.2,\t,{(\a+0.2)*0.41+\A*cos(deg(1.3*(\o*\t+\a+0.2)))})--(\a,\t,{\a*0.41+\A*cos(deg(1.3*(\o*\t+\a)))});
   \draw[color=brown] (\a,\t,{\a*0.41+\A*cos(deg(1.3*(\o*\t+\a)))})--(\a,\t+0.2,{\a*0.41+\A*cos(deg(1.3*(\o*(\t+0.2)+\a)))});
  }
  }

    \foreach \t in {4,4.2,4.4,4.6,4.8,5,5.2}
 {
 \foreach \a in {1,1.2,1.4,1.6,1.8,2,2.2,2.4,2.6,2.8,3,3.2,3.4,3.6,3.8,4,4.2,4.4,4.6,4.8,5,5.2,5.4,5.6,5.8,6}
  {
  \draw[color=brown] (\a,\t,{\a*0.41+\A*cos(deg(1.3*(\o*\t+\a)))})--(\a+0.2,\t,{(\a+0.2)*0.41+\A*cos(deg(1.3*(\o*\t+\a+0.2)))});
  }
  }

    \foreach \t in {4,4.2,4.4,4.6,4.8,5}
  {
  \draw[color=brown] (6.2,\t,{6.2*0.41+\A*cos(deg(1.3*(\o*(\t)+6.2)))})--(6.2,\t+0.2,{6.2*0.41+\A*cos(deg(1.3*(\o*(\t+0.2)+6.2))});
  }

   \path[fill=cyan,opacity=0.21](0,4,-1)--(13,4,-1)--(13,4,3.5)--(0,4,3.5)--(0,4,-1);
   \draw[dashed,color=gray](11.5,4,-1)--(11.5,4,3.5);
   \draw[dashed,color=gray](12.25,4,-1)--(12.25,4,3.5);
    \draw[dashed,color=gray](13,4,-1)--(13,4,3.5);
        \draw[color=blue,->, thick](11.5,4,{0.41+\A*cos(deg(1.3*(\o*4+1)))})--(11.5,4,{3*0.41+\A*cos(deg(1.3*(\o*4+3)))});
        \draw[color=red,->, thick](12.25,4,{3*0.41+\A*cos(deg(1.3*(\o*4+3)))})--(12.25,4,{5.8*0.41+\A*cos(deg(1.3*(\o*4+5.8)))});
        \draw[color=blue,->, thick](13,4,{5.8*0.41+\A*cos(deg(1.3*(\o*4+5.8)))})--(13,4,{6.2*0.41+\A*cos(deg(1.3*(\o*4+6.2)))});

   \node[above,color=black] at (6.2,4,{6.2*0.41+\A*cos(deg(1.3*(\o*4+6.2)))}){$g_{t^{*}_3}(\phi)$};

  \draw[dashed,color=gray](0,4,0)node[above left]{$t^{*}_3$}--(3.2,4,0);

 \foreach \a in {1,1.2,1.4,1.6,1.8,2,2.2,2.4,2.6,2.8,3,3.2,3.4,3.6,3.8,4,4.2,4.4,4.6,4.8,5,5.2,5.4,5.6,5.8,6}
 {
  \foreach \t in {2.4,2.6,2.8,3,3.2,3.4,3.6,3.8}
  {

  \path[fill=yellow,opacity=0.89] (\a,\t,{\a*0.41+\A*cos(deg(1.3*(\o*\t+\a)))})--(\a,\t+0.2,{\a*0.41+\A*cos(deg(1.3*(\o*(\t+0.2)+\a)))})--(\a+0.2,\t+0.2,{(\a+0.2)*0.41+\A*cos(deg(1.3*(\o*(\t+0.2)+\a+0.2)))})--(\a+0.2,\t,{(\a+0.2)*0.41+\A*cos(deg(1.3*(\o*\t+\a+0.2)))})--(\a,\t,{\a*0.41+\A*cos(deg(1.3*(\o*\t+\a)))});
   \draw[color=brown] (\a,\t,{\a*0.41+\A*cos(deg(1.3*(\o*\t+\a)))})--(\a,\t+0.2,{\a*0.41+\A*cos(deg(1.3*(\o*(\t+0.2)+\a)))});
  }
  }

 \foreach \t in {2.4,2.6,2.8,3,3.2,3.4,3.6,3.8}
 {
 \foreach \a in {1,1.2,1.4,1.6,1.8,2,2.2,2.4,2.6,2.8,3,3.2,3.4,3.6,3.8,4,4.2,4.4,4.6,4.8,5,5.2,5.4,5.6,5.8,6}
  {
  \draw[color=brown] (\a,\t,{\a*0.41+\A*cos(deg(1.3*(\o*\t+\a)))})--(\a+0.2,\t,{(\a+0.2)*0.41+\A*cos(deg(1.3*(\o*\t+\a+0.2)))});
  }
  }

  \foreach \t in {4}
 {
 \foreach \a in {1,1.2,1.4,1.6,1.8,2,2.2,2.4,2.6,2.8}
  {
  \draw[color=blue,thick] (\a,\t,{\a*0.41+\A*cos(deg(1.3*(\o*\t+\a)))})--(\a+0.2,\t,{(\a+0.2)*0.41+\A*cos(deg(1.3*(\o*\t+\a+0.2)))});
  }
  \foreach \a in {3,3.2,3.4,3.6,3.8,4,4.2,4.4,4.6,4.8,5,5.2,5.4,5.6}
  {
  \draw[color=red,thick] (\a,\t,{\a*0.41+\A*cos(deg(1.3*(\o*\t+\a)))})--(\a+0.2,\t,{(\a+0.2)*0.41+\A*cos(deg(1.3*(\o*\t+\a+0.2)))});
  }
  \foreach \a in {5.8,6}
  {
  \draw[color=blue,thick] (\a,\t,{\a*0.41+\A*cos(deg(1.3*(\o*\t+\a)))})--(\a+0.2,\t,{(\a+0.2)*0.41+\A*cos(deg(1.3*(\o*\t+\a+0.2)))});
  }
  }

  \foreach \t in {2.4,2.6,2.6,2.8,3,3.2,3.4,3.6,3.8}
  {
  \draw[color=brown] (6.2,\t,{6.2*0.41+\A*cos(deg(1.3*(\o*(\t)+6.2)))})--(6.2,\t+0.2,{6.2*0.41+\A*cos(deg(1.3*(\o*(\t+0.2)+6.2))});
  }

  \path[fill=cyan,opacity=0.21](0,2.4,-1)--(10.75,2.4,-1)--(10.75,2.4,2.7)--(0,2.4,2.7)--(0,2.4,-1);
  \draw[dashed,color=gray](9.25,2.4,-1)--(9.25,2.4,2.7);
   \draw[dashed,color=gray](10,2.4,-1)--(10,2.4,2.7);
    \draw[dashed,color=gray](10.75,2.4,-1)--(10.75,2.4,2.7);
         \draw[color=red,->,thick](9.25,2.4,{0.41+\A*cos(deg(1.3*(\o*(2.4)+1)))})--(9.25,2.4,{2.6*0.41+\A*cos(deg(1.3*(\o*(2.4)+2.6)))});
        \draw[color=blue,->,thick](10,2.4,{2.6*0.41+\A*cos(deg(1.3*(\o*(2.4)+2.6)))})--(10,2.4,{4.6*0.41+\A*cos(deg(1.3*(\o*(2.4)+4.6)))});
        \draw[color=red,->,thick](10.75,2.4,{4.6*0.41+\A*cos(deg(1.3*(\o*(2.4)+4.6)))})--(10.75,2.4,{6.2*0.41+\A*cos(deg(1.3*(\o*(2.4)+6.2)))});


   \draw[dashed,color=gray](0,2.4,0)node[above left]{$t^{*}_2$}--(6.2,2.4,0);
 \draw[dashed,color=gray](6.2,2.4,0)--(6.2,2.4,{6.2*0.41+\A*cos(deg(1.3*(\o*2.4+6.2)))})node[above,color=black]{$g_{t^{*}_2}(\phi)$};

 \foreach \a in {1,1.2,1.4,1.6,1.8,2,2.2,2.4,2.6,2.8,3,3.2,3.4,3.6,3.8,4,4.2,4.4,4.6,4.8,5,5.2,5.4,5.6,5.8,6}
 {
  \foreach \t in {0.4,0.6,0.8,1,1.2,1.4,1.6,1.8,2,2.2}
  {

  \path[fill=yellow,opacity=0.89] (\a,\t,{\a*0.41+\A*cos(deg(1.3*(\o*\t+\a)))})--(\a,\t+0.2,{\a*0.41+\A*cos(deg(1.3*(\o*(\t+0.2)+\a)))})--(\a+0.2,\t+0.2,{(\a+0.2)*0.41+\A*cos(deg(1.3*(\o*(\t+0.2)+\a+0.2)))})--(\a+0.2,\t,{(\a+0.2)*0.41+\A*cos(deg(1.3*(\o*\t+\a+0.2)))})--(\a,\t,{\a*0.41+\A*cos(deg(1.3*(\o*\t+\a)))});
   \draw[color=brown] (\a,\t,{\a*0.41+\A*cos(deg(1.3*(\o*\t+\a)))})--(\a,\t+0.2,{\a*0.41+\A*cos(deg(1.3*(\o*(\t+0.2)+\a)))});
  }
  }

   \foreach \t in {0.4,0.6,0.8,1,1.2,1.4,1.6,1.8,2,2.2}
 {
 \foreach \a in {1,1.2,1.4,1.6,1.8,2,2.2,2.4,2.6,2.8,3,3.2,3.4,3.6,3.8,4,4.2,4.4,4.6,4.8,5,5.2,5.4,5.6,5.8,6}
  {
  \draw[color=brown] (\a,\t,{\a*0.41+\A*cos(deg(1.3*(\o*\t+\a)))})--(\a+0.2,\t,{(\a+0.2)*0.41+\A*cos(deg(1.3*(\o*\t+\a+0.2)))});
  }
  }

   \foreach \t in {2.4}
 {
 \foreach \a in {1,1.2,1.4,1.6,1.8,2,2.2,2.4}
  {
  \draw[color=red,thick] (\a,\t,{\a*0.41+\A*cos(deg(1.3*(\o*\t+\a)))})--(\a+0.2,\t,{(\a+0.2)*0.41+\A*cos(deg(1.3*(\o*\t+\a+0.2)))});
  }
  \foreach \a in {2.6,2.8,3,3.2,3.4,3.6,3.8,4,4.2,4.4}
  {
  \draw[color=blue,thick] (\a,\t,{\a*0.41+\A*cos(deg(1.3*(\o*\t+\a)))})--(\a+0.2,\t,{(\a+0.2)*0.41+\A*cos(deg(1.3*(\o*\t+\a+0.2)))});
  }
  \foreach \a in {4.6,4.8,5,5.2,5.4,5.6,5.8,6}
  {
  \draw[color=red,thick] (\a,\t,{\a*0.41+\A*cos(deg(1.3*(\o*\t+\a)))})--(\a+0.2,\t,{(\a+0.2)*0.41+\A*cos(deg(1.3*(\o*\t+\a+0.2)))});
  }
  }

  \foreach \t in {0.4,0.6,0.8,1,1.2,1.4,1.6,1.8,2,2.2}
  {
  \draw[color=brown] (6.2,\t,{6.2*0.41+\A*cos(deg(1.3*(\o*(\t)+6.2)))})--(6.2,\t+0.2,{6.2*0.41+\A*cos(deg(1.3*(\o*(\t+0.2)+6.2))});
  }

     \path[fill=cyan,opacity=0.21](0,0.4,-1)--(8.5,0.4,-1)--(8.5,0.4,2.9)--(0,0.4,2.9)--(0,0.4,-1);
  \draw[dashed,color=gray](7,0.4,-1)--(7,0.4,2.9);
   \draw[dashed,color=gray](7.75,0.4,-1)--(7.75,0.4,2.9);
    \draw[dashed,color=gray](8.5,0.4,-1)--(8.5,0.4,2.9);
    \draw[color=blue,->,thick](7,0.4,{0.41+\A*cos(deg(1.3*(\o*(0.4)+1)))})--(7,0.4,{1.8*0.41+\A*cos(deg(1.3*(\o*(0.4)+1.8)))});
        \draw[color=red,->,thick](7.75,0.4,{1.8*0.41+\A*cos(deg(1.3*(\o*(0.4)+1.8)))})--(7.75,0.4,{4.6*0.41+\A*cos(deg(1.3*(\o*(0.4)+4.6)))});
        \draw[color=blue,->,thick](8.5,0.4,{4.6*0.41+\A*cos(deg(1.3*(\o*(0.4)+4.6)))})--(8.5,0.4,{6.2*0.41+\A*cos(deg(1.3*(\o*(0.4)+6.2)))});

 \draw[dashed,color=gray](0,0.4,0)node[above]{$t^{*}_1$}--(6.2,0.4,0);
 \draw[dashed,color=gray](6.2,0.4,0)--(6.2,0.4,{6.2*0.41+\A*cos(deg(1.3*(\o*0.4+6.2)))})node[above right,color=black]{$g_{t^{*}_1}(\phi)$};

 \foreach \a in {1,1.2,1.4,1.6,1.8,2,2.2,2.4,2.6,2.8,3,3.2,3.4,3.6,3.8,4,4.2,4.4,4.6,4.8,5,5.2,5.4,5.6,5.8,6}
 {
  \foreach \t in {0,0.2}
  {

  \path[fill=yellow,opacity=0.89] (\a,\t,{\a*0.41+\A*cos(deg(1.3*(\o*\t+\a)))})--(\a,\t+0.2,{\a*0.41+\A*cos(deg(1.3*(\o*(\t+0.2)+\a)))})--(\a+0.2,\t+0.2,{(\a+0.2)*0.41+\A*cos(deg(1.3*(\o*(\t+0.2)+\a+0.2)))})--(\a+0.2,\t,{(\a+0.2)*0.41+\A*cos(deg(1.3*(\o*\t+\a+0.2)))})--(\a,\t,{\a*0.41+\A*cos(deg(1.3*(\o*\t+\a)))});
   \draw[color=brown] (\a,\t,{\a*0.41+\A*cos(deg(1.3*(\o*\t+\a)))})--(\a,\t+0.2,{\a*0.41+\A*cos(deg(1.3*(\o*(\t+0.2)+\a)))});
  }
  }

  \foreach \t in {0,0.2}
 {
 \foreach \a in {1,1.2,1.4,1.6,1.8,2,2.2,2.4,2.6,2.8,3,3.2,3.4,3.6,3.8,4,4.2,4.4,4.6,4.8,5,5.2,5.4,5.6,5.8,6}
  {
  \draw[color=brown] (\a,\t,{\a*0.41+\A*cos(deg(1.3*(\o*\t+\a)))})--(\a+0.2,\t,{(\a+0.2)*0.41+\A*cos(deg(1.3*(\o*\t+\a+0.2)))});
  }
  }

  \foreach \t in {0.4}
 {
 \foreach \a in {1,1.2,1.4,1.6}
  {
  \draw[color=blue,thick] (\a,\t,{\a*0.41+\A*cos(deg(1.3*(\o*\t+\a)))})--(\a+0.2,\t,{(\a+0.2)*0.41+\A*cos(deg(1.3*(\o*\t+\a+0.2)))});
  }

 \foreach \a in {1.8,2,2.2,2.4,2.6,2.8,3,3.2,3.4,3.6,3.8,4,4.2,4.4}
  {
  \draw[color=red,thick] (\a,\t,{\a*0.41+\A*cos(deg(1.3*(\o*\t+\a)))})--(\a+0.2,\t,{(\a+0.2)*0.41+\A*cos(deg(1.3*(\o*\t+\a+0.2)))});
  }
   \foreach \a in {4.6,4.8,5,5.2,5.4,5.6,5.8,6}
  {
  \draw[color=blue,thick] (\a,\t,{\a*0.41+\A*cos(deg(1.3*(\o*\t+\a)))})--(\a+0.2,\t,{(\a+0.2)*0.41+\A*cos(deg(1.3*(\o*\t+\a+0.2)))});
  }
  }

  \foreach \t in {0,0.2}
  {
  \draw[color=brown] (6.2,\t,{6.2*0.41+\A*cos(deg(1.3*(\o*(\t)+6.2)))})--(6.2,\t+0.2,{6.2*0.41+\A*cos(deg(1.3*(\o*(\t+0.2)+6.2))});
  }

       \draw[thick,->] (0,0,0) -- (6.5,0,0) node[below]{$\phi$};

    \draw[thick,->] (0,0,0) -- (0,0,4) node[anchor=south]{$y$};


        \draw[thick,->] (0,-2,0) -- (6.5,0-2,0) node[anchor=south]{$\phi$};

%

 \pgfmathsetmacro{\pu}{2}
   \pgfmathsetmacro{\pd}{-3}
      \pgfmathsetmacro{\pt}{0.55}

      \pgfmathsetmacro{\fa}{2}



     \foreach \a in {1,1.2,1.4,1.6,1.8,2,2.2,2.4,2.6,2.8,3,3.2,3.4,3.6,3.8,4,4.2,4.4,4.6,4.8,5,5.2,5.4,5.6,5.8,6}
 {
  \path[fill=green,opacity=0.41] (\a,-2,0)--(\a,{\pd-\pt*cos(deg(\a*\pu+\fa))},0)--(\a+0.2,{\pd-\pt*cos(deg((\a+0.2)*\pu+\fa))},0)--(\a+0.2,-2,0)--(\a,-2,0);
   \draw[color=gray] (\a,{\pd-\pt*cos(deg(\a*\pu+\fa))},0)--(\a+0.2,{\pd-\pt*cos(deg((\a+0.2)*\pu+\fa))},0);
  }

  \node[color=gray,right] at (6.2,{\pd-\pt*cos(deg(6.2*\pu))},0){$\mu_X(\phi)$};


\pgfmathsetmacro{\t}{0.4}

  \foreach \a in {1,1.1,1.2,1.3,1.4,1.5,1.6,1.7,1.8}
 {
 \path[fill=green,opacity=0.41](7,\t,{\a*0.41+\A*cos(deg(1.3*(\o*\t+\a)))})--(7,{\t-1-\pt*cos(deg(\a*\pu+\fa))},{\a*0.41+\A*cos(deg(1.3*(\o*\t+\a)))})--(7,{\t-1-\pt*cos(deg((\a+0.1)*\pu+\fa))},{(\a+0.1)*0.41+\A*cos(deg(1.3*(\o*\t+\a+0.1)))})--(7,\t,{(\a+0.1)*0.41+\A*cos(deg(1.3*(\o*\t+\a+0.1)))})--(7,\t,{\a*0.41+\A*cos(deg(1.3*(\o*\t+\a)))});
  \draw[color=purple,thick](7,{\t-1-\pt*cos(deg(\a*\pu+\fa))},{\a*0.41+\A*cos(deg(1.3*(\o*\t+\a)))})--(7,{\t-1-\pt*cos(deg((\a+0.1)*\pu+\fa))},{(\a+0.1)*0.41+\A*cos(deg(1.3*(\o*\t+\a+0.1)))});
  }

  \foreach \a in {1.8,1.9,2,2.1,2.2,2.3,2.4,2.5,2.6,2.7,2.8,2.9,3,3.1,3.2,3.3,3.4,3.5,3.6,3.7,3.8,3.9,4,4.1,4.2,4.3,4.4}
 {
 \path[fill=green,opacity=0.41](7.75,\t,{\a*0.41+\A*cos(deg(1.3*(\o*\t+\a)))})--(7.75,{\t-1-\pt*cos(deg(\a*\pu+\fa))},{\a*0.41+\A*cos(deg(1.3*(\o*\t+\a)))})--(7.75,{\t-1-\pt*cos(deg((\a+0.1)*\pu+\fa))},{(\a+0.1)*0.41+\A*cos(deg(1.3*(\o*\t+\a+0.1)))})--(7.75,\t,{(\a+0.1)*0.41+\A*cos(deg(1.3*(\o*\t+\a+0.1)))})--(7.75,\t,{\a*0.41+\A*cos(deg(1.3*(\o*\t+\a)))});
  \draw[color=purple,thick](7.75,{\t-1-\pt*cos(deg(\a*\pu+\fa))},{\a*0.41+\A*cos(deg(1.3*(\o*\t+\a)))})--(7.75,{\t-1-\pt*cos(deg((\a+0.1)*\pu+\fa))},{(\a+0.1)*0.41+\A*cos(deg(1.3*(\o*\t+\a+0.1)))});
  }
  \foreach \a in {4.4,4.5,4.6,4.7,4.8,4.9,5,5.1,5.2,5.3,5.4,5.5,5.6,5.7,5.8,5.9,6}
 {
 \path[fill=green,opacity=0.41](8.5,\t,{\a*0.41+\A*cos(deg(1.3*(\o*\t+\a)))})--(8.5,{\t-1-\pt*cos(deg(\a*\pu+\fa))},{\a*0.41+\A*cos(deg(1.3*(\o*\t+\a)))})--(8.5,{\t-1-\pt*cos(deg((\a+0.1)*\pu+\fa))},{(\a+0.1)*0.41+\A*cos(deg(1.3*(\o*\t+\a+0.1)))})--(8.5,\t,{(\a+0.1)*0.41+\A*cos(deg(1.3*(\o*\t+\a+0.1)))})--(8.5,\t,{\a*0.41+\A*cos(deg(1.3*(\o*\t+\a)))});
  \draw[color=purple,thick](8.5,{\t-1-\pt*cos(deg(\a*\pu+\fa))},{\a*0.41+\A*cos(deg(1.3*(\o*\t+\a)))})--(8.5,{\t-1-\pt*cos(deg((\a+0.1)*\pu+\fa))},{(\a+0.1)*0.41+\A*cos(deg(1.3*(\o*\t+\a+0.1)))});
  }


\pgfmathsetmacro{\t}{2.4}

  \foreach \a in {1,1.1,1.2,1.3,1.4,1.5,1.6,1.7,1.8,1.9,2,2.1,2.2,2.3,2.4}
 {
 \path[fill=green,opacity=0.41](9.25,\t,{\a*0.41+\A*cos(deg(1.3*(\o*\t+\a)))})--(9.25,{\t-1-\pt*cos(deg(\a*\pu+\fa))},{\a*0.41+\A*cos(deg(1.3*(\o*\t+\a)))})--(9.25,{\t-1-\pt*cos(deg((\a+0.1)*\pu+\fa))},{(\a+0.1)*0.41+\A*cos(deg(1.3*(\o*\t+\a+0.1)))})--(9.25,\t,{(\a+0.1)*0.41+\A*cos(deg(1.3*(\o*\t+\a+0.1)))})--(9.25,\t,{\a*0.41+\A*cos(deg(1.3*(\o*\t+\a)))});
  \draw[color=purple,thick](9.25,{\t-1-\pt*cos(deg(\a*\pu+\fa))},{\a*0.41+\A*cos(deg(1.3*(\o*\t+\a)))})--(9.25,{\t-1-\pt*cos(deg((\a+0.1)*\pu+\fa))},{(\a+0.1)*0.41+\A*cos(deg(1.3*(\o*\t+\a+0.1)))});
  }

  \foreach \a in {2.4,2.5,2.6,2.7,2.8,2.9,3,3.1,3.2,3.3,3.4,3.5,3.6,3.7,3.8,3.9,4,4.1,4.2,4.3,4.4}
 {
 \path[fill=green,opacity=0.41](10,\t,{\a*0.41+\A*cos(deg(1.3*(\o*\t+\a)))})--(10,{\t-1-\pt*cos(deg(\a*\pu+\fa))},{\a*0.41+\A*cos(deg(1.3*(\o*\t+\a)))})--(10,{\t-1-\pt*cos(deg((\a+0.1)*\pu+\fa))},{(\a+0.1)*0.41+\A*cos(deg(1.3*(\o*\t+\a+0.1)))})--(10,\t,{(\a+0.1)*0.41+\A*cos(deg(1.3*(\o*\t+\a+0.1)))})--(10,\t,{\a*0.41+\A*cos(deg(1.3*(\o*\t+\a)))});
  \draw[color=purple,thick](10,{\t-1-\pt*cos(deg(\a*\pu+\fa))},{\a*0.41+\A*cos(deg(1.3*(\o*\t+\a)))})--(10,{\t-1-\pt*cos(deg((\a+0.1)*\pu+\fa))},{(\a+0.1)*0.41+\A*cos(deg(1.3*(\o*\t+\a+0.1)))});
  }
  \foreach \a in {4.4,4.5,4.6,4.7,4.8,4.9,5,5.1,5.2,5.3,5.4,5.5,5.6,5.7,5.8,5.9,6}
 {
 \path[fill=green,opacity=0.41](10.75,\t,{\a*0.41+\A*cos(deg(1.3*(\o*\t+\a)))})--(10.75,{\t-1-\pt*cos(deg(\a*\pu+\fa))},{\a*0.41+\A*cos(deg(1.3*(\o*\t+\a)))})--(10.75,{\t-1-\pt*cos(deg((\a+0.1)*\pu+\fa))},{(\a+0.1)*0.41+\A*cos(deg(1.3*(\o*\t+\a+0.1)))})--(10.75,\t,{(\a+0.1)*0.41+\A*cos(deg(1.3*(\o*\t+\a+0.1)))})--(10.75,\t,{\a*0.41+\A*cos(deg(1.3*(\o*\t+\a)))});
  \draw[color=purple,thick](10.75,{\t-1-\pt*cos(deg(\a*\pu+\fa))},{\a*0.41+\A*cos(deg(1.3*(\o*\t+\a)))})--(10.75,{\t-1-\pt*cos(deg((\a+0.1)*\pu+\fa))},{(\a+0.1)*0.41+\A*cos(deg(1.3*(\o*\t+\a+0.1)))});
  }


\pgfmathsetmacro{\t}{4}

  \foreach \a in {1,1.1,1.2,1.3,1.4,1.5,1.6,1.7,1.8,1.9,2,2.1,2.2,2.3,2.4,2.5,2.6,2.7,2.8}
 {
 \path[fill=green,opacity=0.41](11.5,\t,{\a*0.41+\A*cos(deg(1.3*(\o*\t+\a)))})--(11.5,{\t-1-\pt*cos(deg(\a*\pu+\fa))},{\a*0.41+\A*cos(deg(1.3*(\o*\t+\a)))})--(11.5,{\t-1-\pt*cos(deg((\a+0.1)*\pu+\fa))},{(\a+0.1)*0.41+\A*cos(deg(1.3*(\o*\t+\a+0.1)))})--(11.5,\t,{(\a+0.1)*0.41+\A*cos(deg(1.3*(\o*\t+\a+0.1)))})--(11.5,\t,{\a*0.41+\A*cos(deg(1.3*(\o*\t+\a)))});
  \draw[color=purple,thick](11.5,{\t-1-\pt*cos(deg(\a*\pu+\fa))},{\a*0.41+\A*cos(deg(1.3*(\o*\t+\a)))})--(11.5,{\t-1-\pt*cos(deg((\a+0.1)*\pu+\fa))},{(\a+0.1)*0.41+\A*cos(deg(1.3*(\o*\t+\a+0.1)))});
  }

  \foreach \a in {2.8,2.9,3,3.1,3.2,3.3,3.4,3.5,3.6,3.7,3.8,3.9,4,4.1,4.2,4.3,4.4,4.5,4.6,4.7,4.8,4.9,5,5.1,5.2,5.3,5.4,5.5,5.6}
 {
 \path[fill=green,opacity=0.41](12.25,\t,{\a*0.41+\A*cos(deg(1.3*(\o*\t+\a)))})--(12.25,{\t-1-\pt*cos(deg(\a*\pu+\fa))},{\a*0.41+\A*cos(deg(1.3*(\o*\t+\a)))})--(12.25,{\t-1-\pt*cos(deg((\a+0.1)*\pu+\fa))},{(\a+0.1)*0.41+\A*cos(deg(1.3*(\o*\t+\a+0.1)))})--(12.25,\t,{(\a+0.1)*0.41+\A*cos(deg(1.3*(\o*\t+\a+0.1)))})--(12.25,\t,{\a*0.41+\A*cos(deg(1.3*(\o*\t+\a)))});
  \draw[color=purple,thick](12.25,{\t-1-\pt*cos(deg(\a*\pu+\fa))},{\a*0.41+\A*cos(deg(1.3*(\o*\t+\a)))})--(12.25,{\t-1-\pt*cos(deg((\a+0.1)*\pu+\fa))},{(\a+0.1)*0.41+\A*cos(deg(1.3*(\o*\t+\a+0.1)))});
  }
  \foreach \a in {5.6,5.7,5.8,5.9,6}
 {
 \path[fill=green,opacity=0.41](13,\t,{\a*0.41+\A*cos(deg(1.3*(\o*\t+\a)))})--(13,{\t-1-\pt*cos(deg(\a*\pu+\fa))},{\a*0.41+\A*cos(deg(1.3*(\o*\t+\a)))})--(13,{\t-1-\pt*cos(deg((\a+0.1)*\pu+\fa))},{(\a+0.1)*0.41+\A*cos(deg(1.3*(\o*\t+\a+0.1)))})--(13,\t,{(\a+0.1)*0.41+\A*cos(deg(1.3*(\o*\t+\a+0.1)))})--(13,\t,{\a*0.41+\A*cos(deg(1.3*(\o*\t+\a)))});
  \draw[color=purple,thick](13,{\t-1-\pt*cos(deg(\a*\pu+\fa))},{\a*0.41+\A*cos(deg(1.3*(\o*\t+\a)))})--(13,{\t-1-\pt*cos(deg((\a+0.1)*\pu+\fa))},{(\a+0.1)*0.41+\A*cos(deg(1.3*(\o*\t+\a+0.1)))});
  }
    \node[color=purple] at (7.75,0.4,-2){$\tilde{\mu}_Y^{t^*_1}$};
     \node[color=purple] at (10,2.4,-2){$\tilde{\mu}_Y^{t^*_2}$};
      \node[color=purple] at (12.25,4,-2){$\tilde{\mu}_Y^{t^*_3}$};
\end{tikzpicture}
\end{center}
\caption{The physical system
$\varphi(\phi,t)=k\cdot\varphi+A\cos(\omega t+\phi)$ and the
\emph{folding-domain function} approach.} \label{fig1}
\end{figure}
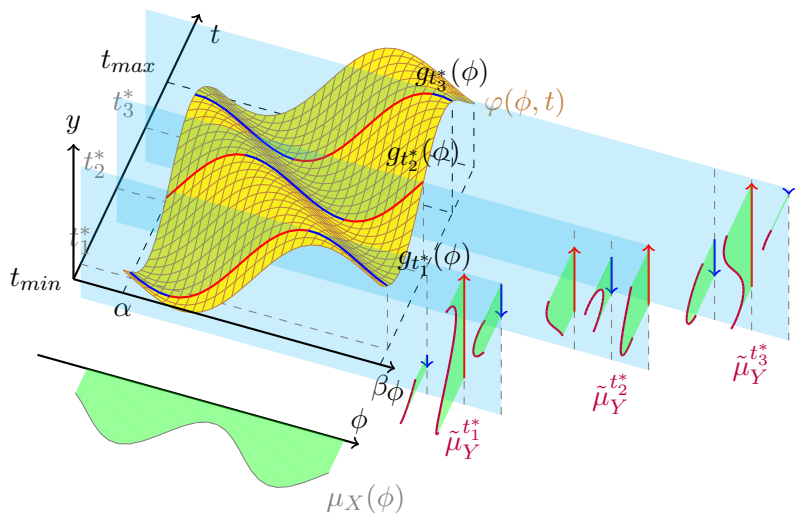

Following Section \ref{sub-1.1}, the procedure starts by replacing the
description of the physical system given by $\varphi(\phi,t),$
for each $t\in[0,t_{max}]$, with:
\begin{enumerate}
\item the point $\chi_0(t):=\varphi(\alpha,t)$,
corresponding to $g(\alpha_0)=g(\alpha)=g_t(\alpha),$
\item the vector $\Lambda_t$ that contains the arrows,
\item the function $\eta_t$,
\item according to \eqref{form-1}, the functions $u_{i,t}(y),$
(with $i\in\{\rho_1,\dots,\rho_{p_{t}}(y)\}=I_t(y)$),
which associate all the values $u_{\rho_1},\dots,u_{\rho_{p_{t}}}$ to each point
$y$ in the physical space,
such that $\eta_t(u_i)=\phi_i$ and $\varphi(\phi_i,t)=y$.
\end{enumerate}

We further notice that the bijection $\eta_{t^*_i}$ allows us to
associate a probability weight $\tilde{\mu}_Y$ to each point on
the three arrows via the PDF $\mu_X$: indeed, in Figure
\ref{fig1}, at the right-hand side, we show the graphs of three
groups (for $t=t^*_1, t^*_2, t^*_3$) of three functions (over
three arrows). We recall that we have denoted these three groups
of functions (the so-called unfolded intermediate) as
$\tilde{\mu}^t_Y$, with $t=t^*_1, t^*_2, t^*_3$. These three plots
have been produced as a composition of $\mu_X$ with the inversions
of $g_{t^*_i}$ in each of the three intervals of monotonicity in
which the domain of $\mu_X$ is split. This can be obtained by
choosing a point $u\in[0,S_{\Lambda_{t^*_i}}]$ and then by
computing $\mu_X\bigl[\eta_{t^*_i}(u)\bigr]$. We stress that these
composite graphs are represented just for the sake of intuitively
describing the underlying idea. The real graphs of $\mu_Y$ should
take into account a scaling factor given by the inverse derivative
and should be overlapped by the unfolding, summing up the
different layers, according to \eqref{eq-1.6}.

The graphs of functions $\eta_t$ are plotted in Figure \ref{fig2}.
In both Figures
\ref{fig1} and \ref{fig2} we consider three cases $t=t^*_1, t^*_2, t^*_3$ for which
the same idea can be applied.

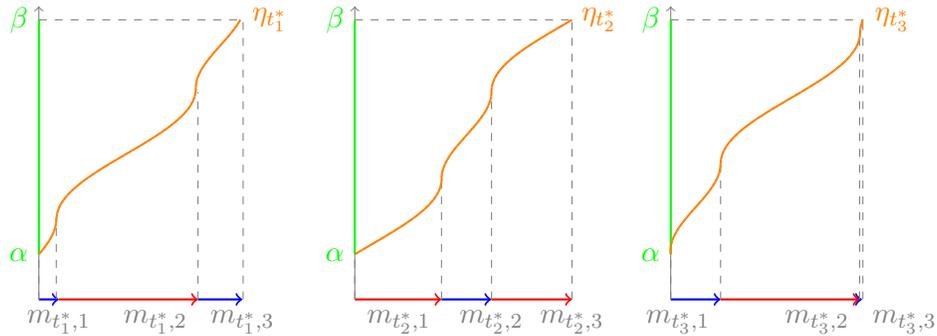
\begin{figure}[ht]
\begin{center}
\begin{tikzpicture}[scale=0.6]

   \pgfmathsetmacro{\o}{1}
      \pgfmathsetmacro{\A}{1}

   \pgfmathsetmacro{\gu}{14}
   \pgfmathsetmacro{\gd}{7}

\pgfmathsetmacro{\DelU}{0.41+\A*cos(deg(1.3*(\o*4+1)))-(3*0.41+\A*cos(deg(1.3*(\o*4+3))))}
   \pgfmathsetmacro{\DelD}{5.8*0.41+\A*cos(deg(1.3*(\o*4+5.8)))-(3*0.41+\A*cos(deg(1.3*(\o*4+3))))}
\pgfmathsetmacro{\DelT}{5.8*0.41+\A*cos(deg(1.3*(\o*4+5.8)))-(6.2*0.41+\A*cos(deg(1.3*(\o*4+6.2))))}

 \draw[color=blue,->,thick](\gu,0)--({\DelU+\gu},0)node[below left,color=gray]{$m_{t^*_3,1}$};
        \draw[color=red,->,thick]({\DelU+\gu},0)--({\DelU+\DelD+\gu},0)node[below left,color=gray]{$m_{t^*_3,2}$};
        \draw[color=blue,->,thick]({\DelU+\DelD+\gu},0)--({\DelU+\DelD+\DelT+\gu},0)node[below right,color=gray]{$m_{t^*_3,3}$};

\node[right,color=orange] at (4.5,6.2){$\eta_{t^*_1}$};
\node[right,color=orange] at ({4.8+\gd},6.2){$\eta_{t^*_2}$};
\node[right,color=orange] at ({4.3+\gu},6.2){$\eta_{t^*_3}$};

  \pgfmathsetmacro{\pu}{1.5}
  \pgfmathsetmacro{\pd}{-3}
     \pgfmathsetmacro{\pt}{0.55}

\pgfmathsetmacro{\t}{4}

\pgfmathsetmacro{\Startu}{0.41+\A*cos(deg(1.3*(\o*4+1)))}
\pgfmathsetmacro{\Startd}{3*0.41+\A*cos(deg(1.3*(\o*4+3)))}
\pgfmathsetmacro{\Startt}{5.9*0.41+\A*cos(deg(1.3*(\o*4+5.9)))}

\draw[dashed,color=gray](\gu,0)--(\gu,1);
\pgfmathsetmacro{\a}{3}
\draw[dashed,color=gray]({\gu-(\a*0.41+\A*cos(deg(1.3*(\o*\t+\a))))+\Startu},0)--({\gu-(\a*0.41+\A*cos(deg(1.3*(\o*\t+\a))))+\Startu},\a);
\pgfmathsetmacro{\a}{5.8}
\draw[dashed,color=gray]({\gu+\a*0.41+\A*cos(deg(1.3*(\o*\t+\a)))-\Startd+\DelU},0)--({\gu+\a*0.41+\A*cos(deg(1.3*(\o*\t+\a)))-\Startd+\DelU},\a);
\pgfmathsetmacro{\a}{6.2}
\draw[dashed,color=gray]({\gu-(\a*0.41+\A*cos(deg(1.3*(\o*\t+\a))))+0.01+\Startt+\DelU+\DelD},0)--({\gu-(\a*0.41+\A*cos(deg(1.3*(\o*\t+\a))))+0.01+\Startt+\DelU+\DelD},\a);
\draw[dashed,color=gray](\gu,6.2)--({\gu-(\a*0.41+\A*cos(deg(1.3*(\o*\t+\a))))+0.01+\Startt+\DelU+\DelD},\a);
\draw[color=gray,->](\gu,0)--(\gu,6.5);
\draw[color=green,thick](\gu,1)node[left]{$\alpha$}--(\gu,6.2)node[left]{$\beta$};
%


  \foreach \a in {1,1.1,1.2,1.3,1.4,1.5,1.6,1.7,1.8,1.9,2,2.1,2.2,2.3,2.4,2.5,2.6,2.7,2.8,2.9}
 {
  \draw[color=orange,thick]({\gu-(\a*0.41+\A*cos(deg(1.3*(\o*\t+\a))))+\Startu},\a)--({\gu-((\a+0.1)*0.41+\A*cos(deg(1.3*(\o*\t+(\a+0.1)))))+\Startu},{\a+0.1});
  }

  \foreach \a in {3,3.1,3.2,3.3,3.4,3.5,3.6,3.7,3.8,3.9,4,4.1,4.2,4.3,4.4,4.5,4.6,4.7,4.8,4.9,5,5.1,5.2,5.3,5.4,5.5,5.6,5.7,5.8}
 {
  \draw[color=orange,thick]({\gu+\a*0.41+\A*cos(deg(1.3*(\o*\t+\a)))-\Startd+\DelU},\a)--({\gu+(\a+0.1)*0.41+\A*cos(deg(1.3*(\o*\t+(\a+0.1))))-\Startd+\DelU},{\a+0.1});
  }

  \foreach \a in {5.9,6,6.1}
 {
  \draw[color=orange,thick]({\gu-(\a*0.41+\A*cos(deg(1.3*(\o*\t+\a))))+0.01+\Startt+\DelU+\DelD},\a)--({\gu-((\a+0.1)*0.41+\A*cos(deg(1.3*(\o*\t+(\a+0.1)))))+0.01+\Startt+\DelU+\DelD},{\a+0.1});
  }


%
\pgfmathsetmacro{\DelUd}{2.6*0.41+\A*cos(deg(1.3*(\o*(2.4)+2.6)))-(0.41+\A*cos(deg(1.3*(\o*(2.4)+1))))}
   \pgfmathsetmacro{\DelDd}{2.6*0.41+\A*cos(deg(1.3*(\o*(2.4)+2.6)))-(4.6*0.41+\A*cos(deg(1.3*(\o*(2.4)+4.6))))}
\pgfmathsetmacro{\DelTd}{6.2*0.41+\A*cos(deg(1.3*(\o*(2.4)+6.2)))-(4.6*0.41+\A*cos(deg(1.3*(\o*(2.4)+4.6))))}

   \draw[color=red,->,thick](\gd,0)--({\gd+\DelUd},0)node[below left,color=gray]{$m_{t^*_2,1}$};
        \draw[color=blue,->,thick]({\DelUd+\gd},0)--({\DelUd+\DelDd+\gd},0)node[below,color=gray]{$m_{t^*_2,2}$};
        \draw[color=red,->,thick]({\DelUd+\DelDd+\gd},0)--({\DelUd+\DelDd+\DelTd+\gd},0)node[below,color=gray]{$m_{t^*_2,3}$};

        \pgfmathsetmacro{\t}{2.4}

        \pgfmathsetmacro{\Startu}{0.41+\A*cos(deg(1.3*(\o*2.4+1)))}
          \pgfmathsetmacro{\Startd}{2.7*0.41+\A*cos(deg(1.3*(\o*2.4+2.7)))}
  \pgfmathsetmacro{\Startt}{4.6*0.41+\A*cos(deg(1.3*(\o*2.4+4.6)))}

\draw[dashed,color=gray](\gd,0)--(\gd,1);
\pgfmathsetmacro{\a}{2.6}
\draw[dashed,color=gray]({\gd+(\a*0.41+\A*cos(deg(1.3*(\o*\t+\a))))-\Startu},0)--({\gd+(\a*0.41+\A*cos(deg(1.3*(\o*\t+\a))))-\Startu},\a);
\pgfmathsetmacro{\a}{4.6}
\draw[dashed,color=gray]({\gd-(\a*0.41+\A*cos(deg(1.3*(\o*\t+\a))))+\Startd+\DelUd},0)--({\gd-(\a*0.41+\A*cos(deg(1.3*(\o*\t+\a))))+\Startd+\DelUd},\a);
\pgfmathsetmacro{\a}{6.2}
\draw[dashed,color=gray]({\gd+(\a*0.41+\A*cos(deg(1.3*(\o*\t+\a))))-\Startt+\DelUd+\DelDd},0)--({\gd+(\a*0.41+\A*cos(deg(1.3*(\o*\t+\a))))-\Startt+\DelUd+\DelDd},\a);
\draw[dashed,color=gray](\gd,6.2)--({\gd+(\a*0.41+\A*cos(deg(1.3*(\o*\t+\a))))-\Startt+\DelUd+\DelDd},\a);
\draw[color=gray,->](\gd,0)--(\gd,6.5);
\draw[color=green,thick](\gd,1)node[left]{$\alpha$}--(\gd,6.2)node[left]{$\beta$};
%

        \foreach \a in {1,1.1,1.2,1.3,1.4,1.5,1.6,1.7,1.8,1.9,2,2.1,2.2,2.3,2.4,2.5,2.6}
 {
  \draw[color=orange,thick]({\gd+(\a*0.41+\A*cos(deg(1.3*(\o*\t+\a))))-\Startu},\a)--({\gd+((\a+0.1)*0.41+\A*cos(deg(1.3*(\o*\t+(\a+0.1)))))-\Startu},{\a+0.1});
  }

  \foreach \a in {2.7,2.8,2.9,3,3.1,3.2,3.3,3.4,3.5,3.6,3.7,3.8,3.9,4,4.1,4.2,4.3,4.4,4.5}
 {
  \draw[color=orange,thick]({\gd-(\a*0.41+\A*cos(deg(1.3*(\o*\t+\a))))+\Startd+\DelUd},\a)--({\gd-((\a+0.1)*0.41+\A*cos(deg(1.3*(\o*\t+(\a+0.1)))))+\Startd+\DelUd},{\a+0.1});
  }

  \foreach \a in {4.6,4.7,4.8,4.9,5,5.1,5.2,5.3,5.4,5.5,5.6,5.7,5.8,5.9,6,6.1}
 {
  \draw[color=orange,thick]({\gd+(\a*0.41+\A*cos(deg(1.3*(\o*\t+\a))))-\Startt+\DelUd+\DelDd},\a)--({\gd+((\a+0.1)*0.41+\A*cos(deg(1.3*(\o*\t+(\a+0.1)))))-\Startt+\DelUd+\DelDd},{\a+0.1});
  }


   \pgfmathsetmacro{\DelUt}{0.41+\A*cos(deg(1.3*(\o*(0.4)+1))-(1.8*0.41+\A*cos(deg(1.3*(\o*(0.4)+1.8))))}
   \pgfmathsetmacro{\DelDt}{4.6*0.41+\A*cos(deg(1.3*(\o*(0.4)+4.6)))-(1.8*0.41+\A*cos(deg(1.3*(\o*(0.4)+1.8))))}
\pgfmathsetmacro{\DelTt}{4.6*0.41+\A*cos(deg(1.3*(\o*(0.4)+4.6)))-(6.2*0.41+\A*cos(deg(1.3*(\o*(0.4)+6.2))))}


   \draw[color=blue,->,thick](0,0)--(\DelUt+0.281,0)node[below,color=gray]{$m_{t^*_1,1}$};
        \draw[color=red,->,thick](\DelUt+0.281,0)--({\DelUt+0.281+\DelDt},0)node[below left,color=gray]{$m_{t^*_1,2}$};
        \draw[color=blue,->,thick]({\DelUt+0.281+\DelDt},0)--({\DelUt+0.281+\DelDt+\DelTt},0)node[below,color=gray]{$m_{t^*_1,3}$};

 \pgfmathsetmacro{\t}{0.4}

\pgfmathsetmacro{\Startu}{0.41+\A*cos(deg(1.3*(\o*0.4+1)))}
\pgfmathsetmacro{\Startd}{1.8*0.41+\A*cos(deg(1.3*(\o*0.4+1.8)))}
 \pgfmathsetmacro{\Startt}{4.7*0.41+\A*cos(deg(1.3*(\o*0.4+4.7)))}

\draw[dashed,color=gray](0,0)--(0,1);
\pgfmathsetmacro{\a}{1.8}
\draw[dashed,color=gray]({-(\a*0.41+\A*cos(deg(1.3*(\o*\t+\a))))+\Startu},0)--({-(\a*0.41+\A*cos(deg(1.3*(\o*\t+\a))))+\Startu},\a);
\pgfmathsetmacro{\a}{4.6}
\draw[dashed,color=gray]({\a*0.41+\A*cos(deg(1.3*(\o*\t+\a)))-\Startd+\DelUt+0.281},0)--({\a*0.41+\A*cos(deg(1.3*(\o*\t+\a)))-\Startd+\DelUt+0.281},\a);
\pgfmathsetmacro{\a}{6.2}
\draw[dashed,color=gray]({-(\a*0.41+\A*cos(deg(1.3*(\o*\t+\a))))+0.01+\Startt+\DelUt+0.281+\DelDt},0)--({-(\a*0.41+\A*cos(deg(1.3*(\o*\t+\a))))+0.01+\Startt+\DelUt+0.281+\DelDt},\a);
\draw[dashed,color=gray](0,6.2)--({-(\a*0.41+\A*cos(deg(1.3*(\o*\t+\a))))+0.01+\Startt+\DelUt+0.281+\DelDt},\a);
\draw[color=gray,->](0,0)--(0,6.5);
\draw[color=green,thick](0,1)node[left]{$\alpha$}--(0,6.2)node[left]{$\beta$};
%

  \foreach \a in {1,1.1,1.2,1.3,1.4,1.5,1.6,1.7}
 {
  \draw[color=orange,thick]({-(\a*0.41+\A*cos(deg(1.3*(\o*\t+\a))))+\Startu},\a)--({-((\a+0.1)*0.41+\A*cos(deg(1.3*(\o*\t+(\a+0.1)))))+\Startu},{\a+0.1});
  }

  \foreach \a in {1.8,1.9,2,2.1,2.2,2.3,2.4,2.5,2.6,2.7,2.8,2.9,3,3.1,3.2,3.3,3.4,3.5,3.6,3.7,3.8,3.9,4,4.1,4.2,4.3,4.4,4.5,4.6}
 {
  \draw[color=orange,thick]({\a*0.41+\A*cos(deg(1.3*(\o*\t+\a)))-\Startd+\DelUt+0.23},\a)--({(\a+0.1)*0.41+\A*cos(deg(1.3*(\o*\t+(\a+0.1))))-\Startd+\DelUt+0.23},{\a+0.1});
  }

  \foreach \a in {4.7,4.8,4.9,5,5.1,5.2,5.3,5.4,5.5,5.6,5.7,5.8,5.9,6,6.1}
 {
  \draw[color=orange,thick]({-(\a*0.41+\A*cos(deg(1.3*(\o*\t+\a))))+\Startt+\DelUt+\DelDt+0.23},\a)--({-((\a+0.1)*0.41+\A*cos(deg(1.3*(\o*\t+(\a+0.1)))))+\Startt+\DelUt+\DelDt+0.23},{\a+0.1});
  }

\end{tikzpicture}
\end{center}
\caption{The functions $\eta_{t^*_1}$, $\eta_{t^*_2}$ and $\eta_{t^*_3}$
for different values of $t^*_1< t^*_2 <t^*_3.$}
\label{fig2}
\end{figure}

Now we conclude this section providing an heuristic derivation of
the PDF $\mu_Y=\psi(y,t)$. Let $y$ be a point of the physical space where it exists a probability of finding the system, with respect to the random variable $\phi$. The PDF $\psi(y,t)$ is
obtained by determining all the points
$\phi_i(y)\in[\alpha,\beta]$ such that $\varphi(\phi_i(y),t)=y$.
Indeed, we expect that the probability of finding the system in
$y$ at time $t$ can be obtained as the sum of the probabilities
of all the values of $\phi_i(y)$ that will make the system to be in the position $y$ at the time
$t$. The subintervals within $[\alpha,\beta]$ where these $\phi_i(y)$
are located, are determined by the set of indexes $I_t(y).$ This is due to the fact that $\varphi(\phi,t)$ is \emph{spreading} the
mass probability $\mu_X$ around in the image space. In Figure \ref{fig3} we show the regions where overlap will occur (darker green). In turn, in Figure \ref{fig4} we show how this overlap will displace the resulting probability. For simplicity, however, in this later figure we have disregarded the Jacobian of the transformation in the summation of the different parts. However, this Jacobian is a essential part in the transformation as $\varphi$ squeezes or stretches
the probability mass $\mu(\phi)$ in the final space.This
deformation also affects the final PDF $\psi$ through the derivative (Jacobian) in \eqref{eq-1.6}.
For example, when $y$ corresponds to a local minimum
or local maximum value of $\varphi(\phi,t)$, we expect a \emph{peak
of probability}. Indeed the derivative
$\partial\varphi/\partial\phi$ is zero and then the derivative of
the corresponding inverse auxiliary function $\eta_t$ is infinity.

In summary, the main idea behind the FDF
approach can be sketched in following scheme:
\begin{equation}\label{eq-schema}
\begin{split}
y&\overset{u_{i,t}}{\longrightarrow}
\{u_{\rho_1},\dots,u_{\rho_{p_t}}\}
\overset{\eta_t}{\longrightarrow}
\{\phi_i:i\in I_t(y)\}\overset{\mu_X}{\longrightarrow}\dots\\
&\dots\overset{\mu_X}{\longrightarrow} \{\mu_X(\phi_i):i\in I_t(y)\}
\overset{\sum}{\longrightarrow}
\sum_{i\in I_t(y)}\mu_X(\phi_i(y))\cdot\frac{d\eta_t}{du}(u_i),
\end{split}
\end{equation}
where we recall that $\phi_i=\eta_t(u_i)$ for each $i\in I_t(y)$.

We want to emphasize that the FDF approach presents several
advantages from the point of view of numerical implementation of an algorithm aiming at the RVT.
In particular,
\begin{enumerate}
\item $u_{i,t}(x)$ is piecewise linear,
\item $\eta_{t}(u)$ is invertible,
\item both of these functions can be very easily determined numerically.
\end{enumerate}
All these facts permits the use of an algorithm which is not based on the reconstruction of the final probability via the generation of bundles of trajectories, using the Monte Carlo method, and calculating the final probability through a histogram. Instead, the piece-wise nature of our approach allows us to construct the final solution through a more effective piece-wise analysis.

\begin{figure}[ht]
\begin{center}
\begin{tikzpicture}[scale=0.8]

   \pgfmathsetmacro{\o}{1}
      \pgfmathsetmacro{\A}{1}

 \pgfmathsetmacro{\pu}{2}
   \pgfmathsetmacro{\pd}{-3}
      \pgfmathsetmacro{\pt}{0.55}

      \pgfmathsetmacro{\fa}{2}
\pgfmathsetmacro{\delu}{5}

\pgfmathsetmacro{\deld}{10}

\draw[->](-0.5,-0.4)--(3.5,-0.4)node[below]{$y$};
\draw[->](0,-0.7)--(0,1.6)node[left]{};

\pgfmathsetmacro{\t}{0.4}

  \foreach \a in {1,1.1,1.2,1.3,1.4,1.5,1.6,1.7,1.8}
 {
 \path[fill=green,opacity=0.41]({\a*0.41+\A*cos(deg(1.3*(\o*\t+\a)))},{-\t})--({\a*0.41+\A*cos(deg(1.3*(\o*\t+\a)))},{-(\t-1-\pt*cos(deg(\a*\pu+\fa)))})--({(\a+0.1)*0.41+\A*cos(deg(1.3*(\o*\t+\a+0.1)))},{-(\t-1-\pt*cos(deg((\a+0.1)*\pu+\fa)))})--({(\a+0.1)*0.41+\A*cos(deg(1.3*(\o*\t+\a+0.1)))},{-\t})--({\a*0.41+\A*cos(deg(1.3*(\o*\t+\a)))},{-\t});
  \draw[color=purple,thick]({\a*0.41+\A*cos(deg(1.3*(\o*\t+\a)))},{-(\t-1-\pt*cos(deg(\a*\pu+\fa)))})--({(\a+0.1)*0.41+\A*cos(deg(1.3*(\o*\t+\a+0.1)))},{-(\t-1-\pt*cos(deg((\a+0.1)*\pu+\fa)))});
  }

  \foreach \a in {1.8,1.9,2,2.1,2.2,2.3,2.4,2.5,2.6,2.7,2.8,2.9,3,3.1,3.2,3.3,3.4,3.5,3.6,3.7,3.8,3.9,4,4.1,4.2,4.3,4.4}
 {
 \path[fill=green,opacity=0.41]({\a*0.41+\A*cos(deg(1.3*(\o*\t+\a)))},{-\t})--({\a*0.41+\A*cos(deg(1.3*(\o*\t+\a)))},{-(\t-1-\pt*cos(deg(\a*\pu+\fa)))})--({(\a+0.1)*0.41+\A*cos(deg(1.3*(\o*\t+\a+0.1)))},{-(\t-1-\pt*cos(deg((\a+0.1)*\pu+\fa)))})--({(\a+0.1)*0.41+\A*cos(deg(1.3*(\o*\t+\a+0.1)))},{-\t})--({\a*0.41+\A*cos(deg(1.3*(\o*\t+\a)))},{-\t});
  \draw[color=purple,thick]({\a*0.41+\A*cos(deg(1.3*(\o*\t+\a)))},{-(\t-1-\pt*cos(deg(\a*\pu+\fa)))})--({(\a+0.1)*0.41+\A*cos(deg(1.3*(\o*\t+\a+0.1)))},{-(\t-1-\pt*cos(deg((\a+0.1)*\pu+\fa)))});
  }
  \foreach \a in {4.4,4.5,4.6,4.7,4.8,4.9,5,5.1,5.2,5.3,5.4,5.5,5.6,5.7,5.8,5.9,6}
 {
 \path[fill=green,opacity=0.41]({\a*0.41+\A*cos(deg(1.3*(\o*\t+\a)))},{-\t})--({\a*0.41+\A*cos(deg(1.3*(\o*\t+\a)))},{-(\t-1-\pt*cos(deg(\a*\pu+\fa)))})--({(\a+0.1)*0.41+\A*cos(deg(1.3*(\o*\t+\a+0.1)))},{-(\t-1-\pt*cos(deg((\a+0.1)*\pu+\fa)))})--({(\a+0.1)*0.41+\A*cos(deg(1.3*(\o*\t+\a+0.1)))},{-\t})--({\a*0.41+\A*cos(deg(1.3*(\o*\t+\a)))},{-\t});
  \draw[color=purple,thick]({\a*0.41+\A*cos(deg(1.3*(\o*\t+\a)))},{-(\t-1-\pt*cos(deg(\a*\pu+\fa)))})--({(\a+0.1)*0.41+\A*cos(deg(1.3*(\o*\t+\a+0.1)))},{-(\t-1-\pt*cos(deg((\a+0.1)*\pu+\fa)))});
  }



 \draw[->]({-0.5+\delu},-0.4)--({3.5+\delu},-0.4)node[below]{$y$};
\draw[->](\delu,-0.7)--(\delu,1.6)node[left]{};

\pgfmathsetmacro{\t}{2.4}

  \foreach \a in {1,1.1,1.2,1.3,1.4,1.5,1.6,1.7,1.8,1.9,2,2.1,2.2,2.3,2.4}
 {
 \path[fill=green,opacity=0.41]({\a*0.41+\A*cos(deg(1.3*(\o*\t+\a)))+\delu},{-\t+2})--({\a*0.41+\A*cos(deg(1.3*(\o*\t+\a)))+\delu},{-(\t-1-\pt*cos(deg(\a*\pu+\fa)))+2})--({(\a+0.1)*0.41+\A*cos(deg(1.3*(\o*\t+\a+0.1)))+\delu},{-(\t-1-\pt*cos(deg((\a+0.1)*\pu+\fa)))+2})--({(\a+0.1)*0.41+\A*cos(deg(1.3*(\o*\t+\a+0.1)))+\delu},{-\t+2})--({\a*0.41+\A*cos(deg(1.3*(\o*\t+\a)))+\delu},{-\t+2});
  \draw[color=purple,thick]({\a*0.41+\A*cos(deg(1.3*(\o*\t+\a)))+\delu},{-(\t-1-\pt*cos(deg(\a*\pu+\fa)))+2})--({(\a+0.1)*0.41+\A*cos(deg(1.3*(\o*\t+\a+0.1)))+\delu},{-(\t-1-\pt*cos(deg((\a+0.1)*\pu+\fa)))+2});
  }

  \foreach \a in {2.4,2.5,2.6,2.7,2.8,2.9,3,3.1,3.2,3.3,3.4,3.5,3.6,3.7,3.8,3.9,4,4.1,4.2,4.3,4.4}
 {
 \path[fill=green,opacity=0.41]({\a*0.41+\A*cos(deg(1.3*(\o*\t+\a)))+\delu},{-\t+2})--({\a*0.41+\A*cos(deg(1.3*(\o*\t+\a)))+\delu},{-(\t-1-\pt*cos(deg(\a*\pu+\fa)))+2})--({(\a+0.1)*0.41+\A*cos(deg(1.3*(\o*\t+\a+0.1)))+\delu},{-(\t-1-\pt*cos(deg((\a+0.1)*\pu+\fa)))+2})--({(\a+0.1)*0.41+\A*cos(deg(1.3*(\o*\t+\a+0.1)))+\delu},{-\t+2})--({\a*0.41+\A*cos(deg(1.3*(\o*\t+\a)))+\delu},{-\t+2});
  \draw[color=purple,thick]({\a*0.41+\A*cos(deg(1.3*(\o*\t+\a)))+\delu},{-(\t-1-\pt*cos(deg(\a*\pu+\fa)))+2})--({(\a+0.1)*0.41+\A*cos(deg(1.3*(\o*\t+\a+0.1)))+\delu},{-(\t-1-\pt*cos(deg((\a+0.1)*\pu+\fa)))+2});
  }
  \foreach \a in {4.4,4.5,4.6,4.7,4.8,4.9,5,5.1,5.2,5.3,5.4,5.5,5.6,5.7,5.8,5.9,6}
 {
 \path[fill=green,opacity=0.41]({\a*0.41+\A*cos(deg(1.3*(\o*\t+\a)))+\delu},{-\t+2})--({\a*0.41+\A*cos(deg(1.3*(\o*\t+\a)))+\delu},{-(\t-1-\pt*cos(deg(\a*\pu+\fa)))+2})--({(\a+0.1)*0.41+\A*cos(deg(1.3*(\o*\t+\a+0.1)))+\delu},{-(\t-1-\pt*cos(deg((\a+0.1)*\pu+\fa)))+2})--({(\a+0.1)*0.41+\A*cos(deg(1.3*(\o*\t+\a+0.1)))+\delu},{-\t+2})--({\a*0.41+\A*cos(deg(1.3*(\o*\t+\a)))+\delu},{-\t+2});
  \draw[color=purple,thick]({\a*0.41+\A*cos(deg(1.3*(\o*\t+\a)))+\delu},{-(\t-1-\pt*cos(deg(\a*\pu+\fa)))+2})--({(\a+0.1)*0.41+\A*cos(deg(1.3*(\o*\t+\a+0.1)))+\delu},{-(\t-1-\pt*cos(deg((\a+0.1)*\pu+\fa)))+2});
  }




 \draw[->]({-0.5+\deld},-0.4)--({3.5+\deld},-0.4)node[below]{$y$};
\draw[->](\deld,-0.7)--(\deld,1.6)node[left]{};

\pgfmathsetmacro{\t}{4}

  \foreach \a in {1,1.1,1.2,1.3,1.4,1.5,1.6,1.7,1.8,1.9,2,2.1,2.2,2.3,2.4,2.5,2.6,2.7,2.8}
 {
 \path[fill=green,opacity=0.41]({\a*0.41+\A*cos(deg(1.3*(\o*\t+\a)))+\deld},{-\t+3.6})--({\a*0.41+\A*cos(deg(1.3*(\o*\t+\a)))+\deld},{-(\t-1-\pt*cos(deg(\a*\pu+\fa)))+3.6})--({(\a+0.1)*0.41+\A*cos(deg(1.3*(\o*\t+\a+0.1)))+\deld},{-(\t-1-\pt*cos(deg((\a+0.1)*\pu+\fa)))+3.6})--({(\a+0.1)*0.41+\A*cos(deg(1.3*(\o*\t+\a+0.1)))+\deld},{-\t+3.6})--({\a*0.41+\A*cos(deg(1.3*(\o*\t+\a)))+\deld},{-\t+3.6});
  \draw[color=purple,thick]({\a*0.41+\A*cos(deg(1.3*(\o*\t+\a)))+\deld},{-(\t-1-\pt*cos(deg(\a*\pu+\fa)))+3.6})--({(\a+0.1)*0.41+\A*cos(deg(1.3*(\o*\t+\a+0.1)))+\deld},{-(\t-1-\pt*cos(deg((\a+0.1)*\pu+\fa)))+3.6});
  }

  \foreach \a in {2.8,2.9,3,3.1,3.2,3.3,3.4,3.5,3.6,3.7,3.8,3.9,4,4.1,4.2,4.3,4.4,4.5,4.6,4.7,4.8,4.9,5,5.1,5.2,5.3,5.4,5.5,5.6}
 {
 \path[fill=green,opacity=0.41]({\a*0.41+\A*cos(deg(1.3*(\o*\t+\a)))+\deld},{-\t+3.6})--({\a*0.41+\A*cos(deg(1.3*(\o*\t+\a)))+\deld},{-(\t-1-\pt*cos(deg(\a*\pu+\fa)))+3.6})--({(\a+0.1)*0.41+\A*cos(deg(1.3*(\o*\t+\a+0.1)))+\deld},{-(\t-1-\pt*cos(deg((\a+0.1)*\pu+\fa)))+3.6})--({(\a+0.1)*0.41+\A*cos(deg(1.3*(\o*\t+\a+0.1)))+\deld},{-\t+3.6})--({\a*0.41+\A*cos(deg(1.3*(\o*\t+\a)))+\deld},{-\t+3.6});
  \draw[color=purple,thick]({\a*0.41+\A*cos(deg(1.3*(\o*\t+\a)))+\deld},{-(\t-1-\pt*cos(deg(\a*\pu+\fa)))+3.6})--({(\a+0.1)*0.41+\A*cos(deg(1.3*(\o*\t+\a+0.1)))+\deld},{-(\t-1-\pt*cos(deg((\a+0.1)*\pu+\fa)))+3.6});
  }
  \foreach \a in {5.6,5.7,5.8,5.9,6}
 {
 \path[fill=green,opacity=0.41]({\a*0.41+\A*cos(deg(1.3*(\o*\t+\a)))+\deld},{-\t+3.6})--({\a*0.41+\A*cos(deg(1.3*(\o*\t+\a)))+\deld},{-(\t-1-\pt*cos(deg(\a*\pu+\fa)))+3.6})--({(\a+0.1)*0.41+\A*cos(deg(1.3*(\o*\t+\a+0.1)))+\deld},{-(\t-1-\pt*cos(deg((\a+0.1)*\pu+\fa)))+3.6})--({(\a+0.1)*0.41+\A*cos(deg(1.3*(\o*\t+\a+0.1)))+\deld},{-\t+3.6})--({\a*0.41+\A*cos(deg(1.3*(\o*\t+\a)))+\deld},{-\t+3.6});
  \draw[color=purple,thick]({\a*0.41+\A*cos(deg(1.3*(\o*\t+\a)))+\deld},{-(\t-1-\pt*cos(deg(\a*\pu+\fa)))+3.6})--({(\a+0.1)*0.41+\A*cos(deg(1.3*(\o*\t+\a+0.1)))+\deld},{-(\t-1-\pt*cos(deg((\a+0.1)*\pu+\fa)))+3.6});
  }
\node at (2,-1){$t^*_1$}; \node[color=purple,thick] at (2,2)
{$\tilde{\mu}_Y^{t^*_1}$}; \node at (2+\delu,-1){$t^*_2$};
\node[color=purple,thick] at (2+\delu,2)
{$\tilde{\mu}_Y^{t^*_2}$}; \node at (2+\deld,-1){$t^*_3$};
\node[color=purple,thick] at (2+\deld,2)
{$\tilde{\mu}_Y^{t^*_3}$};
\end{tikzpicture}
\end{center}
\caption{Here we have $t=t^*_1$, $t=t^*_2$ and $t=t^*_3$. We
notice that $I_{t^*_1}(y)=\{1,2\},$ $I_{t^*_1}(y)=\{2\}$ and
$I_{t^*_1}(y)=\{2,3\}$ as $y$ varies from the $g_{t^*_1,\min}$ to
$g_{t^*_1,\max}.$ Similarly, the three cases of $t=t^*_2$ as $y$
varies, are $I_{t^*_1}(y)=\{1\},$ $I_{t^*_1}(y)=\{1,2,3\}$ and
$I_{t^*_1}(y)=\{3\}$. Finally, the three cases of $t=t^*_3$ as $y$
varies, are $I_{t^*_1}(y)=\{1,2\},$ $I_{t^*_1}(y)=\{2\}$ and
$I_{t^*_1}(y)=\{2,3\}$.} \label{fig3}
\end{figure}
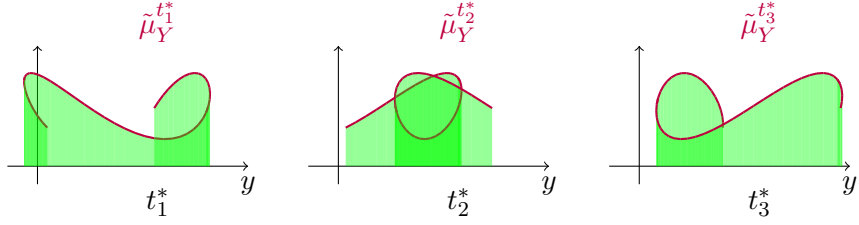

\begin{figure}[ht]
\begin{center}
\begin{tikzpicture}[scale=0.8]
\label{figPSIpre}


   \pgfmathsetmacro{\o}{1}
      \pgfmathsetmacro{\A}{1}

 \pgfmathsetmacro{\pu}{2}
   \pgfmathsetmacro{\pd}{-3}
      \pgfmathsetmacro{\pt}{0.55}

      \pgfmathsetmacro{\fa}{2}


\pgfmathsetmacro{\delu}{5}

\pgfmathsetmacro{\deld}{10}

\draw[->](-0.5,-0.4)--(3.5,-0.4)node[below]{$y$};
\draw[->](0,-0.7)--(0,3.5)node[left]{};

\pgfmathsetmacro{\t}{0.4}

  \foreach \a in {1,1.1,1.2,1.3,1.4,1.5,1.6,1.7}
 {
  \draw[color=purple,thick,opacity=0.2]({\a*0.41+\A*cos(deg(1.3*(\o*\t+\a)))},{-(\t-1-\pt*cos(deg(\a*\pu+\fa)))})--({(\a+0.1)*0.41+\A*cos(deg(1.3*(\o*\t+\a+0.1)))},{-(\t-1-\pt*cos(deg((\a+0.1)*\pu+\fa)))});
  }

  \foreach \a in {1,1.1,1.2,1.3,1.4,1.5,1.6,1.7}
 {
  \draw[color=ForestGreen,thick]({\a*0.41+\A*cos(deg(1.3*(\o*\t+\a)))},{-(\t-1-\pt*cos(deg(\a*\pu+\fa)))-(\t-1-\pt*cos(deg((1.8-(\a-1.7))*\pu+\fa)))+0.4})--({(\a+0.1)*0.41+\A*cos(deg(1.3*(\o*\t+\a+0.1)))},{-(\t-1-\pt*cos(deg((\a+0.1)*\pu+\fa)))-(\t-1-\pt*cos(deg((1.8-(\a-1.6))*\pu+\fa)))+0.4});
}

  \foreach \a in {1.8,1.9,2,2.1,2.2,2.3,2.4,2.5,2.6,2.7,2.8,2.9,3,3.1,3.2,3.3,3.4,3.5,3.6,3.7,3.8,3.9,4,4.1,4.2,4.3,4.4,4.5}
 {
  \draw[color=purple,thick,opacity=0.2]({\a*0.41+\A*cos(deg(1.3*(\o*\t+\a)))},{-(\t-1-\pt*cos(deg(\a*\pu+\fa)))})--({(\a+0.1)*0.41+\A*cos(deg(1.3*(\o*\t+\a+0.1)))},{-(\t-1-\pt*cos(deg((\a+0.1)*\pu+\fa)))});
  }

  \draw[dashed,color=gray]({1.7*0.41+\A*cos(deg(1.3*(\o*\t+1.7)))},-0.4)--({1.7*0.41+\A*cos(deg(1.3*(\o*\t+1.7)))},2.3);
  \draw[dashed,color=gray]({2.45*0.41+\A*cos(deg(1.3*(\o*\t+2.45)))},-0.4)--({2.45*0.41+\A*cos(deg(1.3*(\o*\t+2.45)))},1.6);
  \draw[dashed,color=gray]({3.6*0.41+\A*cos(deg(1.3*(\o*\t+3.6)))},-0.4)--({3.6*0.41+\A*cos(deg(1.3*(\o*\t+3.6)))},1.1);
  \draw[dashed,color=gray]({4.6*0.41+\A*cos(deg(1.3*(\o*\t+4.6)))},-0.4)--({4.6*0.41+\A*cos(deg(1.3*(\o*\t+4.6)))},1.75)node[color=ForestGreen,right]{$f_{t^*_1}$};

  \foreach \a in {2.45,2.5,2.6,2.7,2.8,2.9,3,3.1,3.2,3.3,3.4,3.5}
 {
  \draw[color=ForestGreen,thick]({\a*0.41+\A*cos(deg(1.3*(\o*\t+\a)))},{-(\t-1-\pt*cos(deg(\a*\pu+\fa)))})--({(\a+0.1)*0.41+\A*cos(deg(1.3*(\o*\t+\a+0.1)))},{-(\t-1-\pt*cos(deg((\a+0.1)*\pu+\fa)))});
  }

  \foreach \a in {4.6,4.7,4.8,4.9,5,5.1,5.2,5.3,5.4,5.5,5.6,5.7,5.8,5.9,6}
 {
  \draw[color=purple,thick,opacity=0.2]({\a*0.41+\A*cos(deg(1.3*(\o*\t+\a)))},{-(\t-1-\pt*cos(deg(\a*\pu+\fa)))})--({(\a+0.1)*0.41+\A*cos(deg(1.3*(\o*\t+\a+0.1)))},{-(\t-1-\pt*cos(deg((\a+0.1)*\pu+\fa)))});
  }

  \foreach \a in {4.65,4.7,4.8,4.9,5,5.1,5.2,5.3,5.4,5.5,5.6,5.7,5.8,5.9,6}
 {
    \draw[color=ForestGreen,thick]({\a*0.41+\A*cos(deg(1.3*(\o*\t+\a)))},{-(\t-1-\pt*cos(deg(\a*\pu+\fa)))-(\t-1-\pt*cos(deg((4.5-0.64*(\a-4.6))*\pu+\fa)))+0.4})--({(\a+0.1)*0.41+\A*cos(deg(1.3*(\o*\t+\a+0.1)))},{-(\t-1-\pt*cos(deg((\a+0.1)*\pu+\fa)))-(\t-1-\pt*cos(deg((4.5-0.64*(\a-4.5))*\pu+\fa)))+0.4});
 }


%

 \draw[->]({-0.5+\delu},-0.4)--({3.5+\delu},-0.4)node[below]{$y$};
\draw[->](\delu,-0.7)--(\delu,3.5)node[left]{};

\pgfmathsetmacro{\t}{2.4}

  \foreach \a in {1,1.1,1.2,1.3,1.4,1.5,1.6,1.7,1.8,1.9,2,2.1,2.2,2.3,2.4,2.5}
 {
  \draw[color=purple,thick,opacity=0.2]({\a*0.41+\A*cos(deg(1.3*(\o*\t+\a)))+\delu},{-(\t-1-\pt*cos(deg(\a*\pu+\fa)))+2})--({(\a+0.1)*0.41+\A*cos(deg(1.3*(\o*\t+\a+0.1)))+\delu},{-(\t-1-\pt*cos(deg((\a+0.1)*\pu+\fa)))+2});
  }

  \foreach \a in {2.6,2.7,2.8,2.9,3,3.1,3.2,3.3,3.4,3.5,3.6,3.7,3.8,3.9,4,4.1,4.2,4.3,4.4}
 {
  \draw[color=purple,thick,opacity=0.2]({\a*0.41+\A*cos(deg(1.3*(\o*\t+\a)))+\delu},{-(\t-1-\pt*cos(deg(\a*\pu+\fa)))+2})--({(\a+0.1)*0.41+\A*cos(deg(1.3*(\o*\t+\a+0.1)))+\delu},{-(\t-1-\pt*cos(deg((\a+0.1)*\pu+\fa)))+2});
  }

 \foreach \a in {2.5,2.6,2.7,2.8,2.9,3,3.1,3.2,3.3,3.4,3.5,3.6,3.7,3.8,3.9,4,4.1,4.2,4.3,4.4}
 {



      \draw[color=ForestGreen,thick]({\a*0.41+\A*cos(deg(1.3*(\o*\t+\a)))+\delu},{-(\t-1-\pt*cos(deg(\a*\pu+\fa)))-(\t-1-\pt*cos(deg((2.5-0.55*(\a-2.6))*\pu+\fa)))-(\t-1-\pt*cos(deg((5.7-0.66*(\a-2.7))*\pu+\fa)))+6.4})--({(\a+0.1)*0.41+\A*cos(deg(1.3*(\o*\t+\a+0.1)))+\delu},{-(\t-1-\pt*cos(deg((\a+0.1)*\pu+\fa)))-(\t-1-\pt*cos(deg((2.5-0.55*(\a-2.5))*\pu+\fa)))-(\t-1-\pt*cos(deg((5.7-0.66*(\a-2.6))*\pu+\fa)))+6.4});

  }

  \foreach \a in {4.5,4.6,4.7,4.8,4.9,5,5.1,5.2,5.3,5.4,5.5,5.6,5.7,5.8,5.9,6}
 {
  \draw[color=purple,thick,opacity=0.2]({\a*0.41+\A*cos(deg(1.3*(\o*\t+\a)))+\delu},{-(\t-1-\pt*cos(deg(\a*\pu+\fa)))+2})--({(\a+0.1)*0.41+\A*cos(deg(1.3*(\o*\t+\a+0.1)))+\delu},{-(\t-1-\pt*cos(deg((\a+0.1)*\pu+\fa)))+2});
  }

  \draw[dashed,color=gray]({0.41+\A*cos(deg(1.3*(\o*\t+1)))+\delu},-0.4)--({0.41+\A*cos(deg(1.3*(\o*\t+1)))+\delu},{-(\t-1-\pt*cos(deg(\pu+\fa)))+2});
  \draw[dashed,color=gray]({1.46*0.41+\A*cos(deg(1.3*(\o*\t+1.46)))+\delu+0.05},-0.4)--({1.46*0.41+\A*cos(deg(1.3*(\o*\t+1.46)))+\delu+0.05},2.35);
  \draw[color=gray,dashed]({5.8*0.41+\A*cos(deg(1.3*(\o*\t+5.8)))+\delu},-0.4)--({5.8*0.41+\A*cos(deg(1.3*(\o*\t+5.8)))+\delu},3.25);
  \draw[color=gray,dashed]({6.1*0.41+\A*cos(deg(1.3*(\o*\t+6.1)))+\delu},-0.4)--({6.1*0.41+\A*cos(deg(1.3*(\o*\t+6.1)))+\delu},{-(\t-1-\pt*cos(deg(6.1*\pu+\fa)))+2})node[color=ForestGreen,right]{$f_{t^*_2}$};

   \foreach \a in {1,1.1,1.2,1.3,1.4}
 {
  \draw[color=ForestGreen,thick]({\a*0.41+\A*cos(deg(1.3*(\o*\t+\a)))+\delu},{-(\t-1-\pt*cos(deg(\a*\pu+\fa)))+2})--({(\a+0.1)*0.41+\A*cos(deg(1.3*(\o*\t+\a+0.1)))+\delu},{-(\t-1-\pt*cos(deg((\a+0.1)*\pu+\fa)))+2});
  }


%
%




 \foreach \a in {5.8,5.9,6}
 {
  \draw[color=ForestGreen]({\a*0.41+\A*cos(deg(1.3*(\o*\t+\a)))+\delu},{-(\t-1-\pt*cos(deg(\a*\pu+\fa)))+2})--({(\a+0.1)*0.41+\A*cos(deg(1.3*(\o*\t+\a+0.1)))+\delu},{-(\t-1-\pt*cos(deg((\a+0.1)*\pu+\fa)))+2});
  }




 \draw[->]({-0.5+\deld},-0.4)--({3.5+\deld},-0.4)node[below]{$y$};
\draw[->](\deld,-0.7)--(\deld,3.5)node[left]{};

\pgfmathsetmacro{\t}{4}

  \foreach \a in {1.05,1.1,1.2,1.3,1.4,1.5,1.6,1.7,1.8,1.9,2,2.1,2.2,2.3,2.4,2.5,2.6,2.7,2.8}
 {
  \draw[color=purple,thick,opacity=0.2]({\a*0.41+\A*cos(deg(1.3*(\o*\t+\a)))+\deld},{-(\t-1-\pt*cos(deg(\a*\pu+\fa)))+3.6})--({(\a+0.1)*0.41+\A*cos(deg(1.3*(\o*\t+\a+0.1)))+\deld},{-(\t-1-\pt*cos(deg((\a+0.1)*\pu+\fa)))+3.6});
  }

  \foreach \a in {2.9,3,3.1,3.2,3.3,3.4,3.5,3.6,3.7,3.8,3.9,4,4.1,4.2,4.3,4.4,4.5,4.6,4.7,4.8,4.9,5,5.1,5.2,5.3,5.4,5.5,5.6}
 {
  \draw[color=purple,thick,opacity=0.21]({\a*0.41+\A*cos(deg(1.3*(\o*\t+\a)))+\deld},{-(\t-1-\pt*cos(deg(\a*\pu+\fa)))+3.6})--({(\a+0.1)*0.41+\A*cos(deg(1.3*(\o*\t+\a+0.1)))+\deld},{-(\t-1-\pt*cos(deg((\a+0.1)*\pu+\fa)))+3.6});
  }

  \foreach \a in {2.8,2.9,3,3.1,3.2,3.3,3.4,3.5,3.6,3.7,3.8,3.9,4}
  {
     \draw[color=ForestGreen,thick]({(2.8-1.5*(\a-2.9))*0.41+\A*cos(deg(1.3*(\o*\t+(2.8-1.5*(\a-2.9)))))+\deld},{-(\t-1-\pt*cos(deg(\a*\pu+\fa)))+0.4-(\t-1-\pt*cos(deg((2.8-1.5*(\a-2.9))*\pu+\fa)))+7.2})--({(2.8-1.5*(\a-2.8))*0.41+\A*cos(deg(1.3*(\o*\t+(2.8-1.5*(\a-2.8)))))+\deld},{-(\t-1-\pt*cos(deg((\a+0.1)*\pu+\fa)))+0.4-(\t-1-\pt*cos(deg((2.8-1.5*(\a-2.8))*\pu+\fa)))+7.2});
  }

   \foreach \a in {4.2,4.3,4.4,4.5,4.6,4.7,4.8,4.9,5,5.1,5.2,5.3,5.4,5.5}
 {
  \draw[color=ForestGreen,thick]({\a*0.41+\A*cos(deg(1.3*(\o*\t+\a)))+\deld},{-(\t-1-\pt*cos(deg(\a*\pu+\fa)))+3.6})--({(\a+0.1)*0.41+\A*cos(deg(1.3*(\o*\t+\a+0.1)))+\deld},{-(\t-1-\pt*cos(deg((\a+0.1)*\pu+\fa)))+3.6});
  }

  \foreach \a in {5.7,5.8,5.9,6}
 {
  \draw[color=purple,thick,opacity=0.21]({\a*0.41+\A*cos(deg(1.3*(\o*\t+\a)))+\deld},{-(\t-1-\pt*cos(deg(\a*\pu+\fa)))+3.6})--({(\a+0.1)*0.41+\A*cos(deg(1.3*(\o*\t+\a+0.1)))+\deld},{-(\t-1-\pt*cos(deg((\a+0.1)*\pu+\fa)))+3.6});
  }

   \foreach \a in {5.9,6}
  {
     \draw[color=ForestGreen,thick]({(5.8-(\a-5.9))*0.41+\A*cos(deg(1.3*(\o*\t+(5.8-(\a-5.9)))))+\deld},{-(\t-1-\pt*cos(deg(\a*\pu+\fa)))+0.4-(\t-1-\pt*cos(deg((5.8-(\a-5.9))*\pu+\fa)))+7.2})--({(5.8-(\a-5.8))*0.41+\A*cos(deg(1.3*(\o*\t+(5.8-(\a-5.8)))))+\deld},{-(\t-1-\pt*cos(deg((\a+0.1)*\pu+\fa)))+0.4-(\t-1-\pt*cos(deg((5.8-(\a-5.8))*\pu+\fa)))+7.2});
  }

   \draw[dashed,color=gray]({2.9*0.41+\A*cos(deg(1.3*(\o*\t+2.9)))+\deld},-0.4)--({2.9*0.41+\A*cos(deg(1.3*(\o*\t+2.9)))+\deld},1.75);

    \draw[dashed,color=gray]({4.2*0.41+\A*cos(deg(1.3*(\o*\t+4.2)))+\deld},-0.4)--({4.2*0.41+\A*cos(deg(1.3*(\o*\t+4.2)))+\deld},1.05);

    \draw[dashed,color=gray]({5.6*0.41+\A*cos(deg(1.3*(\o*\t+5.6)))+\deld},-0.4)--({5.6*0.41+\A*cos(deg(1.3*(\o*\t+5.6)))+\deld},2);

    \draw[dashed,color=gray]({5.9*0.41+\A*cos(deg(1.3*(\o*\t+5.9)))+\deld},-0.4)--({5.9*0.41+\A*cos(deg(1.3*(\o*\t+5.9)))+\deld},2.1)node[right,color=ForestGreen]{$f_{t^*_3}$};

\node at (2,-1){$t^*_1$}; \node[color=ForestGreen,thick] at (2,2.5)
{$\sum_{i\in I_{t^*_1}(y)}\mu(\phi_i(y))$}; \node at (2+\delu,-1){$t^*_2$};
\node[color=ForestGreen,thick] at (2+\delu,2)
{$\sum_{i\in I_{t^*_2}(y)}\mu(\phi_i(y))$}; \node at
(2+\deld,-1){$t^*_3$}; \node[color=ForestGreen,thick] at (2+\deld,2.7)
{$\sum_{i\in I_{t^*_3}(y)}\mu(\phi_i(y))$};

\end{tikzpicture}
\end{center}
\caption{Functions $f_{t^*_1}$, $f_{t^*_2}$ and $f_{t^*_3}$. The
graphs are obtained with $f_t=\sum_{i\in I_{t}(y)}\mu(\phi_i(y))$ for
each $t\in\{t^*_1,t^*_2,t^*_3\}$. The function $f_t$ represent the additive components
in the last part of the formula \eqref{eq-schema} without the multiplicative factor
$\frac{d\eta_t}{du}.$
} \label{fig4}
\end{figure}
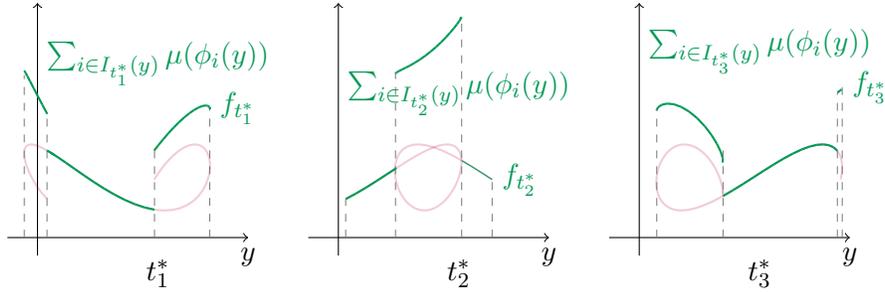

\section{The algorithm}\label{section-4}
In this section we present the application of the
Folding Domain Function algorithm
(FDF) to some cases of interest, in order to check its consistency by comparison with a brute force algorithm. As before, the dynamical system has a collection of solutions described by the function
$\varphi(\phi,t)$, which parametrically depends on the random variable $\phi$. In most of the cases of interest, such a solution can not
be determined analytically and, therefore, approximations to the solution are generated \emph{via} \textbf{Numerical integrators}
\footnote{In our examples we have used those already implemented in the
program Octave \cite{Octave}.}.

\subsection{The brute force algorithm}\label{sub-4.1}
The algorithm based on the Monte Carlo generation of instances of the RV $\phi$ will be referred to as the {\em brute-force} algorithm. We use this method for comparison with the FDF results and is described in the following list the steps:
\begin{enumerate}
\item We numerically generate a large number of points
$\phi\in[\alpha,\beta]$ distributed according to the known PDF
$\mu:[\alpha,\beta]\to {\mathbb R}_{\geq 0}$. Typically, computers provide uniform as well as Gaussian random number generators, which could be used at this point to generate any general distribution function $\mu$ for the random number $\phi$, provided that the latter is well behaved.
\item For each $\phi\in[\alpha,\beta]$ and a given time $t$, we calculate
the value $\varphi(\phi,t)$
\item Given the mapping $y = \varphi(\phi,t)$, we numerically estimate the PDF $\psi(y)$ from a given histogram constructed from the randomly generated trajectories.
\item Finally the histogram needs to be properly normalized.
\end{enumerate}

\subsection{The FDF-algorithm}\label{sub-4.2}

Here we will use $x$ rather than $\phi$ as a random variable, to
stress the fact that the theory also works for non-invertible
functions, which are not solutions of any RDE. We will restrict the
use of $\phi$ for the cases where the mapping is obtained from a
RDE. Again, for mappings arising from RDE, the function $g(x) =
\varphi(\phi,t)$, in agreement
with the notation of Section \ref{section-2}.

In the following, we list the steps of the FDF algorithm:
\begin{enumerate}
\item First of all we split the interval $[\alpha,\beta]$ into $N:=N_{\rm div}$
subintervals of equal length and set
$$x_i:= \alpha + i\frac{\beta-\alpha}{N}, \quad i=0,\dots, N,$$
so that $\alpha=x_0$ and $\beta= x_{N}$. Then we compute the image
values $g(x_i)$ either analytically or numerically.
\item Next, we
consider the indexes $i\geq 1$ such that
\begin{equation}\label{eq-signi}
(g(x_i)- g(x_{i-1}))(g(x_{i+1}) - g(x_i))\leq 0.
\end{equation}
Clearly, if $i$ is such that the above relation is satisfied, then
an extremum point of $g$ is contained in the open interval $]x_{i-1},x_{i+1}[$
and it will be approximated by $x_i.$ These, together with $\alpha$ and $\beta$
are the points $\alpha_j$ in Section \ref{sub-1.1}.
\\~\\
The algorithm saves
the $(k+1)$-tuple $(\alpha_j)_{j=0,\dots,k}$, the values $g(\alpha_j),$
together with the set of $j$-indices is the set of $i$'s such that
\eqref{eq-signi} holds. At this point, also the vector $\Lambda=(\lambda_1,\dots,\lambda_k)$
and the $(k+1)$-tuple $(m_j)_{j=0,\dots,k}$ are determined, together with the
range $[g_{\min},g_{\max}]$ of the function $g.$
\item
In order to produce the ordered set $B=\{b_0,\dots,b_{\ell}\}$ with
$g_{\min}=b_0<b_1< \dots < b_{\ell}=g_{\max},$ we sort the values
$g(\alpha_j)$ in increasing order and erase the possible repeated values (if any). At this step,
also the central points $c_i = \frac{b_{i-1} + b_i}{2}$ are determined
and then find the set of indices $I(c_i)=I(y),$ for all $y\in ]b_{i-1},b_i[.$
In fact, we have that $j\in I(c_i)$ if and only if
$$0 < (c_i -g(\alpha_{j-1})){\rm sign}(\lambda_j) < |\lambda_j|,$$
according to \eqref{eq-yglambda}. This procedure allows to save the
set $I(c_i).$ In a similar manner, we determine the sets $I(b_i).$
\item
The function $\eta$ is approximated with a piecewise linear interpolant
passing through the points
$(\hat{g}(x_i),x_i)$ where the values $\hat{g}(x_i)$ are
computed using \eqref{eq-mgj}.
\item
As a last step, we divide each interval $[b_{i-1},b_{i}]$ into a finite number of parts
with a fixed step-size $\Delta$ (in this manner the larger intervals are divided into
larger parts than the smaller ones). Then, knowing the sets
$I(b_{i-1},I(c_i),I(b_i)$ and using \eqref{form-1}, we can determine
the values $u_{\rho_1}(y),\dots, u_{\rho_p}(y)$
(with $\{\rho_1,\dots,\rho_p\}=I(y)$), for each $y$ in the $\Delta$-partition
of $[b_{i-1},b_{i}].$ From these points we find the images
$\eta(u_{\rho_1}(y)), \dots \eta(u_{\rho_p}(y))$ in $[\alpha,\beta]$
and finally we can easily compute
$$\mu_Y(y)= \sum_{i\in I(y)} \mu_X(\eta(u(y)))\frac{d\eta}{du}(u(y)).$$
Observe that, in the algorithm, the derivative of $\eta,$ can be easily obtained
either with the absolute value of $1/g'\eta(u(y))$ (when the analytic
expression of $g$ is known), or by a standard approximation of the incremental ratio
of $g$ computed by a subroutine in the differential equation solver.
\end{enumerate}

\section{Numerical results}\label{section-5}
In this section we provide some numerical experiments that show the feasibility and reliability of our theory.
For the initial random variable $X$, let us choose the non-uniform, non-gaussian density function
\begin{equation}\label{eq-5.1}
\mu_X(x)=\frac{\sin(\omega x) + 2}{\int_{\alpha}^{\beta} (\sin(\omega x) + 2)dx},
\quad \omega=5.
\end{equation}
which we will use the same for all the examples. Clearly, $\mu_X(x)\geq 0$ for all $x\in [\alpha,\beta]$ and $\int_{\alpha}^{\beta} \mu_X(x)dx=1.$

\subsection{Example $(a)$}\label{sub-5.1}
As a first example, we analyze the logistic map $L(x)=rx(1-x)$, described in the case $(a)$
in the Introduction, and study the third iterate $g(x):= L^{(3)}(x)$, which is analogous to fixing a final time in a physical map, although this example is a case in which the mapping is not created from a RDE. We thus brush the interval of initial values $x \in [\alpha, \beta]$ to generate the function $g(x)$. In agreement with
the case presented in Figure \ref{fig-1.1}, we take $r=3.9.$ The natural
domain is then the interval $[\alpha,\beta]=[0,1].$
Moreover, $g_{\min}=0$ and $g_{\max}=L(1/2)=r/4=0.975.$ For the simulation, we have chosen
$N=N_{\rm div}=400.$

Figure \ref{fig-5.1} shows the density $\mu_X$ and the function $g.$
\begin{figure}[htp]
\begin{center}
\includegraphics[scale=0.3]{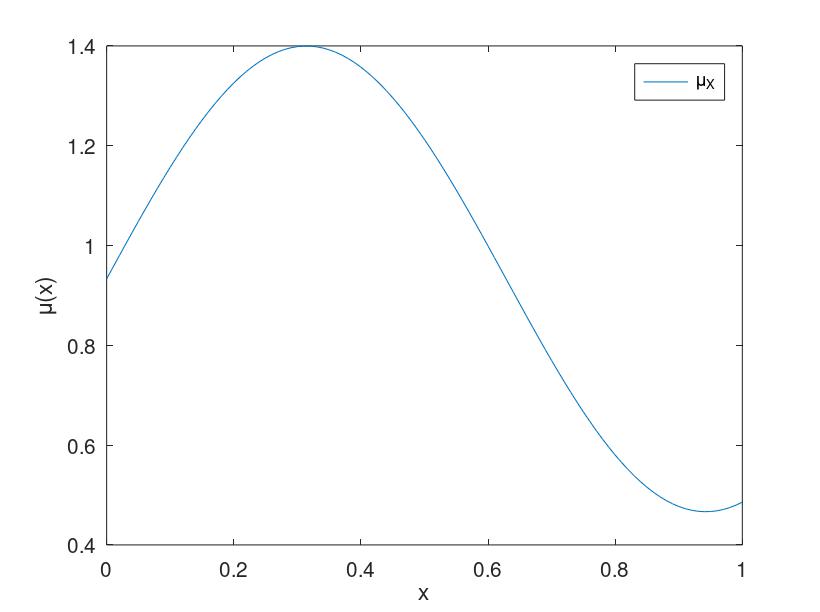}
\quad
\includegraphics[scale=0.3]{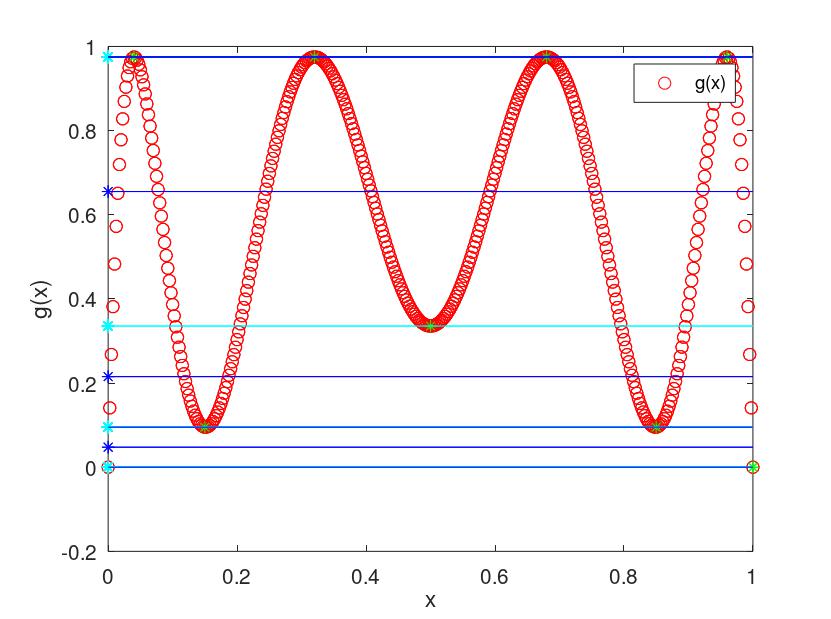}
\end{center}
\caption{\small{For the above set of parameters, we consider the density function $\mu_X$
(upper panel) and the function $g$ (lower panel). In the graph of function $g$
we have marked, on the vertical axis, the points $b_0,\dots,b_{\ell}$,
making the set $B$ (of four points), together
with the three middle points $c_i$. From the figure, the meaning of Lemma \ref{lem-A2}
is evident.}}
\label{fig-5.1}
\end{figure}
Figure \ref{fig-5.2} shows the function $\eta$ (which is the inverse of $\hat{g}$) and the
resulting density function $\mu_Y$, produced numerically via both the brute-force technique and
the FDF-A, for comparison. It is apparent from the figure that
our approach provides a smoother, more precise outcome than the brute-force one. Also very important is that
with the FDF-A the computational cost is remarkably reduced. This is particularly relevant when the number if iterates is very large. In the Monte Carlo approach, the generation of the bundle of trajectories may become a computationally very intensive task.

\begin{figure}[htp]
\begin{center}
\includegraphics[scale=0.3]{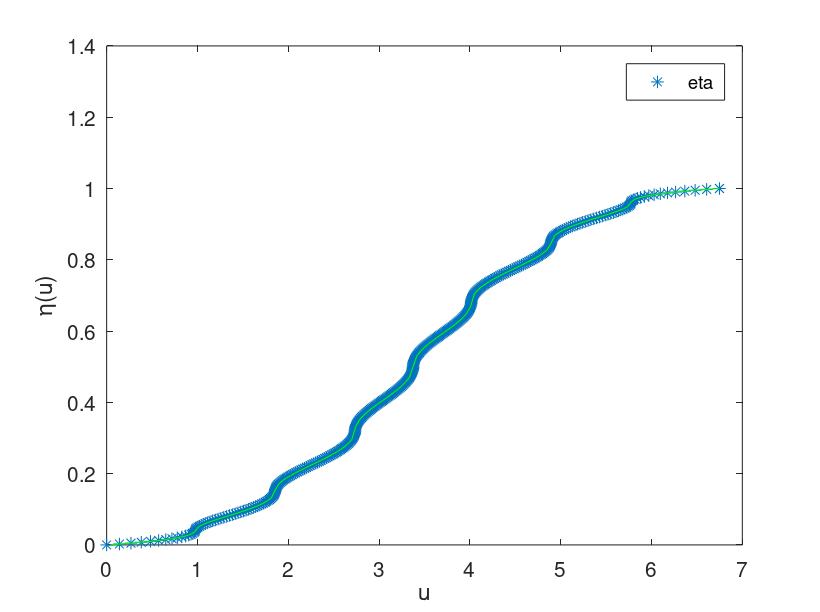}
\quad
\includegraphics[scale=0.3]{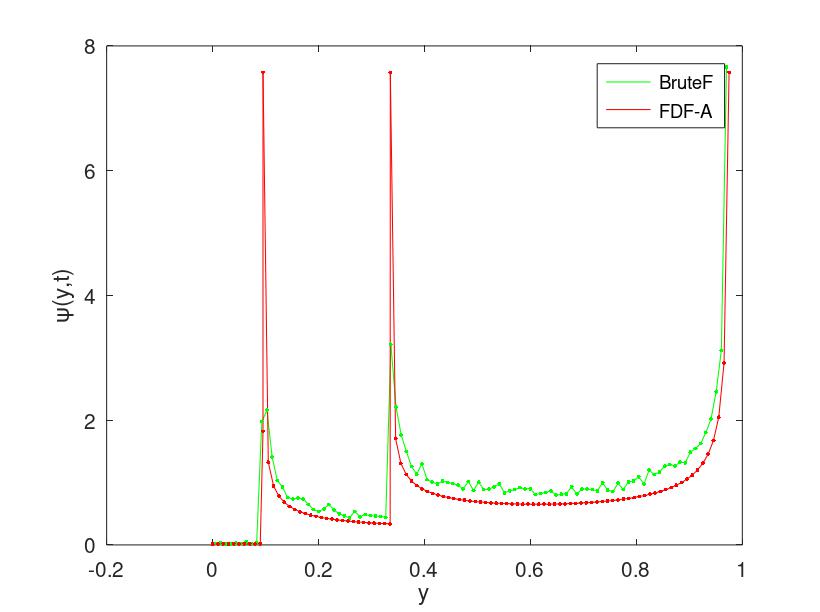}
\end{center}
\caption{\small{The function $\eta$
(upper panel) and the density function $\mu_Y$ (lower panel).
For the density function we compare two different graphs obtained,
respectively, by the brute-force algorithm (in green) and by the FDF-A (in red).}}
\label{fig-5.2}
\end{figure}

\subsection{Example $(b)$}\label{sub-5.2}
As a second example, we consider a projection of the Poincar\'{e} map associated with
the shooting method applied to a second-order scalar equation of Duffing type
$x''+ \rho(x)=0.$ With reference to the two cases introduced in Figure \ref{fig-1.2}
and Figure \ref{fig-1.3}, we study the ``classical'' Duffing equation with $\rho(x)=4x^3$
and the pendulum equation with $\rho(x)=\sin(x).$

For the former case, we take as reference time-interval $[0,T_{\max}]=[0,5]$
and solve the initial value problems
\begin{equation}\label{eq-5.2}
\begin{cases}
y'=v, \; v'= - 4y^3\\
y(0)=0, \; v(0) = \phi
\end{cases}
\end{equation}
on the interval $[0,T_{\max}],$ to obtain the function
$\phi\mapsto y(t,\phi)$ for $t=T_{\max}$ (cf. Figure
\ref{fig-1.2}, lower panel). To solve numerically \eqref{eq-5.2}
we use an already implemented ODE-solver of Octave \cite{Octave},
with a step-size $5/300.$ The chosen density function $\mu_X$ is
the same as in \eqref{eq-5.1} for the interval
$[\alpha,\beta]=[0,5].$ Moreover, $g_{\min}\approx -1.3015$ and
$g_{\max}\approx 1.6717.$ For the simulation, we have chosen
$N=N_{\rm div}=300.$

Figure \ref{fig-5.3} shows the density $\mu_X$ and the function $g(x)=y(t,x).$
\begin{figure}[htp]
\begin{center}
\includegraphics[scale=0.3]{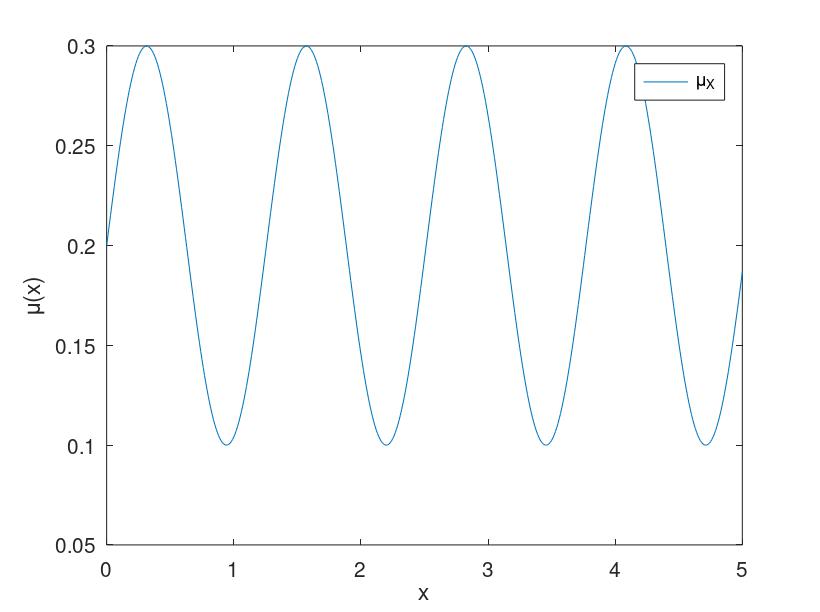}
\quad
\includegraphics[scale=0.3]{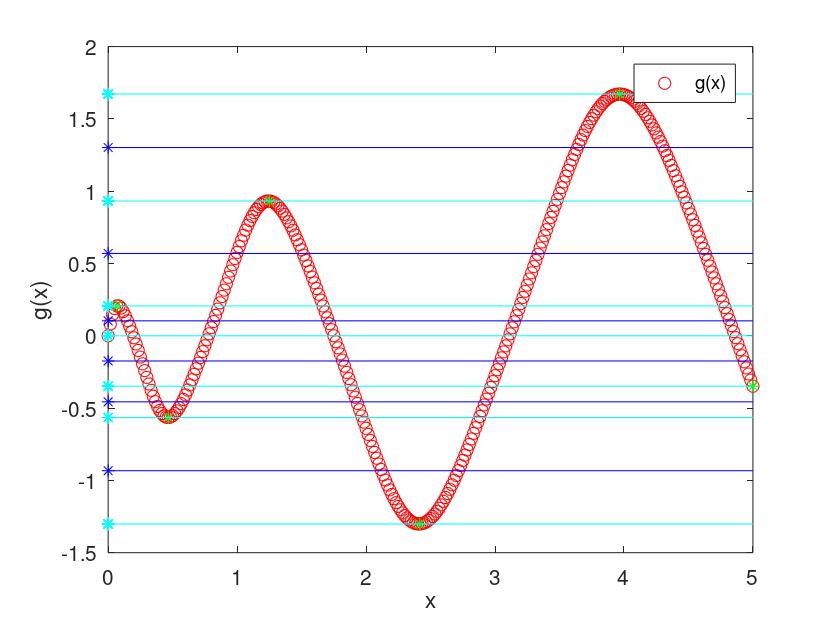}
\end{center}
\caption{\small{For the above set of parameters, we consider the density function $\mu_X$
(upper panel) and the function $g$ (lower panel). In the graph of function $g$
we have marked, on the vertical axis, the points $b_0,\dots,b_{\ell}$,
making the set $B$ (of seven points), together
with the six middle points $c_i$.}}
\label{fig-5.3}
\end{figure}
Figure \ref{fig-5.4} shows the function $\eta$ and the
resulting density function $\mu_Y$, numerically produced via the brute-force technique and also through
the FDF-A.

Again, the results obtained with the FDF algorithm are smoother, more accurate and are produced at lower computational cost than with the brute-force algorithm.

\begin{figure}[htp]
\begin{center}
\includegraphics[scale=0.3]{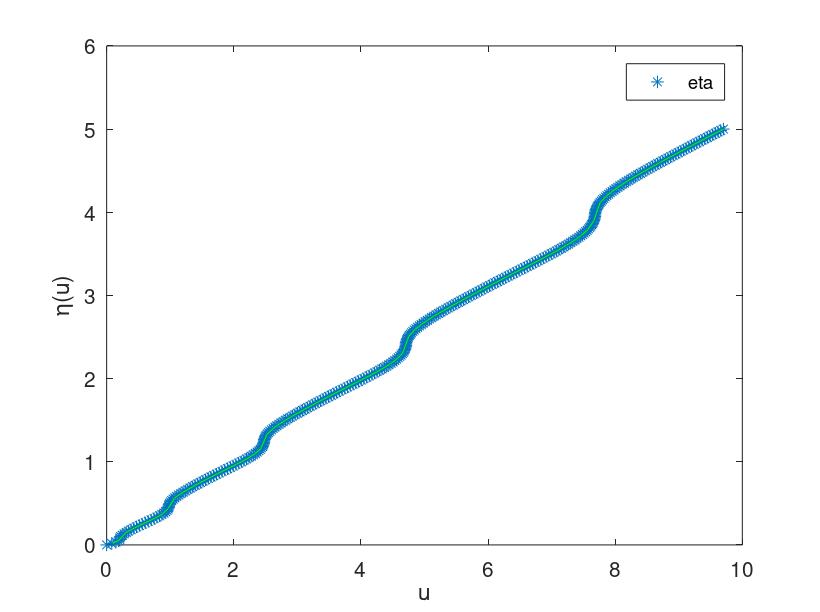}
\quad
\includegraphics[scale=0.3]{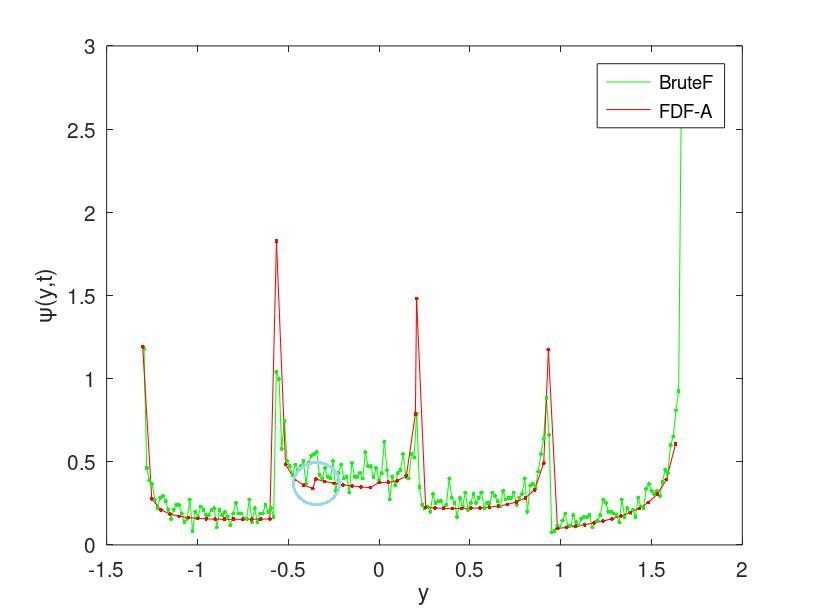}
\end{center}
\caption{\small{The function $\eta$
(upper panel) and the density function $\mu_Y$ (lower panel).
For the density function we compare two different graphs obtained,
respectively, by the brute-force algorithm (in green) and by the FDF-A (in red).}}
\label{fig-5.4}
\end{figure}

Considering now the case of the pendulum equation, we take as
reference time-interval $[0,T_{\max}]=[0,18]$ and solve the
initial value problems
\begin{equation}\label{eq-5.3}
\begin{cases}
y'=v, \; v'= - \sin(y)\\
y(0)=0, \; v(0) = \phi
\end{cases}
\end{equation}
on the interval $[0,T_{\max}],$ obtaining the function
$\phi\mapsto y(t,\phi)$ for $t=T_{\max}$ (cf. Figure
\ref{fig-1.3}, right panel). To solve numerically \eqref{eq-5.3}
we use, as in the above example, an already implemented ODE-solver
of Octave \cite{Octave}, with a step-size $18/200.$ The chosen density function
$\mu_X$ is the same as in \eqref{eq-5.1} for the interval
$[\alpha,\beta]=[0,1.99].$ Moreover, $g_{\min}\approx -2.5019$ and
$g_{\max}\approx 2.9224.$ For the simulation, we have chosen
$N=N_{\rm div}=200.$

Figure \ref{fig-5.5} shows the density $\mu_X$ and the function
$g(x)=y(t,x).$
\begin{figure}[htp]
\begin{center}
\includegraphics[scale=0.3]{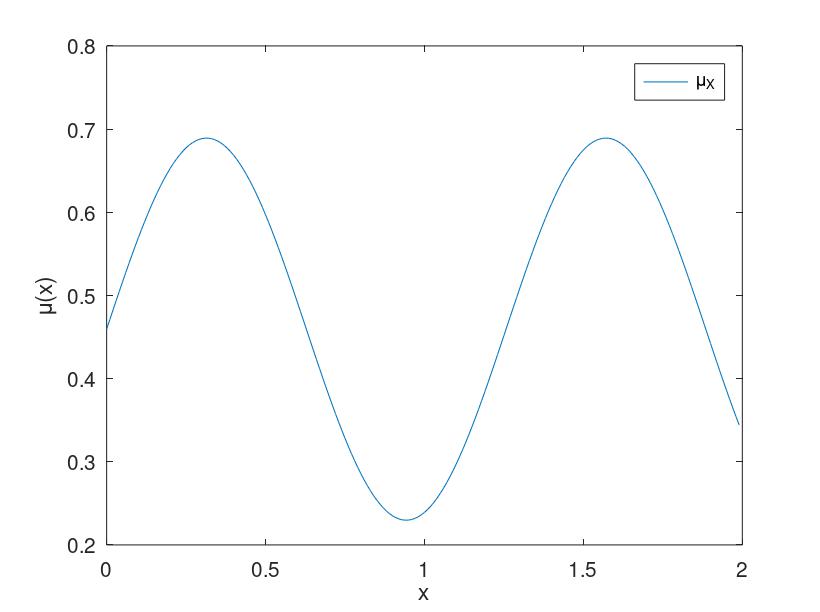}
\quad
\includegraphics[scale=0.3]{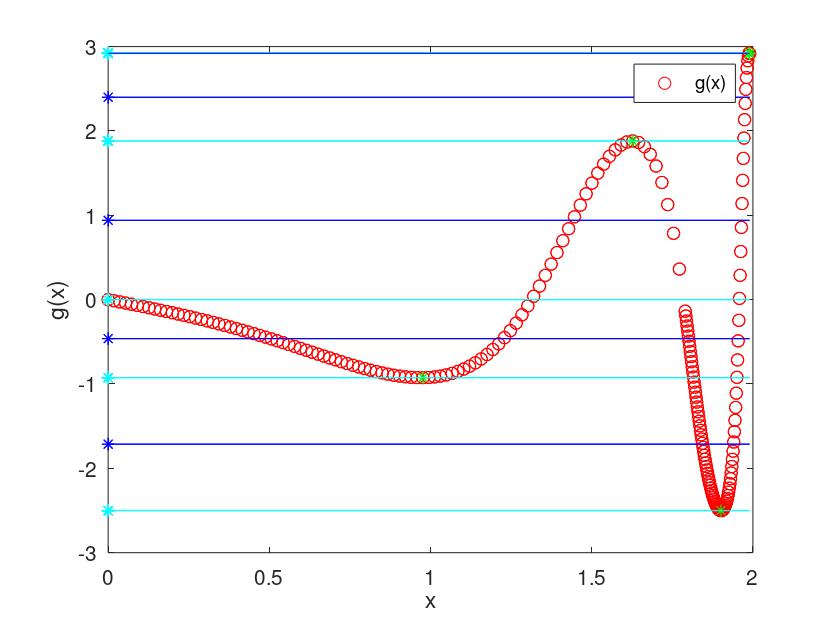}
\end{center}
\caption{\small{For the above set of parameters, we consider the
density function $\mu_X$ (upper panel) and the function $g$ (lower panel). In the graph of function $g$ we have marked, on the
vertical axis, the points $b_0,\dots,b_{\ell}$, making the set $B$
(of five points), together with the four middle points $c_i$.}}
\label{fig-5.5}
\end{figure}

Figure \ref{fig-5.6} shows the function $\eta$ and the resulting
density function $\mu_Y$, produced, numerically, via the
brute-force technique and the FDF-A.
\begin{figure}[htp]
\begin{center}
\includegraphics[scale=0.3]{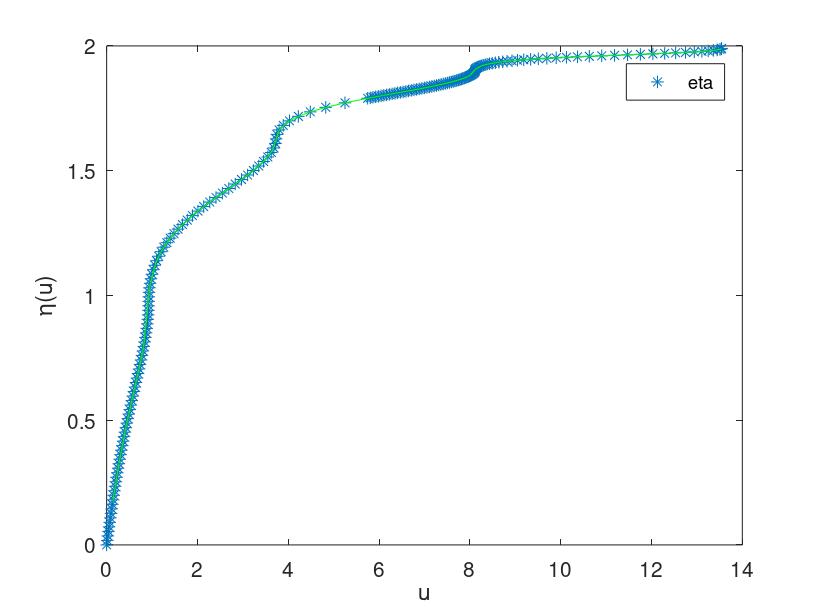}
\quad
\includegraphics[scale=0.3]{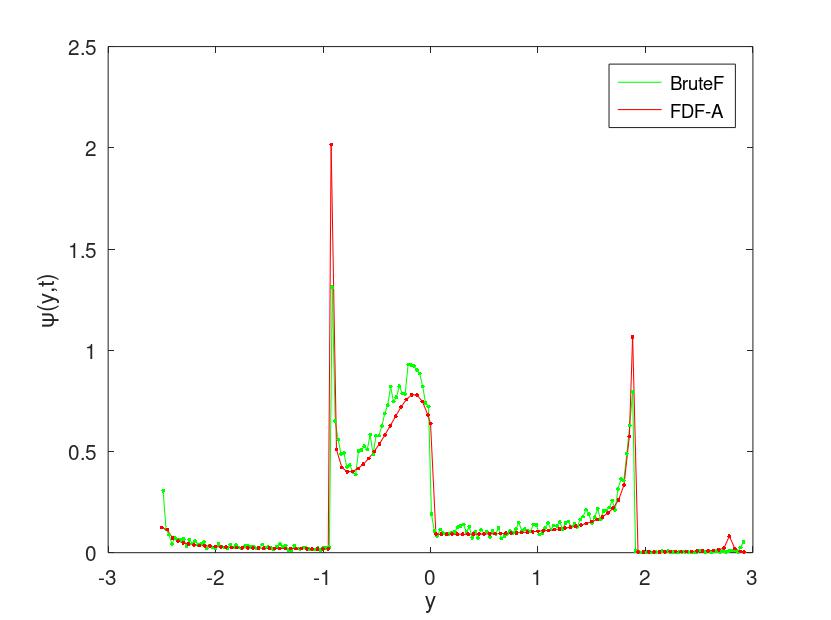}
\end{center}
\caption{\small{The function $\eta$ (upper panel) and the density
function $\mu_Y$ (lower panel). For the density function we
compare two different graphs obtained, respectively, by the
brute-force algorithm (in green) and by the FDF-A (in red).}}
\label{fig-5.6}
\end{figure}

\eject

\subsection{Example $(c)$}\label{sub-5.3}
As a third example, we consider the problem in \eqref{eq-1.4},
leading to an explicit form for $\varphi(\phi,t)$ as in
\eqref{eq-1.5}. In this case, we have an explicit known form for
the function
\begin{equation}\label{eq-5.4}
g(x)=kx+ A\cos(\omega t + x),
\end{equation}
where, $k,A,\omega,t$  are fixed parameters for our
numerical simulation. The chosen density function $\mu_X$ is the
same as in \eqref{eq-5.1} for the interval $[\alpha,\beta]=[2,4].$
For the simulation, we have chosen $N=N_{\rm div}=200.$

Figure \ref{fig-5.7} shows the density $\mu_X$ and the function
$g,$ with $g$ defined for the parameters $k=1,$ $A=2,$ $\omega=6$
and $t=1.$
\begin{figure}[htp]
\begin{center}
\includegraphics[scale=0.3]{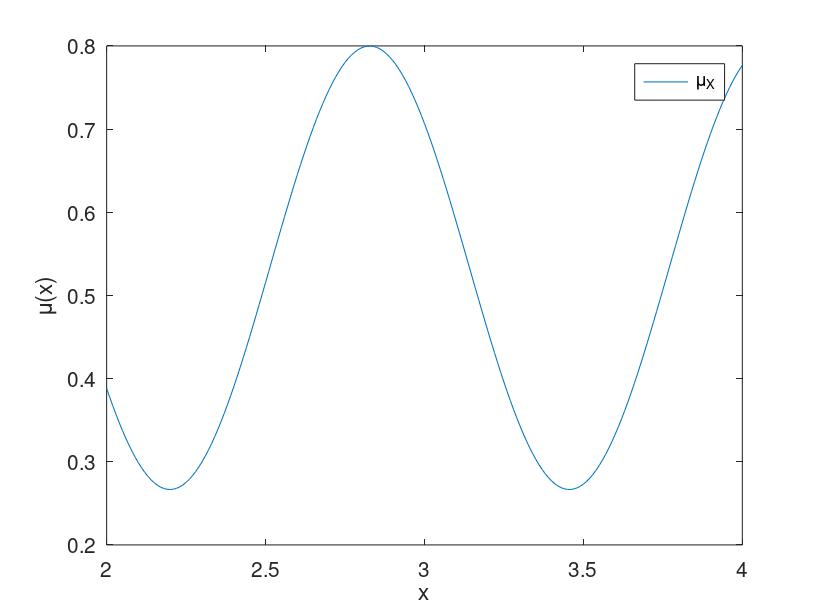}
\quad
\includegraphics[scale=0.3]{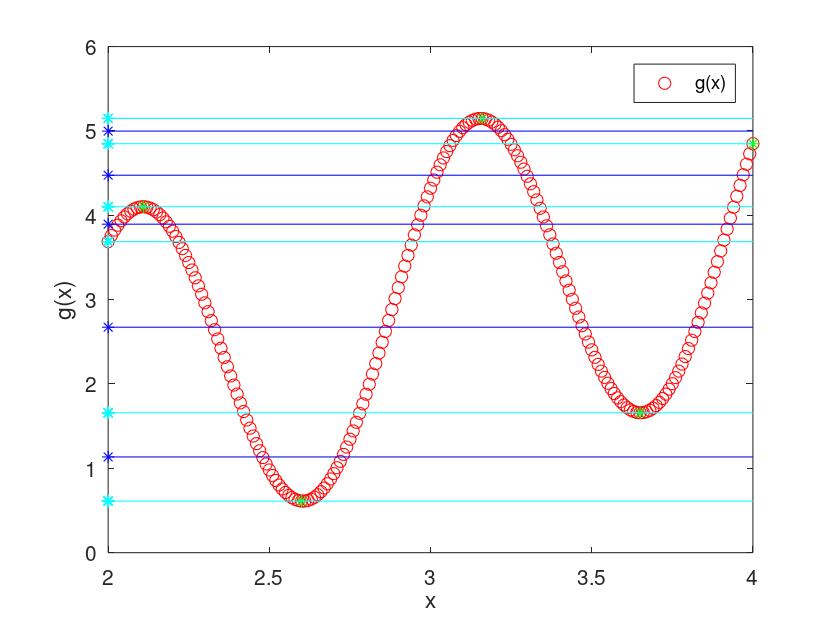}
\end{center}
\caption{\small{We consider the density function $\mu_X$ (upper panel) and the function $g$ (lower panel). In the graph of
function $g$ we have marked, on the vertical axis, the points
$b_0,\dots,b_{\ell}$, making the set $B$ (of six points),
together with the five middle points $c_i$. From the figure, the
meaning of Lemma \ref{lem-A2} is evident.}}
\label{fig-5.7}
\end{figure}

Figure \ref{fig-5.8} shows the function $\eta$ (which is the
inverse of $\hat{g}$) and the resulting density function $\mu_Y$,
via the brute-force procedure as well as the
FDF-A.
\begin{figure}[htp]
\begin{center}
\includegraphics[scale=0.3]{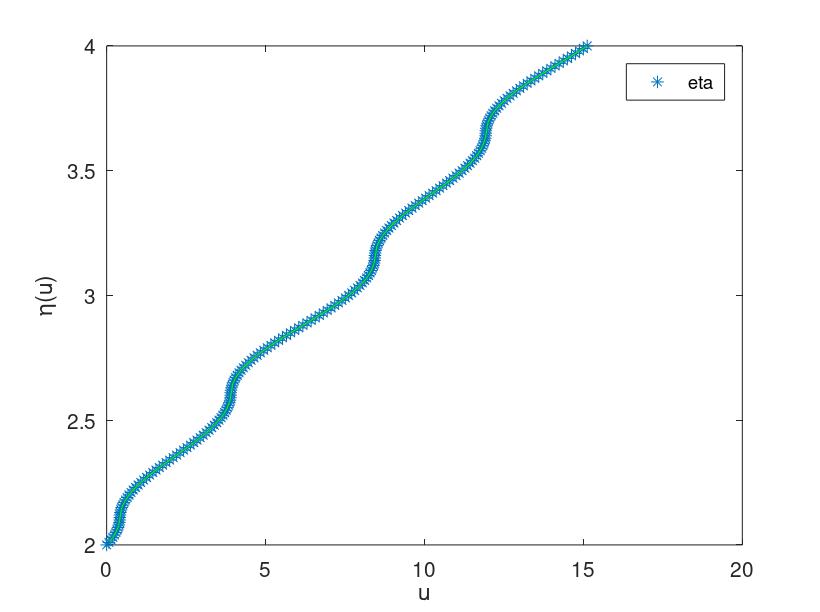}
\quad
\includegraphics[scale=0.3]{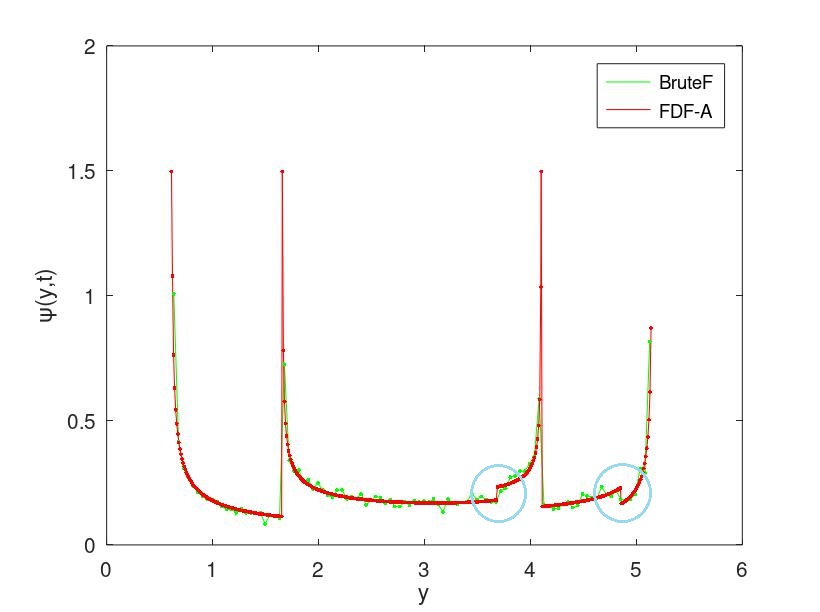}
\end{center}
\caption{\small{The function $\eta$ (upper panel) and the density
function $\mu_Y$ (lower panel). For the density function we
compare two different graphs obtained, respectively, by the
brute-force algorithm (in green) and by the FDF-A (in red). It is
apparent that our approach provides a smoother and more precise
outcome than the other one. Moreover, with the FDF-A the
computational cost is remarkably reduced.}}
\label{fig-5.8}
\end{figure}

Finally, we want to point out the small jumps present in Figure \ref{fig-5.4} and Figure \ref{fig-5.8} which we have highlighted with a small circle. These discontinuities are
due neither to numerical errors nor to any flaw in the FDF algorithm. Instead, they reflect the complexity of the overlapping of different branches in the folded domain, as we have shown in the sketch of Figure \ref{fig4}. In this particular case, they are due to the overlapping of the extremal points
$g(\alpha)$ and/or $g(\beta)$ with other values in the image set
$g(]\alpha,\beta[).$


\section{Conclusion and future work}\label{section-6}
In the present paper we have addressed a classical problem in the study of
random variables, namely, quoting \cite[p.80]{JP-2000}:
\textit{Let $X$ be a RV with density $f$. Suppose $Y=g(X)$.
Can we express the density of $Y$ (if it exists), in terms of $f$ ?}
If $X$ has a continuous density function and $g$ is piecewise strictly monotone and
continuously differentiable, the problem is solved, from a theoretical point of view,
by the formula
\begin{equation}\label{eq-6.1}
\mu_Y(y)=\sum_{i=1}^{n}\mu_X(h_i(y))|h'_i(y)| \mathbbm{1}_{\land}(y)
\end{equation}
(see \cite[Corollary 11.3]{JP-2000}).
In \eqref{eq-6.1},
$\land$ denotes the range of $g$ and it is assumed that there exists intervals
$I_1,I_2,\dots,I_n$ which partition the domain of $g$ and such that
$g$ is strictly monotone and continuously differentiable on the interior of each $I_i.$
Then, $h_i: g(I_i)\to I_i$ denotes the inverse of $g$ restricted to $I_i\,.$
In applying the above formula, one has to decide
\textit{for every $y$ in the range of $g$} the correct functions $h_i$ to be chosen
in the summation. In the simplest case, $g$ is continuously differentiable
with non-vanishing derivative, and the above formula reduces to
\begin{equation}\label{eq-6.2}
\mu_Y(y)=\mu_X(h(y))|h'(y)|, \quad \text{for }\; h=g^{-1}
\end{equation}
(see \cite[Corollary 11.2]{JP-2000}).
This latter result may be easily extended to the vector case; in this situation,
one has to assume that $g$ is a diffeomorphism.

Applications of \eqref{eq-6.2} for one-dimensional and
higher-dimensional models have received a great deal of interest in
the past decades, as witnessed by a growing list of published
research articles in this area in recent years (see, for instance, Refs.
\cite{BCCN-2018,CCDJ-2022,CCRB-2014,CC-2016,CC-2016b,
CC-2016c,CENSS-2020}). Actually, the implementations of
\eqref{eq-6.2} in applied mathematical models is usually referred to
as \textit{Probabilistic Transformation Methods} (according to
Kadry in \cite{Kadry-2007}) or \textit{Random Variable
Transformation (RVT) Technique} (according to Casab\'{a}n et al.
in Ref. \cite{CCRB-2014} and also to the more recent work by Cort\'{e}s et
al. \cite{CLRR-2023}). Further applications of this method have
been developed for various mathematical model equations of great
interest for ecological and biological sciences in subsequent
articles, as those cited above. In some of these articles (as, for
instance, in Ref. \cite{CC-2016}) the domain of the inverse
transformation is split into two disjoint open intervals, so that
\eqref{eq-6.1} simplifies to equation \eqref{eq-6.2} applied
twice. In the very recent work \cite{Almalki-2023}, Almalki,
Elshekhipy and Almutairi deal with a very general case of
non-monotone transformation $g$ and study the probability
distribution of $Y=g(X)$ from a theoretical and numerical point of
view.

Our work, independently follows a similar line of reasoning,
namely, to obtain information of the RV $Y=g(x),$ under the most
general assumptions for $g.$ With this respect, our main
contributions in the present work are the following:
\begin{itemize}
\item{} As already observed, in applying the formula \eqref{eq-6.1}, one has to decide
\textit{for every $y$ in the range of $g$} the correct functions $h_i$ to choose
in the summation. With our approach (and using
Lemma \ref{lem-A1} and Lemma \ref{lem-A2}) rather we can easily determine the index set
$I(y)$ associated with $y;$
\item{} We have proposed a new method to obtain the piecewise-partial inverses
$h_i$ of $g,$ by introducing a unique global function $\eta.$
The function $\eta,$ in some sense, ``unfolds'' the graph of $g$ in a manner such that
the derivatives (in absolute value) $|h'_i(y)|$ in \eqref{eq-6.1}
correspond to the inverse derivative of $\eta$ computed in different intervals.
\item{} Our approach is mathematically rigorous, because we can independently
re-establish \eqref{eq-6.1}, passing through the use of the function $\eta$
(see Section \ref{sub-1.1});
\item{} Our method, when numerically implemented,
requires a minimal
set of information and therefore is extremely simple and fast compared to standard histogram methods based on a Monte Carlo sampling of the initial probability distribution.
\end{itemize}

The very recent articles
\cite{Almalki-2023,CLRR-2023,LLF-2022,NLF-2023} show the persistent interest and relevance
of the problem considered in this work,
both from the theoretical and the applied point of view. We hope that our contribution
can provide a new tool in this area of research.

As a final comment, we outline some possible extensions of our approach
to the higher-dimensional case.
Let $X$ be a $m$-dimensional random variable with domain ${\mathscr D}_X$ and
let $g: {\mathbb R}^m\to {\mathbb R}^m$ be continuously differentiable mapping,
leading to the new Random Variable $Y=g(X),$
with values in ${\mathbb R}^m.$ As in the one-dimensional case,
we are interested in providing an efficient method to estimate the probability
density function $\mu_Y$ in terms of $\mu_X.$
From a theoretical point of view, the problem is solved by
\cite[Corollary 12.1]{JP-2000}; however, implementing the method in
concrete examples may be a formidable task.
This is therefore a problem of great interest, as witnessed by the
recent articles \cite{NLF-2023,LLF-2022}. As Lemma
\ref{lem-A1} and Lemma \ref{lem-A2} can be proven to be also valid in more than one dimension, our algorithm can be extended to maps connecting higher dimensional spaces (see also Remark \ref{rem-A1}).
On the other hand, while in the one-dimensional case, the singular set
reduces to a finite (discrete) set of points, in dimension $m\geq 2,$
we have to face some complex differential geometry problems
concerning the boundaries of the sets where $g$ is partially invertible.
This problem will be addressed next.

\section*{CRediT authorship contribution statement}
\textbf{Fabrizio Masullo:} Conceptualization,Formal analysis, Methodology, Software, Visualization,
Organization of the article.
\textbf{Fabio Zanolin:} Formal analysis, Organization of the article.
\textbf{Josep Bonet Avalos:} Supervision,  Research Project Leader.

\section*{Conflict of interest}
The authors claim that there is no conflict of interest associated with this article
and with the research leading to this article.

\bibliographystyle{plain}
\bibliography{references}

\end{document}